\documentclass[12pt]{amsart}
\usepackage[utf8]{inputenc}
\usepackage{amssymb,hyperref,mathrsfs, mathtools,cancel}
\usepackage{hyperref}
\usepackage[mathscr]{euscript}
\usepackage[lmargin=1.5in,rmargin=1.5in,tmargin=1in,bmargin=1in]{geometry}
\usepackage{enumitem}
\usepackage{color}
\usepackage{tikz,graphicx}
\usetikzlibrary{matrix}
\usetikzlibrary{arrows}
\usetikzlibrary{positioning}
\usepackage{overpic}
\usepackage{todonotes}
\usepackage{subcaption}
\usepackage{overpic}
\usepackage[all]{xy}

\input{diagrams}

\newcommand{\ol}{\overline}

\newcommand{\ZZ}{\mathbb{Z}}
\newcommand{\RR}{\mathbb{R}}
\newcommand{\QQ}{\mathbb{Q}}

\newcommand{\bdy}{\partial}
\DeclareMathOperator{\gr}{gr}

\DeclareMathOperator{\Int}{Int}

\newcommand{\x}{\mathbf{x}}

\theoremstyle{plain}

\numberwithin{equation}{section}
\newtheorem{theorem}[equation]{Theorem}

\newtheorem{lemma}[equation]{Lemma}
\newtheorem{corollary}[equation]{Corollary}

\newtheorem{conjecture}[equation]{Conjecture}

\newtheorem{definition}[equation]{Definition}

\theoremstyle{definition}

\theoremstyle{remark}
\newtheorem{example}[equation]{Example}
\newtheorem{remark}[equation]{Remark}

\definecolor{darkgreen}{rgb}{0,.25,0}
\definecolor{darkred}{rgb}{.25,0,0}
\definecolor{pink}{rgb}{1,.08,.575}

\makeatletter
\providecommand\@dotsep{5}
\def\listtodoname{List of Todos}
\def\listoftodos{\@starttoc{tdo}\listtodoname}
\makeatother

\newcommand{\RN}[1]{%
  \textup{\uppercase\expandafter{\romannumeral#1}}%
}

\newcommand{\rom}[1]{\uppercase\expandafter{\romannumeral #1\relax}}

\newcommand*{\mycap}{\mathrel{\text{{\rotatebox[origin=c]{180}{$\mathsf{U}$}}}}}
\newcommand*{\mycup}{\mathsf{U}}

\newcommand{\A}{\mathcal{A}}

\newcommand{\pc}{\mathsf{p}}
\newcommand{\sing}{\mathsf{x}}
\newcommand{\sm}{\mathsf{sm}}

\newcommand{\cupplus}{{\setbox0\hbox{\large$\cup$}\rlap{\hbox to \wd0{\hss\raisebox{3pt}{\tiny$+$}\hss}}\box0}}
\newcommand{\cupminus}{{\setbox0\hbox{\large$\cup$}\rlap{\hbox to \wd0{\hss\raisebox{3pt}{\tiny$-$}\hss}}\box0}}

\graphicspath{{figs/}}

\begin{document}

\title[A link invariant related to Khovanov homology and HFK]{A link invariant related to Khovanov homology and knot Floer homology}

\author{Akram Alishahi}
\thanks{AA was supported by NSF grants DMS-1505798 and DMS-1811210.}
\address{Department of Mathematics, Columbia University, New York, NY 10027}
\email{\href{mailto:alishahi@math.columbia.edu}{alishahi@math.columbia.edu}}

\author{Nathan Dowlin}
\thanks{ND was supported by NSF grant DMS-1606421.}
\address{Department of Mathematics, Columbia University, New York, NY 10027}
\email{\href{mailto:ndowlin@math.columbia.edu }{ndowlin@math.columbia.edu}}


\date{\today}

\begin{abstract}

In this paper we introduce a chain complex $C_{1 \pm 1}(D)$ where $D$ is a plat braid diagram for a knot $K$. This complex is inspired by knot Floer homology, but it the construction is purely algebraic. It is constructed as an oriented cube of resolutions with differential $d=d_{0}+d_{1}$. We show that the $E_{2}$ page of the associated spectral sequence is isomorphic to the Khovanov homology of $K$, and that the total homology is a link invariant which we conjecture is isomorphic to $\delta$-graded knot Floer homology. 

The complex can be refined to a tangle invariant for braids on $2n$ strands, where the associated invariant is a bimodule over an algebra $\A_{n}$. We show that $\A_{n}$ is isomorphic to $\overline{\mathcal{B}}'(2n+1, n)$, the algebra used for the $DA$-bimodule constructed by Ozsv\'{a}th and Szab\'{o} in their algebraic construction of knot Floer homology \cite{OS:Kauffman}.

\end{abstract}

\maketitle

\tableofcontents

%
%
%

\section{Introduction} The goal of this paper is to introduce a new homology theory for links in $S^{3}$ which is related to both the knot Floer homology of Ozsv\'{a}th-Szab\'{o} and Rasmussen (\cite{OS04:Knots}, \cite{Rasmussen03:Knots}) and Khovanov homology (\cite{Khovanov00:CatJones}). Knot Floer homology and Khovanov homology are defined using very different methods - the former is a Lagrangian Floer homology whose differential is often quite difficult to compute, while the latter is defined algebraically and has its roots in representation theory. Despite these differences, the two theories seem to contain a great deal of the same information and are conjectured to be related by a spectral sequence:

\begin{conjecture}[\cite{Rasmussen:KnotPolynomials}]\label{Conj1}

For any knot $K$ in $S^{3}$, there is a spectral sequence from the Khovanov homology of $K$ to the knot Floer homology of $K$.

\end{conjecture}

\noindent
Our construction gives a candidate for this spectral sequence.

\begin{remark}
Although the construction is inspired by holomorphic disc counts, the actual complex is an algebraic construction, so the reader need not be familiar with the holomorphic geometry typically present in a paper involving knot Floer homology.
\end{remark}

Given a diagram $D$ for a knot $K$ which is the plat closure of a braid, we define a complex $C_{1 \pm 1}(D)$ as an oriented cube of resolutions for $D$. The differential includes only vertex maps and edge maps. We write $d=d_{0}+d_{1}$ where $d_{i}$ increases cube filtration by $i$.

\begin{definition}

Given a knot $K$ in $S^{3}$, we write the $E_{2}$ page of the spectral sequence induced by the cube filtration as $H_{1+1}(K)$ and the total homology as $H_{1-1}(K)$. 

\[H_{1+1}(K)=H_{*}(H_{*}(C_{1 \pm 1}(D), d_{0}), d_{1}^{*})\]
\[H_{1-1}(K)=H_{*}(C_{1 \pm 1}(D), d_{0}+d_{1})\]

\end{definition}

\noindent
The total homology $H_{1-1}(K)$ is singly graded, while $H_{1+1}(K)$ has a second grading coming from the cube filtration. 

\begin{theorem}

The homology $H_{1+1}(K)$ is isomorphic to Khovanov homology.

\end{theorem}

\begin{theorem}

The total homology $H_{1-1}(K)$ is a link invariant.

\end{theorem}

The construction of $C_{1 \pm 1}(K)$ arose from an attempt to construct an oriented cube of resolutions for $\mathit{HFK}_{2}(K)$, a (singly graded) homology theory defined  by counting holomorphic discs which pass through basepoints which are typically blocked in knot Floer homology (\cite{Dowlin1}). The relevant results on $\mathit{HFK}_{2}$ from this paper are listed below:

I) The reduced homology $\widehat{\mathit{HFK}}_{2}(K)$ is isomorphic to $\delta$-graded knot Floer homology.

II) For a completely singular link $S$ (a singular diagram with no crossings) $\mathit{HFK}_{2}(S)$ is isomorphic to the Khovanov homology of $\sm(S)$, where $\sm(S)$ is the diagram obtained by replacing each singularization with the unoriented smoothing.

If there was an oriented cube of resolutions for $\mathit{HFK}_{2}$ and the edge maps corresponded to the Khovanov edge maps, this would give the desired spectral sequence. Unfortunately, the standard construction of the oriented cube of resolutions for knot Floer homology doesn't seem to work with the additional differentials.

The complex $C_{1 \pm 1}(K)$ aims to give an algebraic construction of this exact triangle. Let $D$ be a diagram for $K$ which is the plat closure of a braid. For each complete resolution $S$ of $D$, we replace $\mathit{CFK}_{2}(S)$ with an algebraically defined complex $(C_{1 \pm 1}(S), d_{0})$ which is conjectured to be chain homotopy equivalent. (It is certainly not isomorphic, as it is obtained by making some cancellations on $\mathit{CFK}_{2}(S)$ then guessing at the resulting differentials and module structure.) Then, for each $S_{1}$ and $S_{2}$ which differ at an edge, we define an edge map 
\[d_{1}: C_{ 1 \pm 1 }(S_{1}) \to C_{ 1 \pm 1 }(S_{2})   \]

\noindent
The resulting complex is $C_{ 1 \pm 1}(D)$. 

\begin{conjecture}\label{Conj2}
The total homology $H_{1-1}(K)$ is isomorphic to $\mathit{HFK}_{2}(K)$.
\end{conjecture}

\noindent
Note that this would prove Conjecture \ref{Conj1}.

The ground ring for $C_{ 1 \pm 1}(D)$ is a polynomial ring
\[ R = \mathbb{Q}[U_{1},...,U_{m}] \]

\noindent
where each $U_{i}$ corresponds to an edge $e_{i}$ in the diagram $D$. For each complete resolution $S$ of $D$, the complex $C_{ 1 \pm 1}(S)$ can be viewed as an $R$-module $\mathscr{M}(S)$ tensored with a Koszul complex $\mathsf{K}(D)$. 
\[C_{ 1 \pm 1}(S) = \mathscr{M}(S) \otimes \mathsf{K}(D) \]

The complex $\mathsf{K}(D)$ is the `closing off' factor - it depends only on the cups and caps in the plat closure, so it is the same for any resolution $S$. The edge map is the identity on $\mathsf{K}(D)$, we can write 
\[ d_{0} = 1 \otimes d_{\mathsf{K}} \hspace{20mm} d_{1} = d_{\mathscr{M}} \otimes 1 \]

Let $\mathscr{M}(D)$ denote the cube complex which has the module $\mathscr{M}(S)$ at the vertex in the cube corresponding to $S$, with differential $d_{\mathscr{M}}$. Then $C_{1 \pm 1}(D) = \mathscr{M}(D) \otimes \mathsf{K}(D) $. We are able to construct a local theory for $\mathscr{M}(D)$ which allows us to slice our diagram $D$ into tangles as in Figure \ref{FirstExample}. 

\begin{figure}[ht]
\centering
\def\svgwidth{6.5cm}
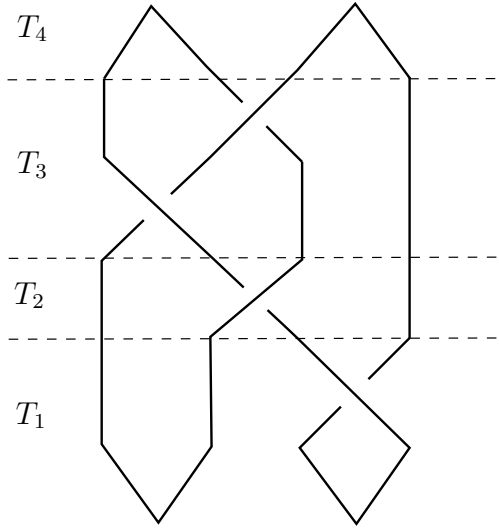
\caption{Breaking a plat diagram for the trefoil into four plat tangles}\label{FirstExample}
\end{figure}

The local theory assigns to each $(2m, 2n)$-tangle $T$ an $(\A_{m}, \A_{n})$-bimodule $\mathsf{M}[T]$. ($\A_{0}$ is defined to be $\mathbb{Q}$.) We orient our tangles from bottom to top, so a $(2m, 2n)$-tangle has $2m$ strands on the bottom and $2n$ strands at the top. Note that since our diagram is the plat closure of a braid, we either have $m=n$, $m=0$, or $n=0$. We will refer to tangles which are horizontal slices of a plat closure as \emph{plat tangles}.

\begin{theorem}

The local theory satisfies the following properties:

a) Let $T_{1}$ be a $(2l, 2m)$ plat tangle and $T_{2}$ a $(2m, 2n)$ plat tangle. Then 
\[ \mathsf{M}[T_{1} \circ T_{2}] \cong \mathsf{M}[T_{1}] \otimes_{\A_{m}} \mathsf{M}[T_{2}]  \]

b) The chain homotopy type of $[T]$ is an invariant for open braids, i.e. it is invariant under braid-like Reidemeister II and III moves.

\end{theorem}

\noindent
For example, if $D$ is the diagram for the trefoil in Figure \ref{FirstExample}, then 
\[\mathscr{M}(D)=\mathsf{M}[T_{1}] \otimes_{\A_{2}} \mathsf{M}[T_{2}] \otimes_{\A_{2}} \mathsf{M}[T_{3}] \otimes_{\A_{2}} \mathsf{M}[T_{4}] \]

\begin{remark}
Note that $\mathsf{M}[T]$ is not a tangle invariant when the cups and caps are involved - the Koszul complex $\mathsf{K}(D)$ is required.
\end{remark}

This construction is quite similar in spirit to the algebraic description of knot Floer homology by Oszv\'{a}th and Szab\'{o} in \cite{OS:Kauffman}. There are a few main differences: our tangle invariant is a bimodule while theirs is a $DA$-bimodule, and our complex comes from a planar Heegaard diagram while theirs comes from the large genus Kauffman states diagram. Especially in light of the second difference, the following theorem is quite surprising.

\begin{theorem}
The algebra $\A_{n}$ is isomorphic to the algebra $\overline{\mathcal{B}}'(2n+1, n)$ from \cite{OS:Kauffman}.
\end{theorem}

\noindent
This theorem gives some evidence for Conjecture \ref{Conj2}, and gives a construction of Khovanov homology using the same algebras present in Oszv\'{a}th and Szab\'{o}'s algebraic description of knot Floer homology.

\subsection*{Acknowledgments} We would like to thank Andy Manion, Ciprian Manolescu, Peter Ozsv\'{a}th, Ina Petkova, Robert Lipshitz, and Zoltan Szab\'{o} for helpful discussions.

\section{Background: Khovanov homology}

In this section we provide a definition of Khovanov's categorification of the Jones polynomial, known as Khovanov homology (\cite{Khovanov00:CatJones}). In order to make the relationship with our construction more clear, we will highlight the module structure of the Khovanov chain complex over the edge ring.

\subsection{Definition of the Khovanov Complex} Let $L$ be a link in $S^{3}$ with diagram $D \subset \mathbb{R}^{2}$. Let $\mathfrak{C}=\{c_{1}, c_{2}, ..., c_{n}\}$ denote the crossings in $D$, and viewing  $D$ as a 4-valent graph, let $E=\{e_{1}, e_{2}, ..., e_{m}\}$ denote the edges of $D$. The \emph{edge ring} is defined to be 

\[  R := \mathbb{Q}[X_{1}, X_{2},...,X_{m}]/\{X_{1}^{2}=X_{2}^{2}=...=X_{m}^{2}=0 \}   \]

\noindent
with each variable $X_{i}$ corresponding to the edge $e_{i}$. 

While our construction is an oriented cube of resolutions complex, the Khovanov complex is an \emph{unoriented} cube of resolutions. Each crossing $c_{i}$ can be resolved in two ways, the 0-resolution and the 1-resolution (see Figure \ref{resolutions}). For each $v \in \{0,1\}^{n}$, let $D_{v}$ denote the diagram obtained by replacing the crossing $c_{i}$ with the $v_{i}$-resolution. The diagram $D_{v}$ is a disjoint union of circles - denote the number of circles by $k_{v}$. The vector $v$ determines an equivalence relation on $E$, where $e_{p} \sim_{v} e_{q}$ if $e_{p}$ and $e_{q}$ lie on the same component of $D_{v}$.

\begin{figure}[ht]
\vspace{4mm}
\centering
\begin{overpic}[width = .8\textwidth]{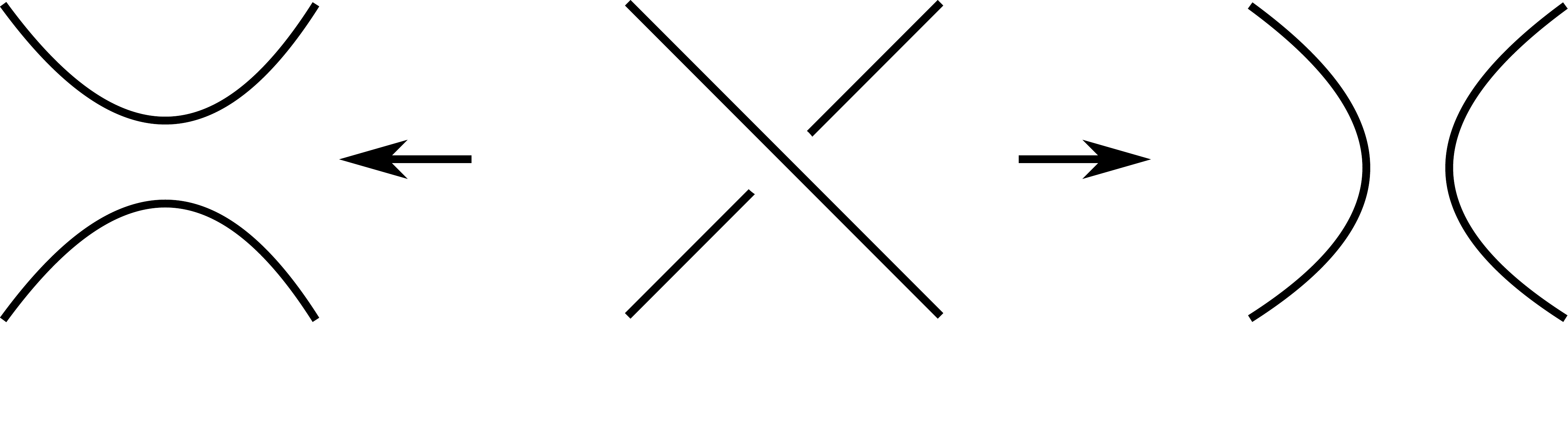}
\put(9.5, 19){$\bullet$}
\put(9.5, 13.75){$\bullet$}
\put(86.3,16.5){$\bullet$}
\put(91.7,16.5){$\bullet$}
\put(0,29){$e_{i}$}
\put(0,4){$e_{j}$}
\put(19, 29){$e_{k}$}
\put(19,4){$e_{l}$}
\put(0,-2){0-resolution}
\put(40,29){$e_i$}
\put(40,4){$e_j$}
\put(59,29){$e_k$}
\put(59,4){$e_l$}
\put(80,29){$e_i$}
\put(80,4){$e_j$}
\put(99,29){$e_k$}
\put(99,4){$e_l$}
\put(80,-2){1-resolution}
\end{overpic}
\caption{}\label{resolutions}
\end{figure}

The module $CKh(D_{v})$ is defined to be a quotient of the ground ring:

\[ CKh(D_{v}) := R / \{X_{p}=X_{q} \text{ if } e_{p} \sim_{v} e_{q} \}  \]

\noindent
we will denote this quotient by $R_{v}$.

There is a partial ordering on $\{0,1\}^{n}$ obtained by setting $u \le v$ if $u_{i} \le v_{i}$ for all $i$. We will write  $u \lessdot v$ if  $u \le v$ and they differ at a single crossing, i.e. there is some $i$ for which $u_{i}=0$ and $v_{i}=1$, and $u_{j}=v_{j}$ for all $j \ne i$. Corresponding to each edge of the cube, i.e. a pair $(u \lessdot v)$, there is an embedded cobordism in $\mathbb{R}^{2} \times [0,1]$ from $D_{u}$ to $D_{v}$ constructed by attaching a 1-handle near the crossing $c_{i}$ where $u_{i}<v_{i}$. This cobordism is always a pair of pants, either going from one circle to two circles (when $k_{u}=k_{v}-1$) or from two circles to one circle (when $k_{u}=k_{v}+1$). We call the former a \emph{merge} cobordism and the latter a \emph{split} cobordism.

For each vertex $v$ of the cube, the quotient ring $R_{v}$ is naturally isomorphic to $\mathcal{A}^{\otimes k_{v}}$, where $\mathcal{A}$ is the Frobenius algebra $\mathbb{Q}[X]/(X^{2}=0)$. Recall that the multiplication and comultiplication maps of $\mathcal{A}$ are given as:

\begin{displaymath}
m:\mathcal{A}\otimes_{\mathbb{Q}}\mathcal{A}\rightarrow\mathcal{A}:\begin{cases}
\begin{array}{lcc}
1 \mapsto 1&,&X_{1} \mapsto X\\
X_{1}X_{2} \mapsto 0&,&X_{2} \mapsto X
\end{array}
\end{cases}
\end{displaymath}
and 
\begin{displaymath}
\Delta:\mathcal{A}\rightarrow\mathcal{A}\otimes_{\mathbb{Q}}\mathcal{A}:\begin{cases}
1 \mapsto X_{1}+X_{2}\\
X \mapsto X_{1}X_{2}
\end{cases}
\end{displaymath}




The chain complex $CKh(D)$ is defined to be the direct sum of the $CKh(D_{v})$ over all vertices in the cube:

\[ CKh(D) := \bigoplus_{v \in \{0,1\}^{n}} CKh(D_{v}) \]

The differential decomposes over the edges of the cube. When $u \lessdot v$ corresponds to a merge cobordism, define 

\[   d_{u,v}: CKh(D_{u}) \to CKh(D_{v})    \]

\noindent
to be the Frobenius multiplication map, and when $u \lessdot v$ corresponds to a split cobordism, define $d_{u,v}$ to be the comultiplication map. In terms of the quotient rings $R_{u}$ and $R_{v}$, the map $m$ is projection, while $d$ is multiplication by $X_{j}+X_{k}$, where $e_{i}$, $e_{j}$, $e_{k}$, $e_{l}$ are the edges at the corresponding crossing as in Figure \ref{resolutions}. Note that $X_{j}+X_{k}=X_{i}+X_{l}$.

If $D_{u}$ and $D_{v}$ differ at crossing $c_{i}$, define $\epsilon_{u,v} = \sum_{j < i } u_{j} $. Then 

\[   d = \sum_{u \lessdot v} (-1)^{\epsilon_{u,v}} d_{u,v}   \]

The Khovanov complex is bigraded, with a homological grading and a quantum grading. Up to an overall grading shift, the homological grading is just the height in the cube. Setting $|v|= \sum_{i} v_{i}$, $n_{+}$ the number of positive crossings in $D$, an $n_{-}$ the number of negative crossings in $D$, we have 

\[ \mathrm{gr}_{h}(R_{v}) = |v|-n_{-}\]

For each vertex $v$ of the cube, the quantum grading of $1\in R_{v}$ is given by 

\[ \mathrm{gr}_{q}(1 \in R_{v}) = n_{+}-2n_{-}+|v|+k_{v} \]

\noindent
and each variable $X_{i}$ has quantum grading $-2$. With respect to the bigrading $(\mathrm{gr}_{h}, \mathrm{gr}_{q})$, the differential $d$ has bigrading $(1,0)$. The Khovanov homology $Kh(D)$ is the homology of this complex

\[ Kh(D) = H_{*}(CKh(D), d)       \]

There is a third grading which will be of interest called the $\delta$-grading. It is given by 
\[ \gr_{\delta} = \gr_{q} - 2 \gr_{h} \]

\noindent
The differential is homogeneous of degree $-2$ with respect to the $\gr_{\delta}$, and multiplication by $X_{i}$ lowers $\gr_{\delta}$ by 2.

The grading on our complex $C_{1 \pm 1}(D)$ will correspond to the $\delta$-grading on Khovanov homology.

\section{Construction of $C_{1 \pm 1}$ for singular knots}

The complex $C_{1 \pm 1}(D)$ is constructed as an oriented cube of resolutions. Each crossing can be resolved in one of two ways - with the oriented smoothing smoothing, or a singularization (see Figure \ref{OrientedResolutions}). 

\begin{figure}[ht]
\centering
\def\svgwidth{10cm}
\small
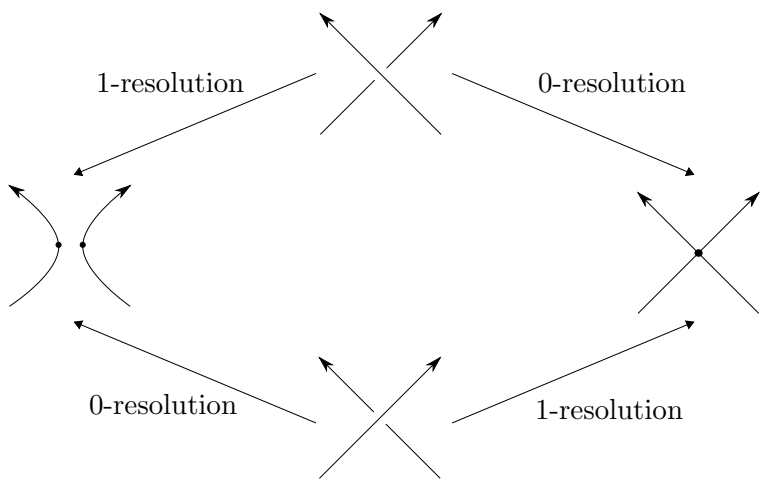
\caption{Resolutions for a positive and negative crossing}\label{OrientedResolutions}
\end{figure}

\begin{definition}

A diagram $S$ which has been obtained from $D$ by replacing every crossing with the oriented smoothing or singularization is called a \emph{complete resolution} of $D$.

\end{definition}

\begin{remark}

A complete resolution $S$ can also be viewed as a trivalent graph by replacing each 4-valent vertex with a wide edge, as in Figure \ref{WideEdge}.

\end{remark}

\begin{figure}[ht]
\centering
\def\svgwidth{10cm}
\small
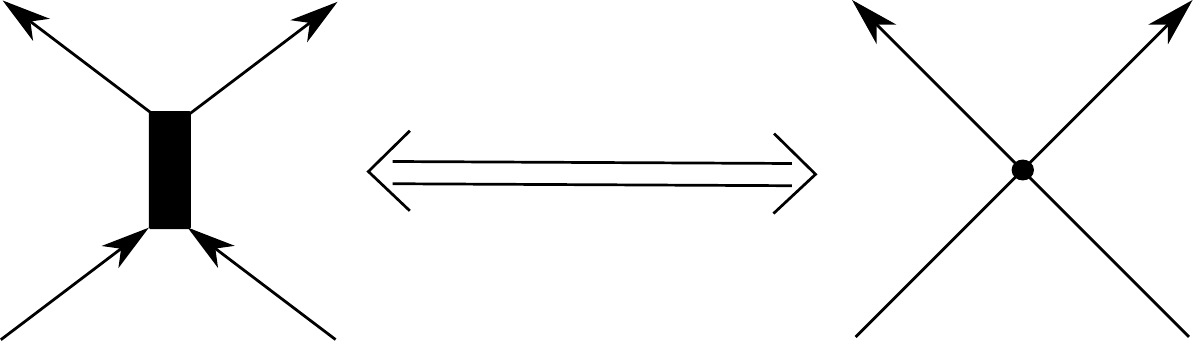
\caption{Wide edge vs singularization notation}\label{WideEdge}
\end{figure}

In this section we will describe the chain complex $C_{1\pm1}(S)$ for a complete resolution $S$. In the following section, we will give a local theory for the complex, which will make the construction of the edge differential $d_{1}$.

\subsection{Plat and Singular Closures of Braids} \label{closuresection}

Let $b\in Br_{2n}$ be a braid with $2n$ strands. The \emph{plat closure} of $b$, denoted by $\pc(b)$, is defined to be the link obtained by joining the pairs of strings consecutively at the top and bottom. More precisely, given tangles 
\[T_n^-=\cup\ \cup\  ...\ \cup\ \ \ \ \ \ \text{and}\ \ \ \ \ \ T_n^+=\cap\ \cap\  ...\ \cap\]
consisting of $n$ cups and $n$ caps, respectively, $\pc(b)=T_n^-bT_n^+$, as in Figure \ref{PlatClosure}. We say that $D$ is a plat diagram if $D=\pc(b)$ for some braid $b$. For plat closure we have the following well-known analogue of Alexander's theorem \cite{Alex:BraidClosure}.

\begin{figure}[ht]
\centering
\def\svgwidth{10cm}
\small
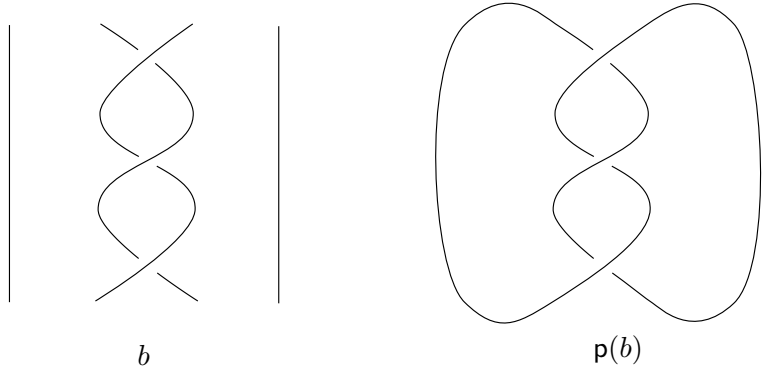
\caption{A braid $b$ and its plat closure $\pc(b)$}\label{PlatClosure}
\end{figure}

\begin{theorem}
For every link $L$, there is some braid $b$ with an even number of strands strands such that $L=\pc(b)$. 
\end{theorem}

The construction of $C_{1 \pm 1}(D)$ will take as input a plat diagram $D$. However, it will often be useful to replace the plat closure $\pc(b)$ with a diagram $\x(b)$ which we will call the \emph{singular closure}. The singular closure of a braid $b$ is obtained by adding a singularization at the top of the braid between strands $2i-1$ and $2i$ for $i=1,...,n$, then taking the standard braid closure, as in Figure \ref{SingularClosure}.

\begin{figure}[ht]
\centering
\def\svgwidth{12cm}
\small
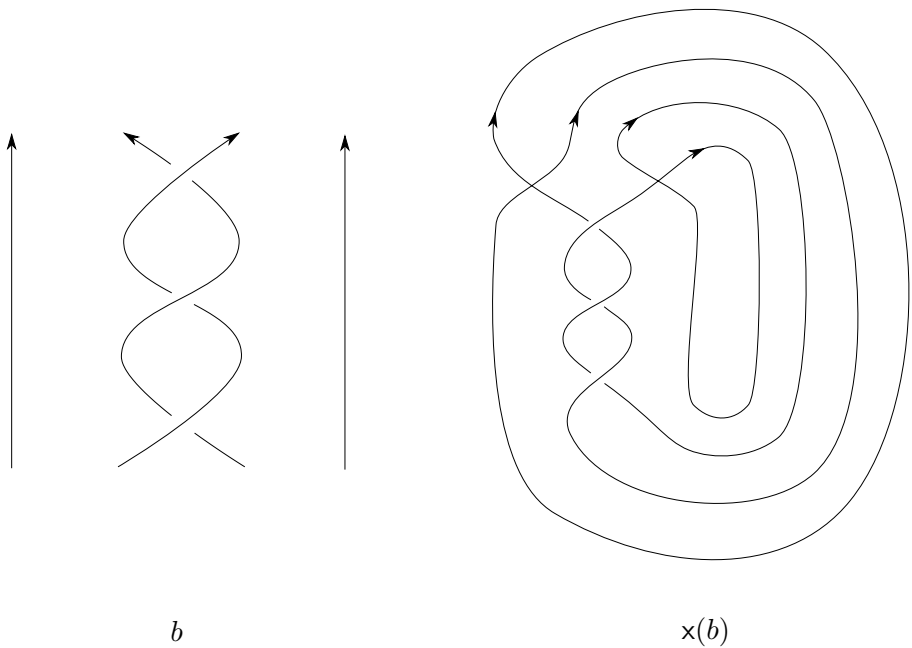
\caption{A braid $b$ and its singular closure $\sing(b)$}\label{SingularClosure}
\end{figure}

The reason we are interested in the singular braid is that in $\mathit{HFK}_{2}$, singularizations are closely related to the unoriented smoothing. Replacing the singularizations in $\sing(b)$ with unoriented smoothings gives the plat closure $\pc(b)$. However, the orientations are better behaved on $\sing(b)$, and it gives a pairing between the cups and caps in $\pc(b)$.


\subsection{Singular braids,  the edge ring, and cycles} \label{sub3.2}

\begin{definition}
A \emph{singular braid} is a properly embedded $4$-valent graph in $\RR\times [0,1]$ obtained by replacing all crossings of a braid diagram with singular points.  
\end{definition}

Let $S$ be a singular braid with $n$ strands. Then $S$ is an oriented graph with $n$ incoming edges, $n$ outgoing edges and all interior vertices have two incoming and two outgoing edges.  To each edge $e_i$ of $S$, we assign a variable $U_i$, and define the \emph{edge ring} of $S$, to be
\[R(S)=\QQ[U_1,...,U_m].\]
Here, $m$ is the number of edges of $S$. 

Let $S$ be a singular braid with $2n$ strands.  We define \emph{plat closure} of $S$, denoted by $\pc(S)$, to be the oriented graph obtained by attaching the $n$ oriented, singular caps (Figure~\ref{S-caps}) and $n$ oriented, singular cups (Figure~\ref{S-cups}), to the top and bottom of $S$, respectively. 

\begin{figure}[ht]
\centering
\def\svgwidth{6.5cm}
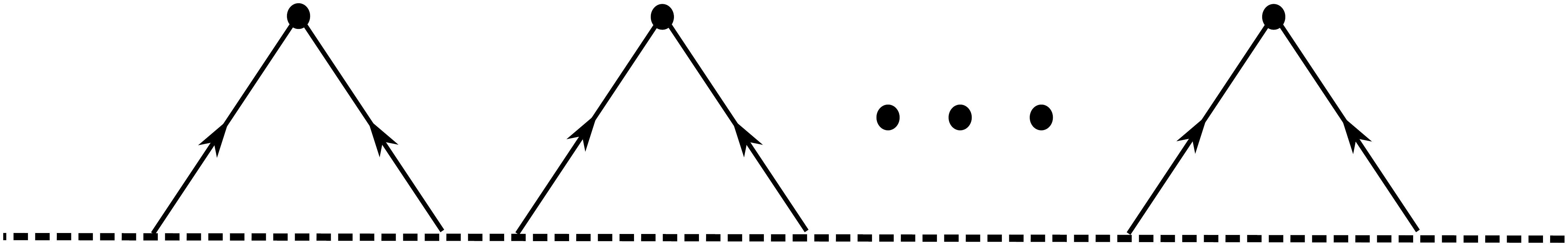
\caption{}\label{S-caps}
\end{figure}

\begin{figure}[ht]
\centering
\def\svgwidth{6.5cm}
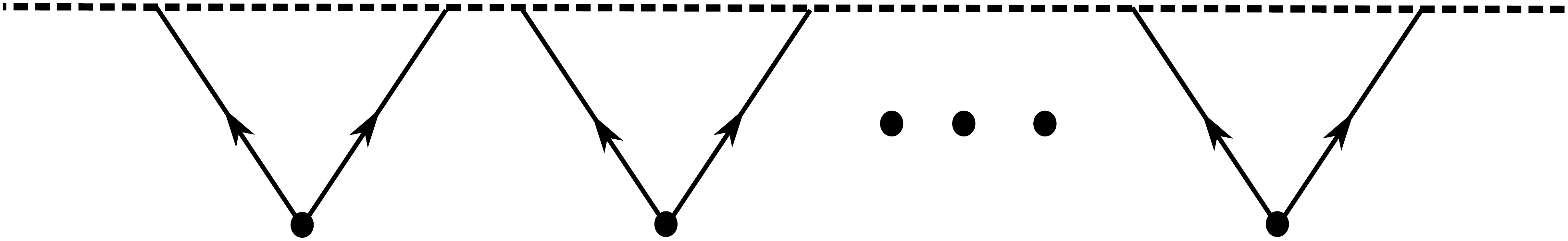
\caption{}\label{S-cups}
\end{figure}

The result is a singular oriented graph with some $4$-valent interior vertices, $n$ $2$-valent bottom vertices and $n$ $2$-valent top vertices. Bottom vertices have two outgoing edges and no incoming edges, while top vertices have two incoming edges and no outgoing edges.

\begin{definition}
A \emph{cycle} $Z$ in $\pc(S)$ is a subset of edges in $S$ such that their union is $n$ pairwise disjoint paths connecting bottom vertices to top vertices such that $\bdy Z=w_1^++...+w_n^+-w_1^--...-w_n^-$. Here, $w_1^+,...,w_n^+$ denote the top vertices and $w_1^-,...,w_n^-$ denote the bottom vertices.
\end{definition}

\begin{remark} Note that the graph $\pc(S)$ can be viewed as the singular closure of $S$ pulled apart at the $n$ singular points involved in the closure. If we were to reidentify each $w_{i}^{+}$ with $w_{i}^{-}$, then we can view cycles as subsets of $\sing(S)$ which are homeomorphic to $n$ disjoint circles.
\end{remark}


For any $4$-valent vertex $v$ of $S$,  assume $e_{i(v)}$ and $e_{j(v)}$ are the left and right incoming edges and $e_{k(v)}$ and $e_{l(v)}$ are the left and right outgoing edges.  Associated to $v$ we define a quadratic polynomial:
\[
Q_v=U_{i(v)}U_{j(v)}-U_{k(v)}U_{l(v)}
\]
in $R(S)$. The edges in $\pc(S)$ are in one-to-one correspondence with the edges in $S$. Considering this correspondence, any cycle $Z$ specifies two ideal
\[
\begin{split}
&I_Z=\langle U_i\ |\ \text{for any}\ e_i\subset Z\rangle\\
&Q_Z=\langle Q_v\ |\ v\ \text{is any $4$-valent vertex disjoint from $Z$} \rangle
\end{split}
\]
in $R(S)$ and we define the $\QQ$-module
 \[R_Z=\frac{\QQ[U_1,...,U_m]}{Q_Z+I_Z}.\]
Finally, associated to $S$ we define the $\QQ$-module
\[M(S)=\bigoplus _{Z\in c(S)}R_{Z}\]
where $c(S)$ denotes the set of cycles in $\pc(S)$. 

For any pair $(Z,e_i)$ of a cycle $Z$ in $\pc(S)$ and an edge $e_i$ in $Z$, let $\mathcal{D}(Z,e_i)$ denote the set of disks in the plane satisfying the followings:
\begin{enumerate}
\item Each $D\in \mathcal{D}(Z,e_i)$ is a bounded subset of $\RR^2$ homeomorphism to a disk so that $\bdy D\subset S$. 
\item There are vertices $v_{t}(D)$ and $v_{b}(D)$ on the boundary of $D$ so that they decompose the boundary as $\bdy D=\bdy_L D\cup_{\{v_t(D),v_b(D)\}}\bdy_RD$ where both $\bdy_L D$ and $\bdy_RD$ are oriented from $v_b(D)$ to $v_t(D)$. 
\item $\bdy_LD$ is a subset of $Z$ containing $e_i$.
\end{enumerate}
The intersection of any two disk in $\mathcal{D}(Z,e_i)$ is a disk in  $\mathcal{D}(Z,e_i)$. So, let $D(Z,e_i)$ denote the \emph{smallest} disk in this set i.e. the intersection of all disks. See Figure \ref{Discs} for an example. We allow the special case that $\mathcal{D}(Z,e_i)$ is the empty set - in this case, set $D(Z,e_i)$ to be the empty disc.

Similarly, for any vertex $v$ in $Z$, we define $\mathcal{D}(Z,v)$ to be the set of all disks satisfying the above conditions, just replace $e_i$ with $v$ in $(3)$ and assume $v_t(D)$ and $v_b(D)$ are distinct from $v$ in $(2)$ . Also, denote the intersection of all such disks with $D(Z,v)$.

For any edge $e$ in $S$, let $v_t(e)$ and $v_b(e)$ denote the top and bottom vertices of $e$. Let $D$ be a disk in $\mathcal{D}(Z,e_i)$ (or $\mathcal{D}(Z,v)$). We define $I(D)$ to be the set of edges $e$ in $S$ that intersect $D$ in $v_t(e)$, and $v_t(e)$ lies on the segment of $Z$ from $v_b(D)$ and $v_{b}(e_i)$ (respectively, $v$), excluding $v_b(D)$. In other words, $I(D)$ consists of incoming edges on the left boundary of $D$ from $v_b(D)$ to $v_{b}(e_i)$. Similarly, $O(D)$ is the set of edges, pointing to the left between $v_t(e_j)$ and $v_{t}(D)$. More precisely, each edge $e$ in $O(D)$ intersects $D$ in $v_b(e)$ and this vertex lies on the segment of $Z$ from $v_t(e_j)$ (respectively, $v$) to $v_{t}(D)$, excluding $v_{t}(D)$.

\begin{figure}[ht]
\centering
\def\svgwidth{8.5cm}
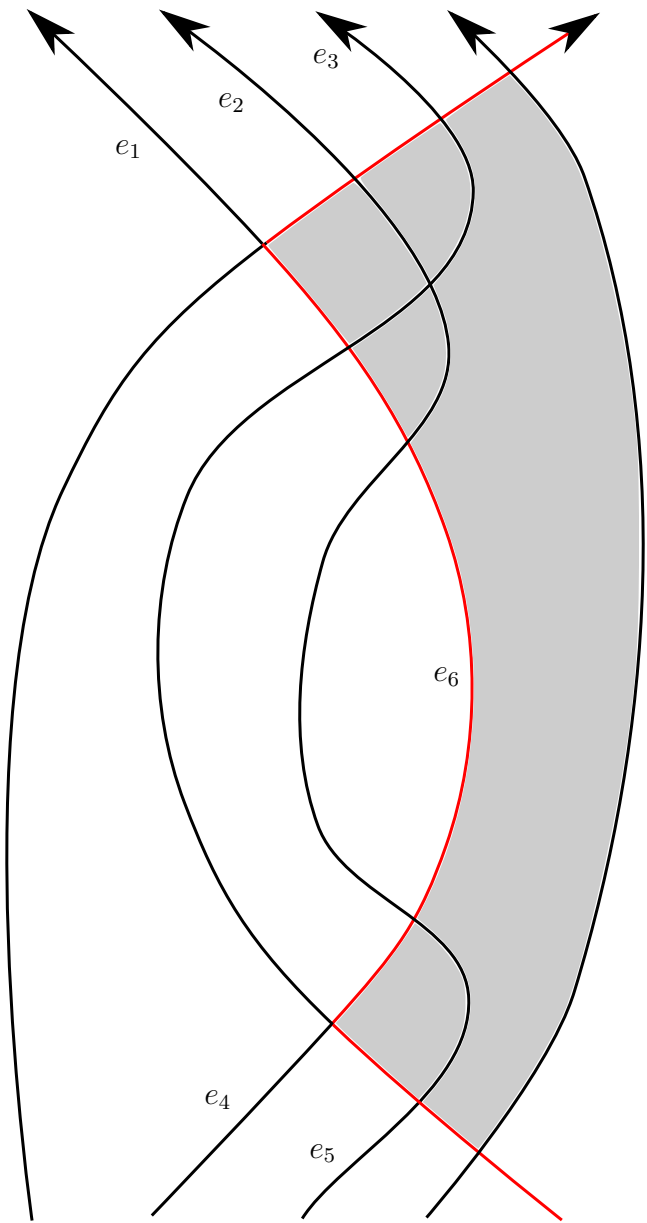
\caption{The grey region is the disc $D(Z,e_{6})$ where the edges in $Z$ are red. In this case, $I(D(Z,e_{6}))=\{e_{4},e_{5}\}$ and $O(D(Z,e_{6}))=\{e_{1},e_{2},e_{3}\}$, so $U(D(Z,e_{6})) = U_{1}U_{2}U_{3}U_{4}U_{5}$.}\label{Discs}
\end{figure}

For any disk $D$ in $\mathcal{D}(Z,\star)$, where $\star\subset Z$ is an edge or vertex, we define a monomial in $R(S)$, called its \emph{coefficient}, by setting 
 \[U(D)=\prod_{e_i\subset \left(I(D)\cup O(D)\right)}U_i \in R(S).\]
We define an $R(S)$-module structure on $M(S)$ by setting
\[
U_i x_{Z}=\begin{cases}
\begin{array}{ll}
U_ix_Z&\text{if}\ e_i\not\subset Z\\
U\left(D(Z,e_i)\right)x_{U_i(Z)}&\text{if}\ e_i\subset Z\ \text{and $\bdy_R(D(Z,e_i))$ only intersects Z} \\
&\text{at its endpoints}\\
0&\text{otherwise}
\end{array}
\end{cases}
\]
Here, $x_Z$ denotes $1$ in the summand $R_Z$ of $M(S)$, and $U_i(Z)$ denotes the cycle obtained from $Z$ by replacing $\bdy_LD(Z,e_i)$ with $\bdy_RD(Z,e_i)$ (see Figure \ref{CycleExample}). Note that when $\mathcal{D}(Z,e_i)=\emptyset$, $U_{i}x_{Z}=0$.

For $e_i\subset Z$, we refer to the equality $U_ix_Z=U\left(D(Z,e_i)\right)x_{U_i(Z)}$ as \emph{$U_i$ maps the cycle $Z$ to $U_i(Z)$ with coefficient $U(D(Z,e_i))$}.

\begin{figure}[ht]
\centering
\def\svgwidth{7.5cm}
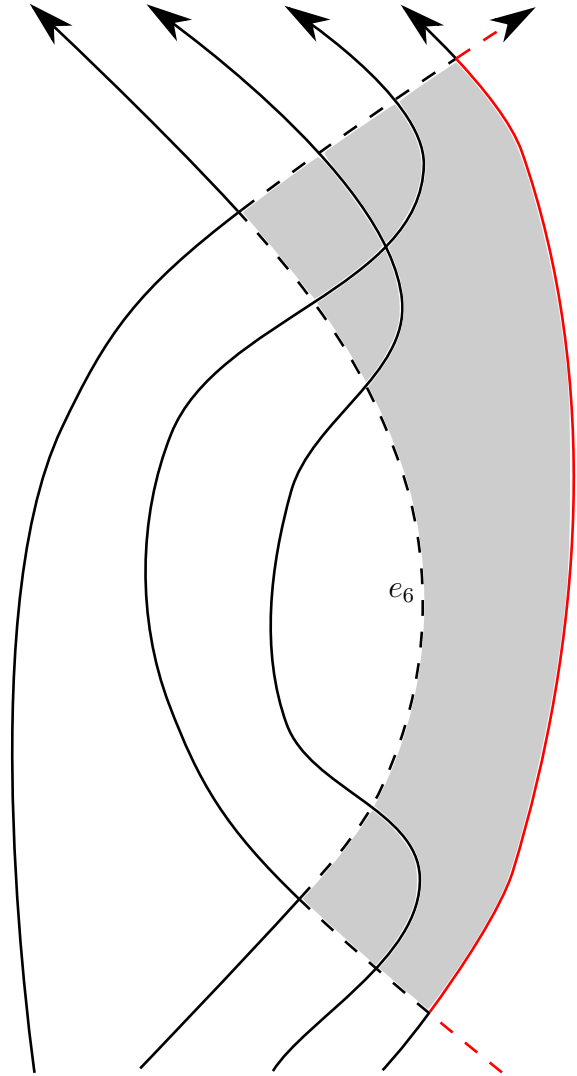
\caption{The cycle $e_{6}(Z)$ is in red, where the  cycle $Z$ consists of the dashed edges.}\label{CycleExample}
\end{figure}

Similarly, for any vertex $v$ in $Z$, if replacing $\bdy_LD(Z,v)$ by $\bdy_R D(Z,v)$ in $Z$ gives a cycle, denote it by $v(Z)$, and define  
\[f_v:M(S)\longrightarrow M(S)\]
by setting:
\[
f_v(x_Z)=
\begin{cases}
\begin{array}{ll}
U_{i(v)}U_{j(v)}x_Z=U_{k(v)}U_{l(v)}x_Z&\text{if}\ v\subset Z\\
U(D(Z,v))x_{v(Z)}&\text{if}\ v\subset Z\ \text{and $\bdy_R(D(Z,v))$ only intersects Z} \\
&\text{at its endpoints}\\
0&\text{otherwise}
\end{array}
\end{cases}
\]

In the next Section, we will show that the $R(S)$-module structure on $M(S)$ is well-defined, and $f_v=U_{i(v)}U_{j(v)}=U_{k(v)}U_{l(v)}$ for any vertex $v$ in $Z$.

%

%
%
%
%
%
%
%
%

\subsection{Well-definedness of Module Structure} \label{welldef}
 Let $v$ be a 4-valent vertex in $\pc(S)$ with incoming edges $e_{1}$ and $e_{2}$ and outgoing edges $e_{3}$ and $e_{4}$. When a cycle $Z$ is locally empty at $v$, the associated R-module $R_{Z}$ comes with the relation $Q_{v}=U_{1}U_{2}-U_{3}U_{4}$. However, it is possible that multiplication by some $U_{i}$ will map this cycle so that it is locally $e_{1}e_{3}$, $e_{1}e_{4}$, $e_{2}e_{3}$, or $e_{2}e_{4}$. Thus, we need to show that $U_{1}U_{2}=U_{3}U_{4}$ on these local cycles.

Consider the local cycle $e_{1}e_{3}$. As shown in Figures \ref{U1U2} and \ref{fv}, multiplication by $U_{1}$ followed by multiplication by $U_{2}$ maps to the same cycle as applying $f_{v}$. Moreover it does so with the same coefficient, as the right boundary of the yellow disc will not have any outgoing edges.

\begin{figure}[ht]
\centering
\def\svgwidth{6cm}
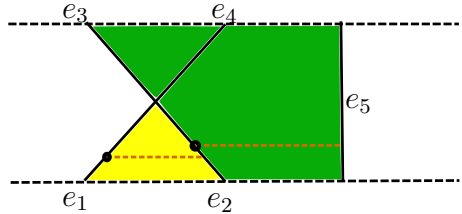
\caption{Multiplication by $U_{1}U_{2}$ on the local cycle $e_{1}e_{3}$}\label{U1U2}
\end{figure}

\begin{figure}[ht]
\centering
\def\svgwidth{6cm}
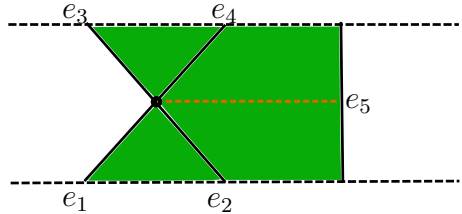
\caption{Applying $f_{v}$ to the local cycle $Ze_{1}e_{3}$}\label{fv}
\end{figure}

The same argument applies to multiplication by $U_{3}$ followed by multiplication by $U_{4}$. Thus, we have that on the local cycle $e_{1}e_{3}$, $U_{1}U_{2}=U_{3}U_{4}=f_{v}$.

\begin{figure}[ht]
\centering
\def\svgwidth{6cm}
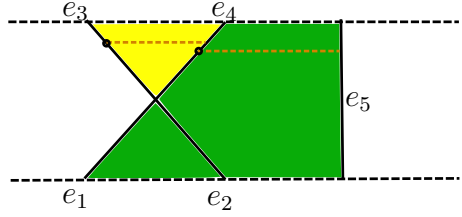
\caption{Multiplication by $U_{3}U_{4}$ on the local cycle $e_{1}e_{3}$}\label{U3U4}
\end{figure}

The other 3 local cycles behave in much the same way, with $U_{1}U_{2}=U_{3}U_{4}=f_{v}$. The only difference in the proof is that there is a local outgoing edge that appears as a coefficient in $f_{v}$ - fortunately, this edge also appears as a coefficient in both $U_{1}U_{2}$ and $U_{3}U_{4}$. 

\begin{remark}

A consequence of this argument is that for any 4-valent $v$ in $S$, $Q_{v}M(S)=0$.

\end{remark}

The second thing that we need to show is that the edge actions are actually commutative. More precisely, we need to show that for $x \in R_{Z}$, $U_{i}U_{j}x=U_{j}U_{i}x$ for any pair of edges $e_{i}$, $e_{j}$. There are 3 cases to consider: 

\vspace{2mm}

1) Both $e_{i}$ and $e_{j}$ are in $Z^{c}$.

2) One edge is in $Z$ (WLOG assume this is $e_{i}$) and the other edge ($e_{j}$) is in $Z^{c}$. 

3) Both $e_{i}$ and $e_{j}$ are in $Z$.

\vspace{2mm}

The first two cases follow directly from the definitions. However, the third case is fairly subtle, as there can be interactions between the two multiplications. Before proving this case, we will need a lemma inspired by Khovanov-Rozansky (\hspace{1sp}\cite{khovanov2008matrix}, \cite{KhovanovRozansky08:MatrixFactorizations}).

\begin{lemma} \label{ProductLemma}

Let $Z$ be a cycle in $S$, and let $D$ be an embedded disc in $\mathbb{R} \times [0,1]$ whose boundary is transverse to $S$. Suppose also that $Z \cap D = \emptyset$ and that $D$ doesn't contain any caps or cups from the closing off portions of $S$. Then 

\[   \prod_{ e_{i} \in In(D)}U_{i} = \prod_{e_{j} \in Out(D)}U_{j}      \]

\noindent
in $R_{Z}$, where $In(D)$ is the set of incoming edges in $D \cap S$ and $Out(D)$ is the set of outgoing edges in $D \cap S$.

\end{lemma}

\begin{proof}

To see this, we just start with the product of the incoming edges. Let $v_{1}$,...,$v_{k}$ be the vertices in $S$ which are in the interior of $D$, ordered by their first coordinate in $\mathbb{R} \times I$. If we apply the quadratic relations $Q(v_{i})$ sequentially, the result is the product of the outgoing edges.

\end{proof}

For case 3, there are several subcases to consider:

\vspace{2mm}

a) $e_{i}$ and $e_{j}$ lie on the same strand in $Z$

b) $e_{i}$ and $e_{j}$ lie on adjacent strands in $Z$

c) $e_{i}$ and $e_{j}$ lie on different strands in $Z$ which are not adjacent.

\vspace{2mm}

For (a), assume $e_{i}$ is above $e_{j}$. If $\Int(D(Z, e_{i})) \cap \Int(D(Z, e_{j})) = \emptyset$, then again we can see right away that $U_{i}U_{j}x_{Z}=U_{j}U_{i}x_{Z}$. However, if this is not the case, then the two multiplications interact. There are 
4 possible ways that they can interact, based on whether $e_{j}$ is in $\partial D(Z,e_{i})$, and whether $e_{i}$ is in $\partial D(Z,e_{j})$. One of these cases is shown in Figures \ref{Commute1a} and \ref{Commute1b}. In Figure \ref{Commute1a}, we see that $U_{i}U_{j}x_{Z}$ and $U_{j}U_{i}x_{Z}$ end up on the same cycle. To see that the coefficients are the same, we apply Lemma \ref{ProductLemma} to the shaded discs in Figure \ref{Commute1b}. The cases where $e_{j} \in \partial D(Z,e_{i})$ or $e_{i} \in \partial D(Z,e_{j})$ are similar, so we leave them to the reader.

\begin{figure}[ht]
\centering
\large
\def\svgwidth{5cm}
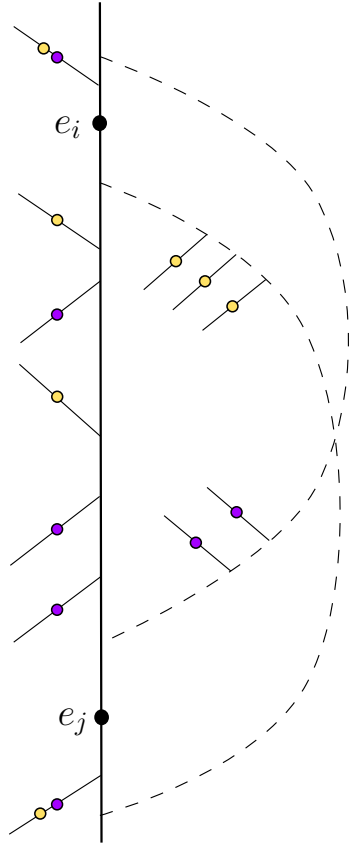
\caption{\textbf{Multiplication by} $U_{i}U_{j}$ \textbf{vs. multiplication by} $U_{i}U_{j}$\textbf{:} The vertical line is the cycle $Z$, and the dashed lines indicate how $Z$ changes under multiplication by $U_{i}$ and $U_{j}$. The coefficients from multiplication by $U_{i}U_{j}$ are indicated with yellow dots, and the coefficients from multiplication by $U_{j}U_{i}$ are indicated with purple dots.}\label{Commute1a}
\end{figure}

\begin{figure}[ht]
\centering
\large
\def\svgwidth{5cm}
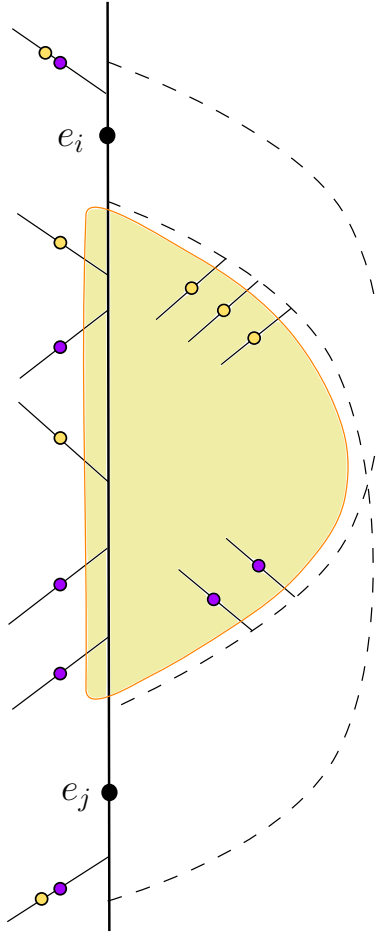
\caption{The product of the edges with yellow dots can be identified with the product of the edges with purple dots by applying Lemma \ref{ProductLemma} to the shaded disc.} \label{Commute1b}
\end{figure}

For (b), assume that $e_{i}$ lies on the strand in $Z$ to the left of the strand $e_{j}$ lies on. If $D(Z, e_{i}) \cap D(Z, e_{j}) = \emptyset$, then $U_{i}$ and $U_{j}$ clearly commute. However, if they intersect at all (even on the boundary this time) then we get interactions. In particular, $U_{i}x_{Z} = 0$, so $U_{j}U_{i}x_{Z} = 0$. Thus, to show commutativity, we need to show that $U_{i}U_{j}x_{Z}=0$. 

Suppose $v$ is the 4-valent vertex in $D(Z, e_{i}) \cap D(Z, e_{j})$ which is closest to the top of the braid. If $e_{a}$ is the left outgoing edge of $v$ and $e_{b}$ is the right outgoing edge of $v$, then $e_{b}$ is in $Z$, so $U_{a}$ is a coefficient of $U_{j}x_{Z}$. But $U_{a}$ lies in $U_{i}(U_{j}(Z))$, so $U_{i}U_{j}x_{Z}$ gets mapped farther to the right. Since $U_{b}$ lies in $U_{a}(U_{i}(U_{j}(Z)))$, we can travel along this strand in $Z$ to the point where $U_{a}(U_{i}(U_{j}(Z)))$ diverges from $Z$, and repeat this argument. After doing this some number of times, the cycle will hit the top vertex in $D(Z, e_{j})$, making $U_{i}U_{j}x_{Z}$ equal to zero.

For (c), there is no possible way for the two multiplications to interact, so they commute.

\subsection{Definition of the Complex} \label{defcomplex}
The goal of this section is to introduce a chain complex for any singular braid $S$. Let $R=R(S)$, the edge ring of $S$. To any $4$-valent vertex $v$ we associate the linear elements
\[L_v=U_{i(v)}+U_{j(v)}-U_{k(v)}-U_{l(v)}\in R.\]
\[L'_v=U_{i(v)}+U_{j(v)}+U_{k(v)}+U_{l(v)}\in R.\]
Recall, that indices $i(v), j(v), k(v)$ and $l(v)$ are defined so that $e_{i(v)}$ and $e_{j(v)}$ are the left and right incoming edges, respectively, while $k(v)$ and $l(v)$ are left and right outgoing edges, respectively. Let $L \subset R$ to be the ideal generated by the  $L_v$ corresponding to all $4$-valent vertices. We define the $R$-module $\mathscr{M}(S)$ as
\[\mathscr{M}(S):=M(S)\otimes_{R} R/L=M(S)/LM(S).\]

For any $1\le i\le n$, associated to the bottom and top vertices $w_i^-$ and $w_i^+$ in $\pc(S)$, we define the linear elements
\[
\begin{array}{ll}
L_{w_{i}^{\pm}}=U_{i(w_i^+)}+U_{j(w_i^+)}-U_{k(w_i^-)}-U_{l(w_{i}^-)}, &\text{and}\\
L'_{w_{i}^{\pm}}=U_{i(w_i^+)}+U_{j(w_i^+)}+U_{k(w_i^-)}+U_{l(w_{i}^-)}&
\end{array}
\]
in $R$ and the \emph{closing off} matrix factorization:

\begin{figure}[!h]
\centering
\begin{tikzpicture}
  \matrix (m) [matrix of math nodes,column sep=8em,minimum width=2em] {
     R &R\\};
  \path[-stealth]
    (m-1-1) edge [bend left=15] node [above] {$L_{w_i^\pm}$} (m-1-2)
    (m-1-2) edge [bend left=15] node [below] {$L'_{w_i^{\pm}}$} (m-1-1);
  \end{tikzpicture}
 \end{figure}
 
Let $\mathsf{K}(S)$ denote the Koszul complex
\[ \mathsf{K}(S)=
\big{(}
\xymatrix{R\ar@<1ex>[r]^{L_{w_1^\pm}}&R\ar@<1ex>[l]^{L'_{w_1^{\pm}}}}\big{)}\otimes\big{(}
\xymatrix{R\ar@<1ex>[r]^{L_{w_2^\pm}}&R\ar@<1ex>[l]^{L'_{w_2^{\pm}}}}\big{)}\otimes...\otimes\big{(}
\xymatrix{R\ar@<1ex>[r]^{L_{w_n^\pm}}&R\ar@<1ex>[l]^{L'_{w_n^{\pm}}}}\big{)}\]

\begin{definition}

We define $C_{1\pm 1}(S)$ to be the tensor product $\mathscr{M}(S)\otimes \mathsf{K}(S)$.

\end{definition}

 
\begin{lemma}
$C_{1+1}(S)$ is a chain complex i.e. $d^2=0$.
\end{lemma} 
\begin{proof}
For each $i$, the matrix factorization $\xymatrix{R\ar@<1ex>[r]^{ L_{w_i^\pm}}&R\ar@<1ex>[l]^{L'_{w_i^{\pm}}}}$ has potential $L_{w_i^\pm}L'_{w_i^\pm}$. Thus, to prove $d^2=0$, we need to show that \[(\sum_{i=1}^nL_{w_i^\pm}L'_{w_i^\pm})\mathscr{M}(S)=0.\]
For any $1\le i\le n$
\[\begin{split}
L_{w_i^\pm}L'_{w_i^\pm}&= (U_{i(w_i^+)}+U_{j(w_i^+)})^2-(U_{k(w_i^-)}+U_{l(w_i^-)})^2\\
&=U_{i(w_i^\pm)}^2+U_{j(w_i^+)}^2-U_{k(w_i^-)}^2-U_{l(w_i^-)}^2\\
&+2U_{i(w_i^+)}U_{j(w_i^+)}-2U_{k(w_i^-)}U_{l(w_i^-)}.
\end{split}
\]
Follows from the module structure on $\mathscr{M}(S)$ that \[U_{i(w_i^+)}U_{j(w_i^+)}x_Z=0\ \ \ \ \text{and}\ \ \ \ U_{k(w_i^-)}U_{l(w_i^-)}x_Z=0\]
for any cycle $Z$ and any $1\le i\le n$. Thus, it suffices to show that 
\[\left(\sum_{i=1}^n(U_{i(w_i^+)}^2+U_{j(w_i^+)}^2-U_{k(w_i^-)}^2-U_{l(w_i^-)}^2)\right)\mathscr{M}(S)=0.\]
On the other hand, 
\[\sum_{i=1}^n(U_{i(w_i^+)}^2+U_{j(w_i^+)}^2-U_{k(w_i^-)}^2-U_{l(w_i^-)}^2)=-\sum_{v}(U_{i(v)}^2+U_{j(v)}^2-U_{k(v)}^2-U_{l(v)}^2)
\]
where the second sum is over all $4$-valent vertices in $S$. For any 4-valent $v$ in $S$, we have that $L_{v}=Q_{v}=0$ on $\mathscr{M}(S)$. It follows that $U_{i(v)}^2+U_{j(v)}^2-U_{k(v)}^2-U_{l(v)}^2=0$ for any $v$, so the sum is zero as well.
\end{proof}

\subsection{Some Properties of the Edge Ring Action}

In this section we will give some relations among the $U_{i}$-actions on $C_{1 \pm 1}(S)$. Note that for a complete resolution $S$, $H_{1+1}(S) \cong H_{1-1}(S)$, as there are no edge maps. We will write the results in terms of $H_{1+1}(S)$, but they also apply to $H_{1-1}(S)$.

\begin{lemma} \label{prop1}

Suppose $e_{i}$ and $e_{j}$ are two incoming or outgoing edges at a vertex $v$ in $S$. Then $U_{i}=-U_{j}$ on $H_{1+1}(S)$.

\end{lemma}

\begin{proof}
Suppose first that $v=w_{i}^{+}$ for some $i$, so that the two incoming edges are $e_{i(w_{i}^{+})},e_{j(w_{i}^{+})} $. The complex $\mathsf{K}(S)$ includes the factor 
\[\xymatrix{R\ar@<1ex>[r]^{L_{w_i^\pm}}&R\ar@<1ex>[l]^{L'_{w_i^{\pm}}}} \]

\noindent
We can define a homotopy $H$ to act on this Koszul factor by 
\[\xymatrix{R\ar@<1ex>[r]^{1}&R\ar@<1ex>[l]^{1}} \]

\noindent
and by the identity on the other factors. Then $dH+Hd = 2U_{i(w_{i}^{+})}+2U_{j(w_{i}^{+})}$. Since we are working over $\mathbb{Q}$, this shows that $U_{i(w_{i}^{+})}=-U_{j(w_{i}^{+})}$ on homology.

Similarly, for $w_{i}^{-}$, we can use the homotopy $H'$ which acts on this Koszul factor by 
\[\xymatrix{R\ar@<1ex>[r]^{1}&R\ar@<1ex>[l]^{-1}} \]

\noindent
and by the identity on other factors. Then $dH'+Hd' = 2U_{k(w_{k}^{-})}+2U_{l(w_{l}^{-})}$, so $U_{k}(w_{k}^{-})=-U_{l}(w_{i}^{-})$ on homology.

For the 4-valent vertices, it helps to modify the complex to a quasi-isomorphic one. We can write $C_{1\pm 1}(S)$ as 
\[ M(S) \otimes \mathsf{K}(S) \otimes R/L \]

\noindent
Since the $L(v)$ form a regular sequence, $R/L$ can be replaced with the Koszul complex 
\[ \mathcal{L} = \bigotimes_{v}
\big{(}
\xymatrix{R\ar@<1ex>[r]^{L_{v}}&R\ar@<1ex>[l]^{L'_{v}}}\big{)}\]

Then $C_{1\pm 1}(S)$ is quasi-isomorphic to $M(S) \otimes \mathsf{K}(S) \otimes \mathcal{L}$. Using the same homotopies $H$ and $H'$ on the factor $\xymatrix{R\ar@<1ex>[r]^{L_{v}}&R\ar@<1ex>[l]^{L'_{v}}}$ gives the desired result.

\end{proof}

\begin{lemma} \label{prop2}

For any $i$, $U_{i}^{2}=0$ on $H_{1+1}(S)$.

\end{lemma}

\begin{proof}

For any two edges $e_{i}$ and $e_{j}$ on the same component of $\sm(S)$, this follows from the lemma. Since $S$ is connected, the quadratic relations $Q_{v}$ force $U_{i}^{2}=U_{j}^{2}$ for any $U_{i}$, $U_{j}$. Together with the observation that on homology $U^{2}_{i(w^{+}_{1})}=-U_{i(w^{+}_{1})}U_{j(w^{+}_{1})}=0$, this proves the corollary.

\end{proof}

Suppose $\sm(\pc(S))$ consists of $k$ circles. Let $X_{i}=(-1)^{l}U_{j}$, where $e_{j}$ is an edge on the $i$th circle in $s(S)$ and lies on the $l$th strand of $S$.

\begin{corollary} \label{modulelemma1}

The homology $H_{1+1}(S)$ is a finitely generated module over 
\[\mathcal{A}^{\otimes k} = \QQ[X_{1},...,X_{k}]/(X_{1}^{2}=...=X_{k}^{2}=0)\]

\end{corollary}

\section{A Local Bimodule Representation and the Total Complex}

\subsection{A Bimodule for Tangles}
In this section we will define an algebra $A_n$ for any positive integer $n$ and a bimodule $\mathsf{M}[S]$ over $A_n$ for any singular braid $S$ with $2n$ strands such that 
\[\mathsf{M}[S_1]\otimes_{A_n}\mathsf{M}[S_2]=\mathsf{M}[S_1\circ S_2]\]
for all $2n$-strand singular braids $S_1$ and $S_2$. In addition, we will define a right $A_n$-module $\mathsf{M}[S_{cup}]$ for the singular cups and a left $A_n$-module $\mathsf{M}[S_{cap}]$ for singular cap such that 
\[\mathsf{M}[S_{cup}]\otimes_{A_n}\mathsf{M}[S]\otimes_{A_n}\mathsf{M}[S_{cap}]=\mathscr{M}(S)\]
for every singular braid $S$ with $2n$ strands. 

\subsubsection{The algebra $A_n$}  Let $[2n]$ denote the set $\{1,2,...,2n\}$.  The algebra $A_{n}$ is generated over $\mathbb{Q}[u_1,...,u_{2n}]$ by monotone bijections $P: S_{1} \to S_{2}$, where $S_{1}$ and $S_{2}$ are $n$-element subsets of $[2n]$. By \emph{monotone} we mean that if $i<j$, then $P(i)<P(j)$. These generators can be viewed pictorially as a strands algebra with $n$ strands out of $2n$ slots. The monotonicity imposes the restriction that there are no crossings between the strands.

Given $P_{1}: S_{1} \to S_{2}$ and $P_{2}: S_{3} \to S_{4}$, we define the product $P_{1}P_{2}=\prod_{i=1}^nu_i^{\alpha_i}P$ where $P$ is the concatenation of the two strands diagrams when $S_{2}=S_{3}$, and $0$ otherwise, and 
\[\alpha_i=\#\left\{j\ |\ \max\{j,P_2(P_1(j))\}\le i<P_1(j)\ \text{or}\ \min\{j,P_2(P_1(j))\}\ge i>P_1(j)\right\}.\]
See Figure \ref{Strands} for an $n=2$ example. These strands diagrams are drawn vertically, with left-to-right multiplication corresponding to bottom-to-top concatenation.

The idempotents are given by $id: S \to S$, which are denoted $\iota_{S}$. Note that $\sum_{S} \iota_{S} = 1$.

\begin{figure}[h!]
\begin{subfigure}{.32\textwidth}\scriptsize
 \centering
\def\svgwidth{3cm}
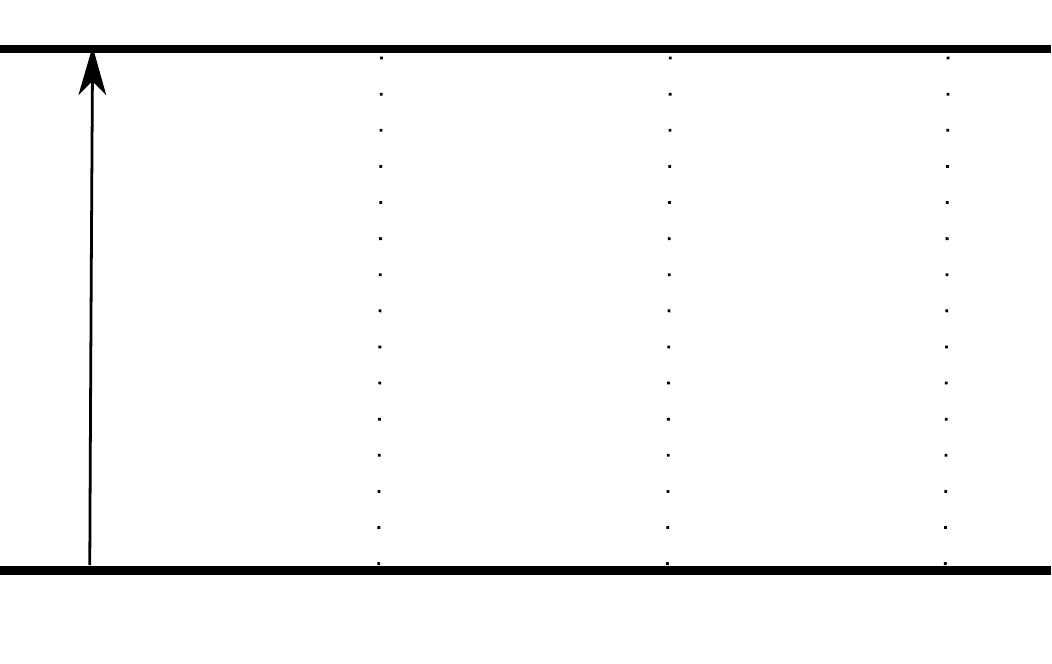
  \caption{$P_{1}$}  
\end{subfigure}%
\begin{subfigure}{.32\textwidth}\scriptsize
  \centering
\def\svgwidth{3cm}
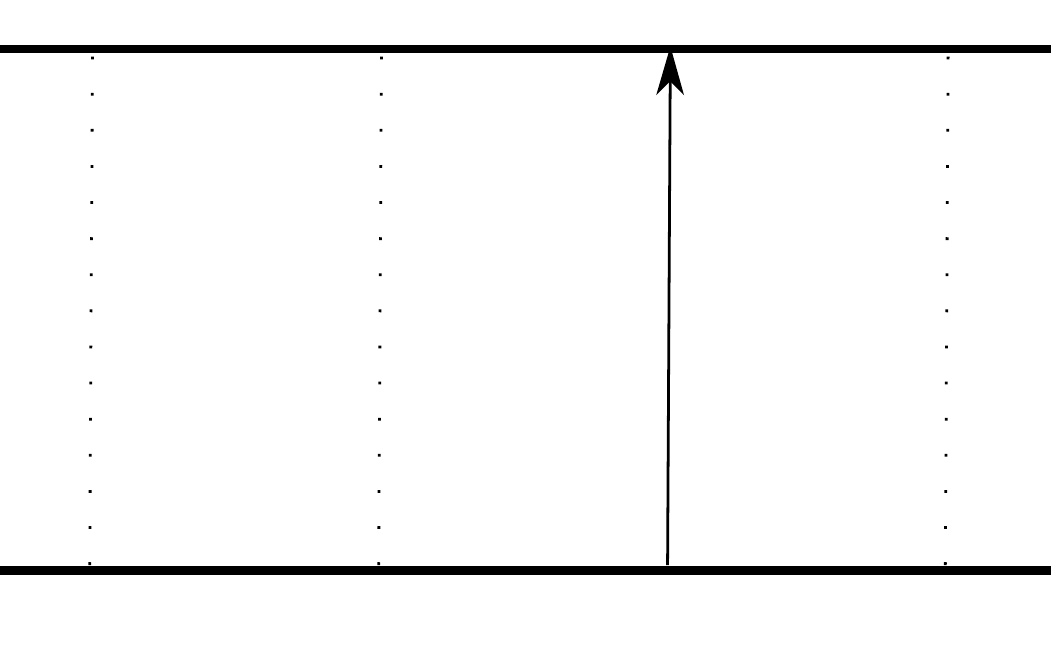
  \caption{$P_{2}$}  
\end{subfigure}
\begin{subfigure}{.32\textwidth}\scriptsize
  \centering
\def\svgwidth{3cm}
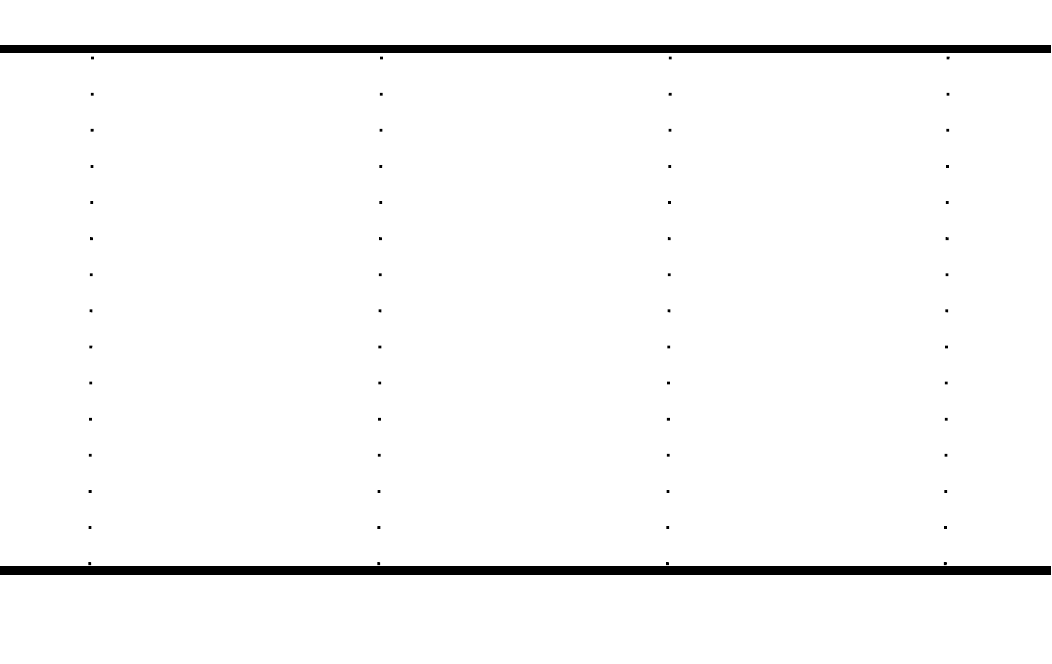
  \caption{$P_{3}$}  
\end{subfigure}
\caption{$P_{1}$, $P_{2}$, and $P_{3}$ are three examples of monotone bijections for $n=2$. For these examples, $P_{1}P_{2}=P_{3}$, and all other products are zero.} \label{Strands}
\end{figure}

The algebra has a simplified set of generators consisting the idempotents together with a set of elementary strand moves. Given a subset $S$ of $[2n]$ with $n$ elements, such that $i \in S$, $i+1 \notin S$, define $r_{i}(S)=S \cup \{i+1\} - \{i\}$. Let $R_{i}^{S}: S \to r_{i}(S)$ denote the bijection 

\[
R_{i}^{S}(j)=\begin{cases}
\begin{array}{ll}
i+1&\text{ if } j = i \\
j &\text{ if } j \ne i 
\end{array}
\end{cases}
\]

\noindent
For example, in Figure \ref{Strands}, $P_{2}=R_{1}^{\{1,3\}}$. Note that $R_{i}^{S}=\iota_{S}R_{i}^{S}\iota_{r_{i}(S)}$. Let $R_{i}^{S}=0$ when $r_{i}(S)$ is not well-defined (i.e. $i\notin S$ or $i+1 \in S$). The element $R_{i}$ is defined to be the sum over all the $R_{i}^{S}$:

\[ R_{i} = \sum_{S} R_{i}^{S}    \]

There is a similar definition for the elementary leftward-moving elements. Given an idempotent $\iota_{S}$ with $i+1 \in S$ and $i \notin S$, define $l_{i}(S)=S \cup \{i\}-\{i+1\}$. Let $L_{i}^{S}: S \to l_{i}(S)$ denote the bijection 

\[
L_{i}^{S}(j)=\begin{cases}
\begin{array}{ll}
i&\text{ if } j = i+1 \\
j &\text{ if } j \ne i+1 
\end{array}
\end{cases}
\]

\noindent
In Figure \ref{Strands}, $P_{1}=L_{3}^{\{1,4\}}$. In terms of idempotents, $L_{i}^{S}=\iota_{S}L_{i}^{S}\iota_{l_{i}(S)}$. Let $L_{i}^{S}=0$ when $l_{i}(S)$ is not well-defined. Then we define

\[ L_{i} = \sum_{S} L_{i}^{S}    \]

The strands algebra is generated by the idempotents together with $R_{i}$ and $L_{i}$, for $i=1,...,2n-1$. These elements satisfy five relations, some of which are already described, but we will list them all here for clarity.

\begin{itemize}

\item[R1)] The $u_{i}$ commute with each $R_i$, $L_i$ and all idempotents.

\item[R2)] For every $S$, $\iota_S\iota_S=\iota_S$. 

\item[R3)] For every $S\neq S'$, $\iota_S\iota_{S'}=0$.

\item[R4)] For every $S$, if $r_{i}(S)$ (resp. $l_{i}(S)$) is not well-defined, then $\iota_{S}R_{i}=0$ (resp. $\iota_{S}L_{i}=0$). Otherwise, $\iota_SR_i=R_i\iota_{r_i(S)}=\iota_SR_i\iota_{r_i(S)}$ and $\iota_SL_i=L_i\iota_{l_i(S)}=\iota_SL_i\iota_{l_i(S)}$.




\item[R5)] For $i=1,...,2n-1$, $R_{i}L_{i}=\iota_{i}u_{i}$ and $L_{i}R_{i} = \iota_{i+1}u_{i}$, where \[ \iota_{i} = \sum_{S \text{ containing } i} \iota_{S} \]

\end{itemize}


%
%
%
%
%
%
%

\noindent
Relations R2-R5 are illustrated in Figure \ref{algrelations}.

\begin{figure}[h!]
\begin{subfigure}{.45\textwidth}\scriptsize
 \centering
\def\svgwidth{6cm}
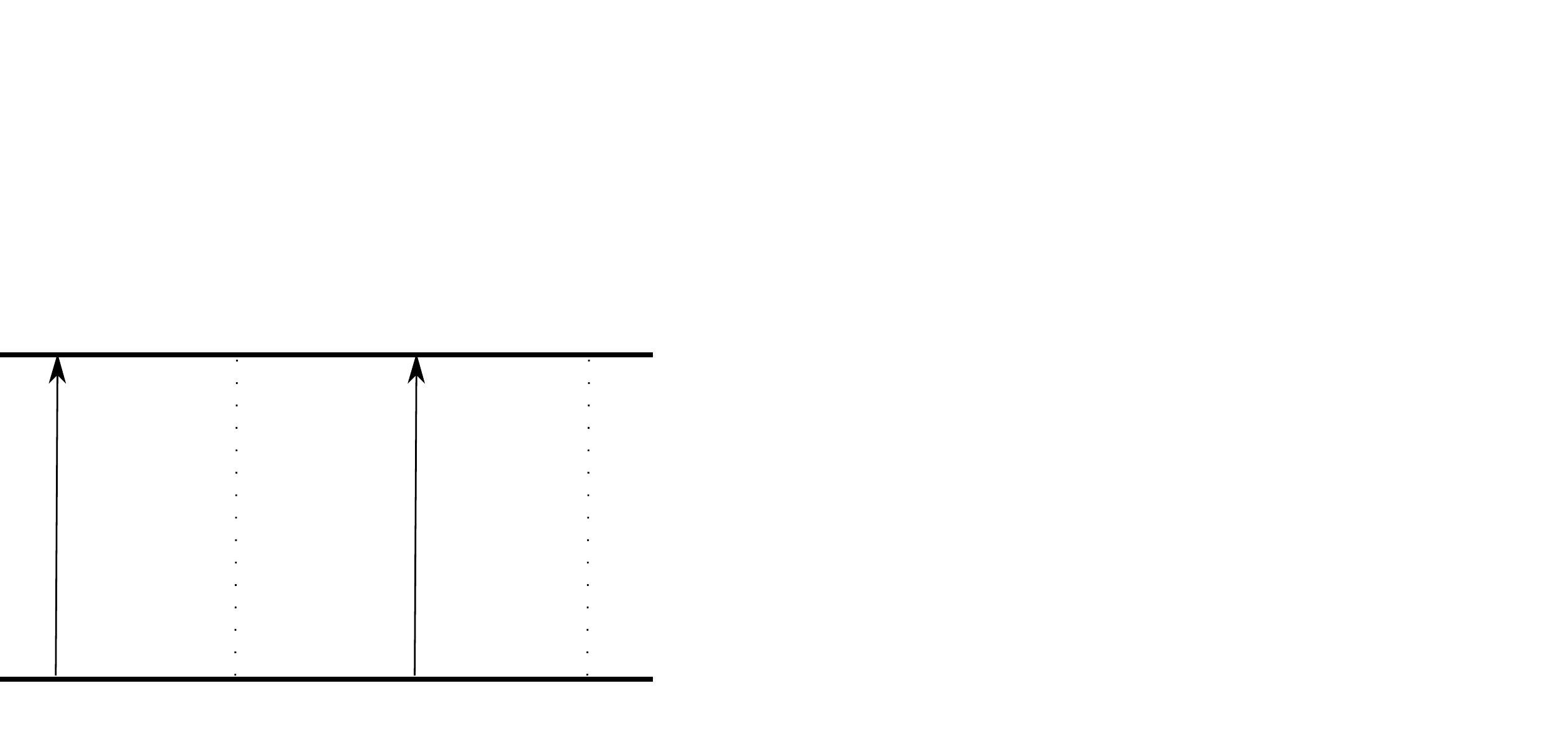
  \caption{An example of R2}  
  \vspace{12mm}
\end{subfigure}%
\begin{subfigure}{.45\textwidth}\scriptsize
  \centering
\def\svgwidth{3.3cm}
\hspace{20mm}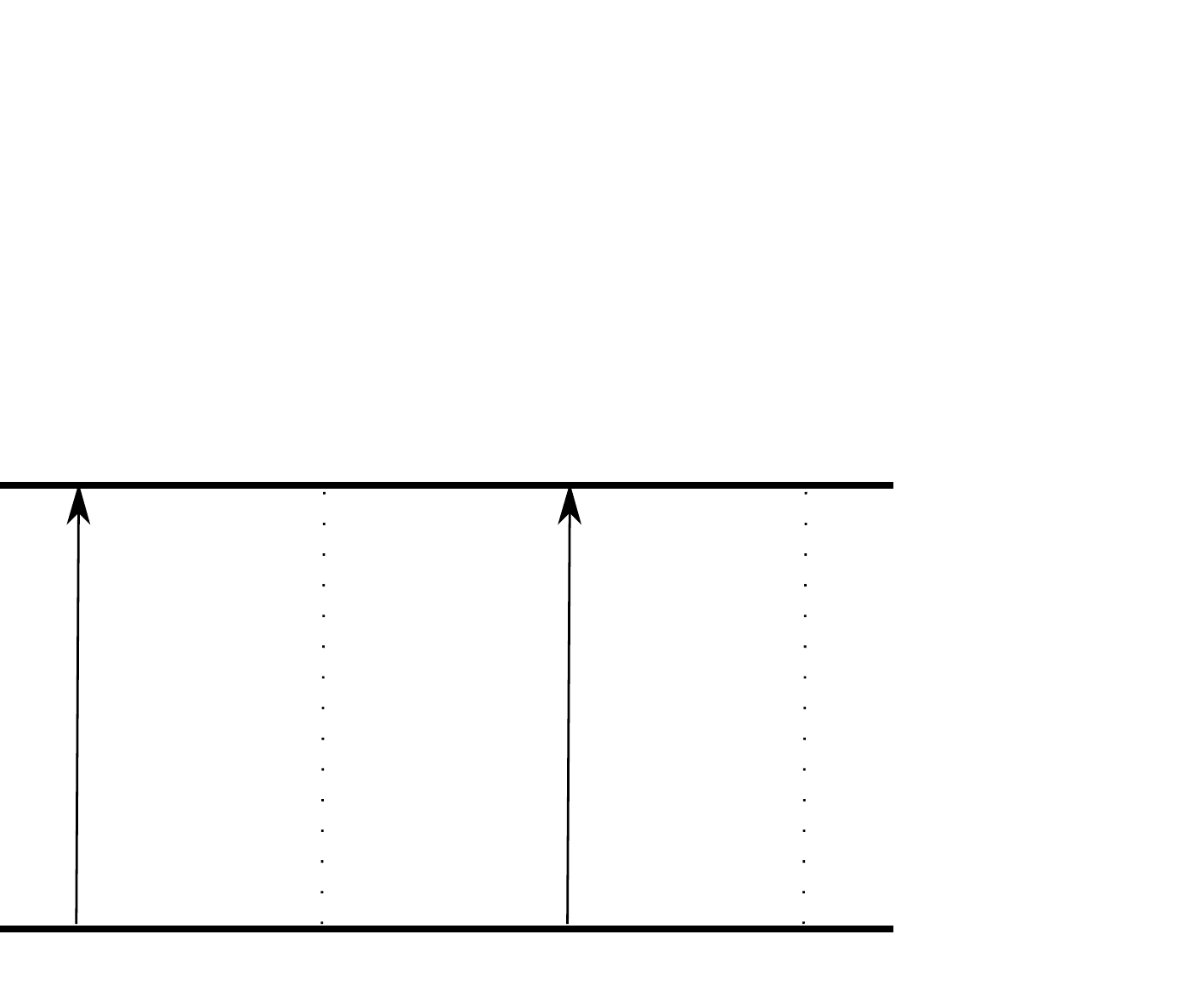
  \caption{An example of R3} 
  \vspace{12mm}
\end{subfigure}
\vspace{8mm}
\begin{subfigure}{.8\textwidth}\scriptsize
  \centering
\def\svgwidth{10cm}
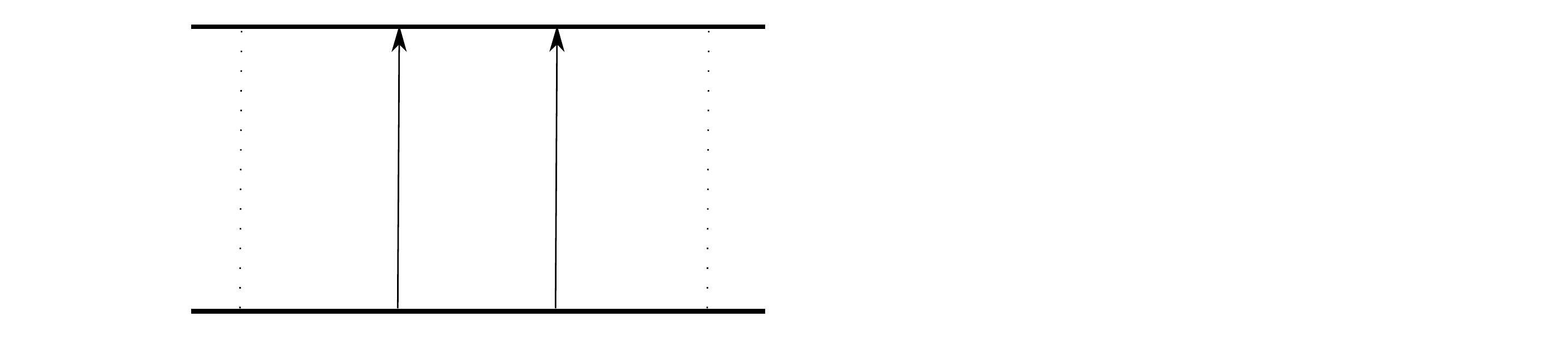
  \caption{An example of R4}  
\end{subfigure}
\vspace{4mm}
\begin{subfigure}{.6\textwidth}\scriptsize
  \centering
\def\svgwidth{8cm}
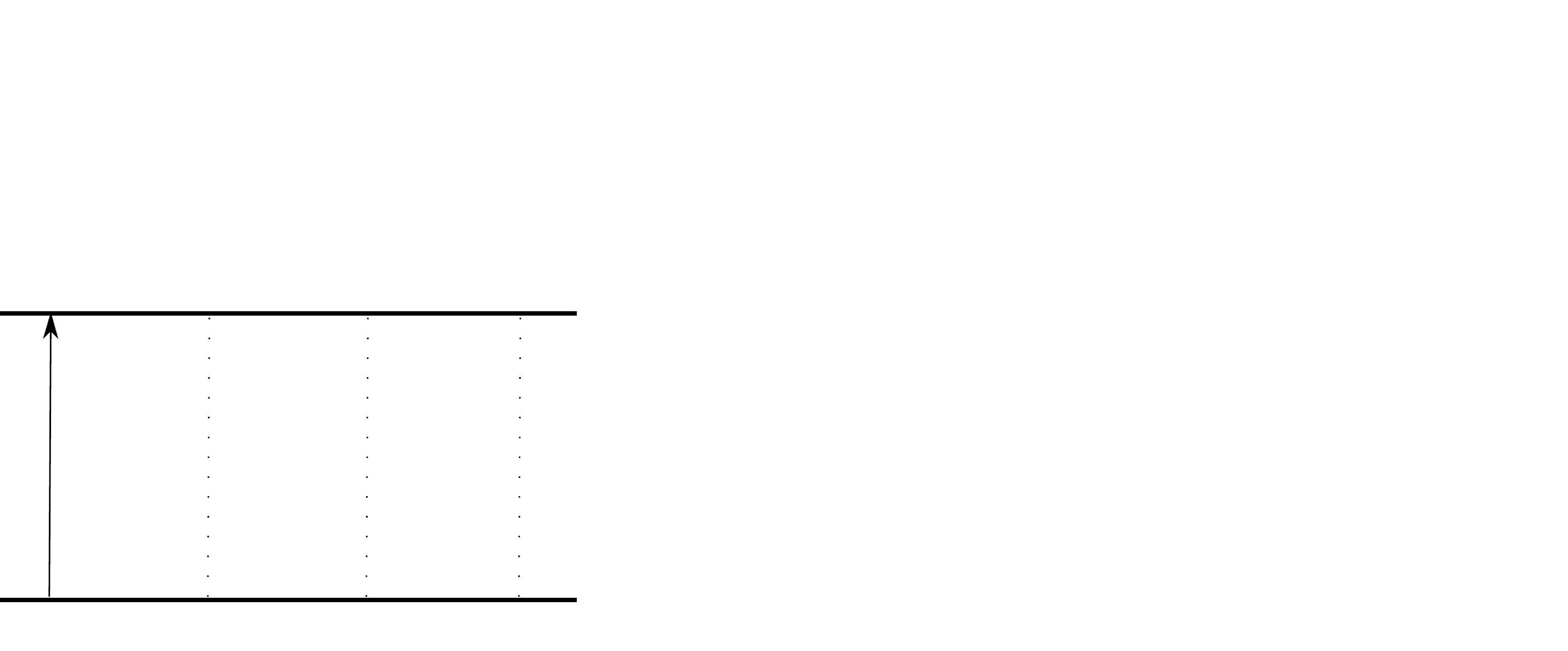
  \caption{An example of R5}  
\end{subfigure}
\caption{Some pictorial examples of the algebra relations} \label{algrelations}
\end{figure}

For each pair $j>i$, let \[\rho_{i,j}=R_iR_{i+1}...R_{j-1}\ \ \ \text{and}\ \ \ \delta_{j,i}=L_{j-1}L_{j-2}...L_{i}.\]
Thus, $R_{i}=\rho_{i,i+1}$ and $L_{i}=\delta_{i+1,i}$.

\subsubsection{The bimodule for a singular braid.} Let $S$ be a singular braid with $2n$ strands and $m$ edges. Corresponding to the top and bottom boundaries of $S$ we have the algebra $A_n$ defined in previous section. We will define a $\mathbb{Q}[U_1,...,U_m]$-module $\mathsf{M}[S]$ with a left and right action by $A_n$ such that the left action corresponds to the bottom boundary and the right action corresponds to the top boundary.  

For a singular braid $S$, a \emph{cycle} $Z$ in $S$ is a set of $n$  pairwise disjoint paths oriented from the bottom boundary to the top boundary. Let $b(Z), t(Z)\subset [2n]$ denote the set of vertices occupied at the bottom and top boundary, respectively. In addition, let $e_{b(i)}$ and $e_{t(i)}$ be the outgoing edge from $\{i\}\times\{0\}$ and incoming edge to $\{i\}\times\{1\}$, respectively. 

Similar to Section \ref{sub3.2}, for any cycle $Z$ and edge $e_i\subset Z$ we define $\mathcal{D}(Z,e_i)$ to be the set of disks $D$ in $\RR\times [0,1]$ such that $\bdy D\subset S\cup(\RR\times\{0\})\cup(\RR\times\{1\})$ decomposes as $\bdy D=\bdy_LD\cup \bdy_RD\cup\bdy_TD\cup \bdy_BD$ satisfying the following conditions: 
 \begin{enumerate}
\item $\bdy_BD$ (resp. $\bdy_TD$) is either a $4$-valent vertex of $S$ or a connected line segment in $\RR\times\{0\}$ (resp. $\RR\times\{1\}$). 
\item $\bdy_LD$ (resp. $\bdy_R(D)$) is an oriented path connecting left (resp. right) endpoint of $\bdy_BD$ to the left (resp. right) endpoint of $\bdy_TD$. Note that if $\bdy_BD$ or $\bdy_TD$ is a vertex then its both left and right endpoints are equal to the vertex. 
\item $D\cap Z=\bdy_LD$ and $\bdy_LD$ contains $e_i$. 
\end{enumerate}
As before, let $D(Z,e_i)$ denotes the smallest disk, i.e. the intersection of all disks in $\mathcal{D}(Z,e_i)$. Also, $U_i(Z)$ denotes the cycle obtained from $Z$ by replacing the $\bdy_LD(Z,e_i)\subset Z$ with $\bdy_RD(Z,e_i)$. As in Section \ref{sub3.2}, each disk $D(Z,e_i)$ has a coefficient $U(D(Z,e_i))$ which is a monomial in $\mathbb{Q}[U_1,...,U_m]$, where as before $U_i$ is the variable corresponding to the edge $e_i$. 

For each cycle $Z$, the bimodule $\mathsf{M}[S]$ contains an element denoted by $x_Z$. We define $\mathsf{M}[S]$  to be the $\mathbb{Q}[U_1,...,U_m]$-module generated by elements of the form $ax_Za'$ where $a,a'\in A_n$ such that $a\iota_{b(Z)}=a$ and $\iota_{t(Z)}a'=a'$, modulo the following conditions.

\begin{itemize}
\item[M1)] $x_Z=\iota_{b(Z)}x_Z\iota_{t(Z)}$.
\item[M2)] For every $4$-valent vertex $v$ of $S$ and any cycle $Z$, $U_{i(v)}U_{j(v)}x_Z=U_{k(v)}U_{l(v)}x_Z$.
\item[M3)] For every $4$-valent vertex $v$ of $S$ and any cycle $Z$, $(U_{i(v)}+U_{j(v)})x_Z=(U_{k(v)}+U_{l(v)})x_Z$.
\item[M4)] $u_{i}x_Z=U_{b(i)}x_Z$ for every cycle $Z$ and $1\le i\le 2n$.
\item[M5)] $x_Zu_i=U_{t(i)}x_Z$ for every cycle $Z$ and $1\le i\le 2n$.
\item[M6)] Consider an edge $e_i\subset Z$. If for $D=D(Z,e_i)$, both $\bdy_BD$ and $\bdy_TD$ are $4$-valent vertices in $S$, then $U_ix_Z=U(D)x_{U_i(Z)}$, and if $D=\emptyset$ then $U_ix_Z=0$.

\item[M7)] For every $D=D(Z,e_i)$, if $\bdy_TD$ is a vertex in $S$ and $\bdy_BD=[j,k]\times \{0\}\subset\RR\times\{0\}$, then
\[\delta_{k,j}x_Z=U(D)x_{U_i(Z)}\ \ \ \text{and}\ \ \ \rho_{j,k}(U(D)x_{U_i(Z)})=U_ix_Z.\]
\item[M8)] For every $D=D(Z,e_i)$, if $\bdy_BD$ is a vertex and $\bdy_TD=[j,k]\times\{1\}\subset\RR\times\{1\}$, then 
\[x_Z\rho_{j,k}=x_{U_i(Z)}\ \ \ \text{and}\ \ \ (U(D)x_{U_i(Z)})\delta_{k,j}=U_ix_Z,\]
\item[M9)] For every $D=D(Z,e_i)$ with $\bdy_BD=[j_0,k_0]\subset \RR\times\{0\}$ and $\bdy_TD=[j_1,k_1]\subset\RR\times\{1\}$
we have
\[\begin{split}
\delta_{k_0,j_0}x_{Z}=x_{U_i(Z)}\delta_{k_1,j_1},\ \ \ \ \ x_z\rho_{j_1,k_1}=\rho_{j_0,k_0}x_{U_i(Z)}\\
\text{and}\ \ U_ix_Z=\rho_{j_0,k_0}(U(D)x_{U_i(Z)})\delta_{k_1,j_1}.
\end{split}
\]
\end{itemize}

\begin{example} Let $S_{id}$ be the identity braid on $2n$ strands as in Figure \ref{sid}. In this case, the cycles are in one-to-one correspondence with the idempotents in $A_n$ and $\mathsf{M}[S_{id}]\cong A_n$ as $A_n$-bimodules.

\begin{figure}[ht]
\centering
\def\svgwidth{10cm} \scriptsize
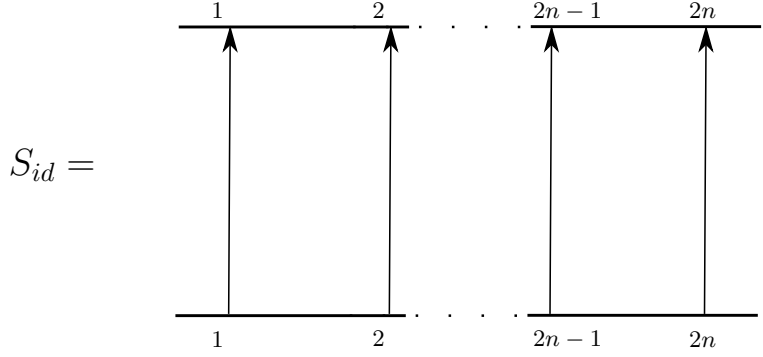
\caption{The trivial tangle $S_{id}$}
\end{figure}\label{sid}

%
%
%
\end{example}

\begin{example}\label{ex:sing}
 Let $\mathsf{X}_i$ denote the elementary singular braid on $2n$ strands with a singularization between strands $i$ and $i+1$. Let $e_1$ and $e_2$ (resp. $e_3$ and $e_4$) denote the left and right incoming (resp. outgoing) edges to (resp. from) the only $4$-valent vertex of $\mathsf{X}_i$, respectively (see Figure \ref{xi}.  For each subset $I\subset\{1,2,3,4\}$, let $\mathcal{Z}I$ denote the set of cycles that locally consist of the edges labelled with elements in $I$.

\begin{figure}[ht]
\centering
\def\svgwidth{12cm} \scriptsize
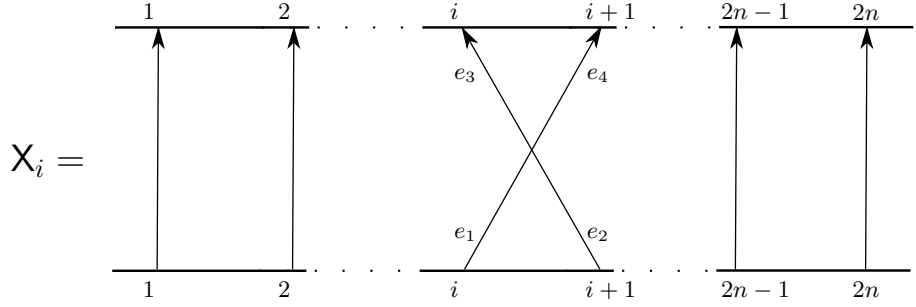
\caption{The elementary singular braid $\mathsf{X}_{i}$}\label{xi}
\end{figure}
 
 Similar to identity bimodule, the bimodule $\mathsf{M}(\mathsf{X}_i)$ is the set of elements $A_n.x_Z.A_n$ modulo the following relations: 
 
 \begin{enumerate}
 \item $x_{Z}=\iota_{b(Z)}x_{Z}\iota_{t(Z)}$,
 \item For $j<i$ or $j>i+1$, $u_jx_Z=x_Zu_j$, 
 \item For $j<i-1$ or $j>i+1$, if $j+1\in b(Z)$, then $L_{j}x_{l_j(Z)}=x_{Z}L_{j}$,
 \item For $j<i-1$ or $j>i+1$,, if $j\in b(Z)$, then $R_{j}x_{r_{j}(Z)}=x_{Z}R_{j}$,
 \item $(u_i+u_{i+1})x_Z=x_Z(u_i+u_{i+1})$,
 \item $(u_iu_{i+1})x_{Z}=x_Z(u_iu_{i+1})$,
 \item For $Z\in\mathcal{Z}\emptyset$, if $i-1\in b(Z)$, then 
 \[L_{i-1}x_{Z}=x_{Z(1,3)}L_{i-1} \ \ \ \ \text{and}\ \ \ \ x_ZR_{i-1}=R_{i-1}x_{Z(1,3)},\]
  where $Z(1,3)\in\mathcal{Z}\{1,3\}$ denotes the cycle obtained from $Z$ by replacing the $(i-1)$-th strand with local cycle $e_1e_3$.
 \item For $Z\in\mathcal{Z}\{1,3\}$, \[L_ix_Z=x_{Z(2,3)}\ \ \ \text{and}\ \ \ x_{Z}R_i=x_{Z(1,4)},\]
 where $Z(2,3)=(Z\setminus e_1)\cup e_2$ and $Z(1,4)=(Z\setminus e_3)\cup e_4$.
 \item For $Z\in\mathcal{Z}\{1,4\}$, $L_ix_Z=x_{Z(2,4)}$ where $Z(2,4)=(Z\setminus e_1)\cup e_2$.
 \item For $Z\in\mathcal{Z}\{2,3\}$, $x_ZR_i=x_{Z(2,4)}$ where $Z(2,4)=(Z\setminus e_3)\cup e_4$.
 \item For each $Z\in\mathcal{Z}\emptyset$, if $i+2\in b(Z)$ then, 
 \[ \begin{split}
 &x_ZL_{i+1}L_i=L_{i+1}x_{Z}(2,3), \ \ \ \ x_ZU_3L_{i+1}=L_{i+1}x_{Z(2,4)}\\
 &R_iR_{i+1}x_Z =U_3x_{Z(1,4)}R_{i+1}\ \ \ \ R_{i+1}U_3x_Z=x_{Z(2,4)}R_{i+1}.
 \end{split}
 \]
where $Z(i,j)$ denotes the cycle obtained from $Z$ by replacing the $(i+2)$-th strand with local cycle $e_ie_j$. 
 \end{enumerate}
 \end{example}

\subsubsection{The Cup and Cap Modules}
Let $\mycap_i$ denote the elementary tangle on $2n$ strands with a singular cap between the strands $i$ and $i+1$ as in Figure \ref{elcapdiagram}.  This tangle consists of $2n$ edges, we label them by their bottom vertex. As before, a cycle $Z$ in $\mycap_i$ is a union of $n$ disjoint edges such that it contains exactly one of $e_i$ or $e_{i+1}$. Associated to $\mycap_i$ we define a bimodule $\mathsf{M}[\mycap_i]$, generated by the cycles so that it is a left module over $A_n$ and right module over $A_{n-1}$. Specifically, it  is the set of elements $A_n.x_Z.A_{n-1}$ for all cycles $Z$ subject to the following relations.\begin{enumerate}
\item $x_Z=\iota_{b(Z)}x_Z\iota_{t(Z)}$
\item For all $1\le j<i-1$ we have $L_{j}x_Z=x_{U_j(Z)}L_j$ and $x_ZR_j=R_jx_{U_j(Z)}$.
\item For all  $i+2\le j< 2n$ we have $L_jx_{Z}=x_{U_j(Z)}L_{j-2}$ and $R_jx_{U_j(Z)}=x_ZL_{j-2}$.
\item For every $Z$, $L_{i-1}x_Z=L_{i+1}x_Z=x_ZL_{i-1}=0$ and $R_{i-1}x_Z=R_{i+1}x_Z=x_{Z}R_{i-1}=0$.
\item If $e_i\subset Z$, then  $L_ix_Z=x_{U_i(Z)}$.
\end{enumerate}

\begin{figure}[ht]
\centering
\def\svgwidth{12cm} \scriptsize
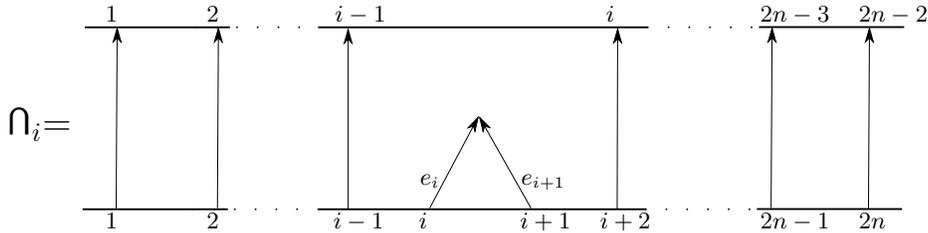
\caption{The elementary singular cap diagram} \label{elcapdiagram}
\end{figure}

\noindent

Similarly, $\mycup_i$ denotes the elementary tangle on $2n$ strands with a singular cup between the strands $i$ and $i+1$ as in Figure \ref{elcupdiagram}. Associated with this we define a bimodule $\mathsf{M}[\mycup_i]$ which is a left $A_{n-1}$-module and right $A_n$-module generated by the cycles. Similarly, if we label the edges by their top boundary, a cycle in $Z$ is a union of $n$ pairwise disjoint edges containing exactly one of $e_i$ or $e_{i+1}$, and $x_Z$ denotes the corresponding generator. Then, $\mathsf{M}[\mycup_i]$ is the set of elements $A_{n-1}.x_Z.A_n$ for all cycles $Z$ subject to the following relations. 

\begin{enumerate}
\item $x_Z=\iota_{b(Z)}x_Z\iota_{t(Z)}$
\item For all $1\le j<i-1$ we have $L_{j}x_Z=x_{U_j(Z)}L_j$ and $x_ZR_j=R_jx_{U_j(Z)}$.
\item For all  $i+2\le j< 2n$ we have $L_jx_{Z}=x_{U_j(Z)}L_{j+2}$ and $R_jx_{U_j(Z)}=x_ZL_{j+2}$.
\item For every $Z$, $L_{i-1}x_Z=x_ZL_{i-1}=x_ZL_{i+1}=0$ and $R_{i-1}x_Z=x_ZR_{i-1}=x_ZR_{i+1}=0$.
\item If $e_i\subset Z$, then $x_ZR_i=x_{U_i(Z)}$.
\end{enumerate}

\begin{figure}[ht]
\centering
\def\svgwidth{12cm} \scriptsize
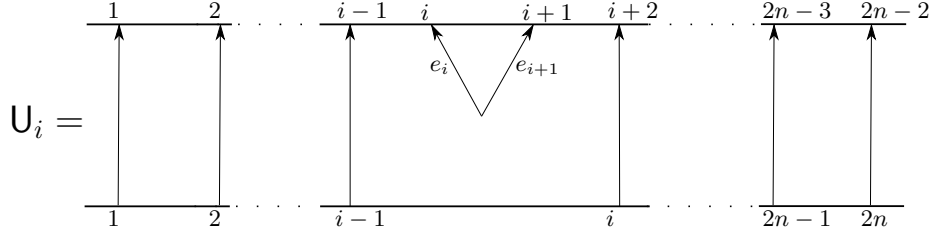
\caption{The elementary singular cup diagram} \label{elcupdiagram}
\end{figure}

Finally, by tensoring these bimodules we define $\mathsf{M}[S_{cup}]$ and $\mathsf{M}[S_{cap}]$ for the tangle consisting of $n$ singular cups (Figure \ref{S-cups}) and caps (Figure \ref{S-caps}), respectively.

Specifically, $\mathsf{M}[S_{cup}]$ is a right $A_n$-module defined as follows. Label the edges such that $e_i$ denotes the edge occupying the $i$-th vertex on top boundary. The generators are given by local cycles $Z$ with $n$ strands that contain exactly one of the edges $e_{2j}$ and $e_{2j-1}$ for each $j\in\{1,...,n\}$. As before, $x_Z$ denotes the generator corresponding to $Z$ and $t(Z)$ denotes the set of occupied strands at the top boundary. The module $\mathsf{M}[S_{cup}]$ is the set of elements $x_Z.A_n$ for every cycle $Z$ modulo the following relations: 
\begin{enumerate}
\item[C1)] $x_Z=x_Z.\iota_{t(Z)}$.
\item[C2)] If $2j-1\in t(Z)$ for some $1\le j\le n$, then $x_ZR_{2j-1}=x_{U_j(Z)}$.
\item[C3)] If $2j\in t(Z)$ for some $1\le j\le n$, then $x_ZR_{2j}=0$.
\end{enumerate}

Similarly, for $S_{cap}$, $\mathsf{M}[S_{cap}]$ is a left $A_n$-module consisting of the elements $A_n.x_Z$ for all cycles $Z$ modulo the following relations: 
\begin{enumerate}
\item[C'1)] $x_Z=\iota_{b(Z)}x_Z$
\item[C'2)] If $2j-1\in b(Z)$ for some $1\le j\le n$, then $L_{2j-1}x_Z=x_{U_j(Z)}$.
\item[C'3)] If $2j\in b(Z)$ for some $1\le j\le n$, then $L_{2j}x_{Z}=0$.
\end{enumerate}

\begin{theorem}
Let $S_1$ and $S_2$ be singular braids with $2n$ strands. Then, the $\QQ[U_1,...,U_m]$-modules
\[\mathsf{M}(S_1)\otimes_{A_n}\mathsf{M}(S_2)\cong\mathsf{M}(S_1\circ S_2)\]
and this isomorphism preserves the left and right $A_n$-actions. Here, $m$ denotes the number of edges in $S=S_1\circ S_2$.

\end{theorem}

\begin{proof}
Let 
\[h:\mathsf{M}[S_1\circ S_2]\to \mathsf{M}[S_1]\otimes_{A_n}\mathsf{M}[S_2]\]
be the map defined by $h(x_Z)=x_{Z_1}\otimes x_{Z_2}$ where $Z_i=Z\cap S_i$. First, we show that $h$ extends to a well-defined $\QQ[U_1,...,U_m]$-homomorphism that preserves the right and left $A_n$ actions.

Suppose $S_1\subset \RR\times [0,1]$ and $S_2\subset \RR\times [1,2]$ such that $S_1\cap (\RR\times\{1\})=S_2\cap (\RR\times\{1\})$. Let $Z$ be a cycle for $S$ and $Z_i=Z\cap S_i$.  Consider an edge $e_i\subset Z$. If $D=D(Z,e_i)$ is a disk away from $\RR\times\{1\}$, then it is clear from the definition that \[h(U_ix_Z)=U_ix_{Z_1}\otimes x_{Z_2}.\]

Suppose $e_i\subset S_1$, and $D\cap (\RR\times\{1\})=[j,k]\times\{1\}$.  Then, $D=D_1\cup D_2$ where $D_1=D(Z_1,e_i)$ and $D_2\subset \RR\times [1,2]$ is the smallest disk with left boundary on $Z_2$ and bottom boundary equal $[j,k]\times \{1\}$. Depending on the type of $D$, one of the relations M6, M7, M8 or M9, implies that $U_ix_Z=a_0(U(D)x_{U_i(Z)})a_2$
where $a_0,a_2$ are generators of $A_n$.  Since, $D_1=D(Z_1,e_i)$ we have
\[U_ix_{Z_1}=a_0 U(D_1)x_{Z_1'}\delta_{k,j}.\]
On the other hand, it is not hard to show that \[D_2=D(Z_2,e_{b(j)})\cup D(u_j(Z_2),e_{b(j+1)})\cup...\cup D(u_{k-1}(...(u_j(Z_2))),e_{b(k-1)}),\]
and thus $\delta_{k,j}x_{Z_2}=\frac{U(D)}{U(D_1)}x_{Z_2'}a_2$ where $U_i(Z)=Z_1'\cup Z_2'$. Therefore, \[U_ih(x_{Z})=U_ix_{Z_1}\otimes x_{Z_2}=a_0U(D)x_{Z_1'}\otimes x_{Z_2'}a_2=a_0U(D)h(x_{U_i(Z)})a_2.\]

Finally, if $e_i\cap (\RR\times\{1\})\neq \emptyset$, then $U_ix_{Z_1}=x_{Z_1}R_iL_i$. On the other hand, $D=D_1\cup D_2$ where $D_1=D(Z_1,e_{t(i)})$ and $D_2=D(Z_2,e_{b(i)})$. Thus, 
\[\begin{split}
U_ix_{Z_1}\otimes x_{Z_2}=x_{Z_1}R_i\otimes L_ix_{Z_2}&=(U(D_1)a_0x_{Z_1'})\otimes (U(D_2)x_{Z_2'}a_2)\\
&=U(D)a_0x_{Z_1'}\otimes x_{Z_2'}a_2.
\end{split}
\]
where $U_ix_Z=U(D)a_0x_{U_i(Z)}a_2$ and $U_i(Z)=Z_1'\cup Z_2'$.

The proof for the rest of relations is similar. 

Next, we show that $h$ is surjective. It is enough to prove that any element of the form $x_{Z_1}\otimes ax_{Z_2}$ is in the image of $h$. Here, $Z_1$ and $Z_2$ are cycles for $S_1$ and $S_2$, respectively, and $a:t(Z_1)\to b(Z_2)$ is a monotone bijection.  The relations M7, M8, M9 imply that $x_{Z_1}\otimes ax_{Z_2}=q.a_1x_{Z_1'}\otimes x_{Z_2'}a_2$, where $q\in\QQ[U_1,...,U_m]$, $a_1,a_2\in A_n$ and $t(Z_1')=b(Z_2')$. Therefore, $h$ is surjective. It is not hard to see that $h$ is injective, and thus is an isomorphism.  
\end{proof}

\begin{theorem}
Let $S$ be a singular braid with $2n$ strands and $m$ edges. Then,
\[\mathscr{M}(S)\cong\mathsf{M}(S_{cup})\otimes_{A_n}\mathsf{M}(S)\otimes_{A_n}\mathsf{M}(S_{cap})\]
as $\QQ[U_1,...,U_m]$-modules.
\end{theorem}

\begin{proof}
By definition $\mathsf{M}(S_{cup})\otimes_{A_n}\mathsf{M}(S)\otimes_{A_n}\mathsf{M}(S_{cap})$ is generated by the elements of the form $x_{Z_0}a_0\otimes x_{Z}\otimes a_1x_{Z_1}$ where $Z, Z_0$ and $Z_1$ are cycles for $S$, $S_{cup}$ and $S_{cap}$, respectively. Moreover, $a_0:t(Z_0)\to b(Z)$ and $a_{1}:t(Z)\to b(Z_2)$ are the corresponding monotone bijections. Factoring $a_0$ and $a_1$ as products of $R_j$'s and $L_j$'s, the relations imply that $x_{Z_0}a_1\otimes x_{Z}\otimes a_1x_{Z_1}=q.x_{Z'}$ where $q\in\QQ[U_1,...,U_m]$ and $Z'$ is a cycle in the plat closure of $S$. Because of the relations M2 and M3, we are done.
\end{proof}

\subsection{The Edge Bimodule Homomorphisms and the Total Complex}

In this Section, we define $A_n$-bimodule homomorphisms 
\[
\begin{split}
d^-:\mathsf{M}[\mathsf{X}_i]\to \mathsf{M}[S_{id}]\\
d^+:\mathsf{M}[S_{id}]\to\mathsf{M}[\mathsf{X}_i].
\end{split}
\]

The cycles in $S_{id}$ are in one-to-one correspondence with subsets  of $[2n]$ with $n$ elements. For each such subset $S$, denote the corresponding cycle and generator by $Z_{S}$ and $x_{S}$, respectively. Furthermore, for each cycle $Z=Z_{S}$ (i.e., $S=b(Z)=t(Z)$) that $\{i,i+1\}\not\subset S$ there is unique cycle $Z^z$ in $\mathsf{X}_i$ for which $b(Z^z)=t(Z^z)=S$. Conversely, for any cycle $Z$ in $\mathsf{X}_i$ with $b(Z)=t(Z)$, there is a unique cycle $Z^u=Z_{b(Z)}$ in $S_{id}$. So we define:

\[
d^-(x_Z)=\begin{cases}
\begin{array}{lcl}
x_{Z^u}&&\text{if}\ \{i+1\}\cap (b(Z)\cup t(Z))=\emptyset\\
x_{b(Z)}.L_i&&\text{if}\ i+1\in b(Z)\ \text{and}\ i\in t(Z)\\
R_i.x_{t(Z)}&&\text{if}\ i\in b(Z)\ \text{and}\ i+1\in t(Z)\\
U_i.x_{Z^u}&&\text{if}\ \{i+1\}\subset b(Z)\cap t(Z)
\end{array}
\end{cases}
\]

\begin{lemma}
$d^-$ is well-defined $A_n$-bimodule homomorphism.
\end{lemma}
\begin{proof}
The proof is straightforward by checking the local relations.
\end{proof}
Next, we define 
\[
d^+(x_{Z})=
\begin{cases}
\begin{array}{lcl}
x_{Z^z}u_i-u_{i+1}x_{Z^z}&&\text{if}\ i+1\notin b(Z)=t(Z)\\
x_{Z^z}-\epsilon R_{i+1} x_{e_{i+1}(Z)^z}L_{i+1}&&\text{if}\ \{i,i+1\}\cap b(Z)=\{i+1\}\\
-\epsilon R_{i+1}x_{e_{i+1}(Z)^z}L_{i+1}&&\text{if}\ \{i,i+1\}\subset b(Z)
\end{array}
\end{cases}
\]
where $\epsilon=0$ if $i+2\in b(Z)$, otherwise $\epsilon=1$.

\begin{lemma}\label{lem:comp}
For any $i$ and $x\in\mathsf{M}[\mathsf{X}_i]$,
\[d^+_{i}\circ d^-_i=U_i-U_{i+1}\quad\text{and}\quad d^-_{i}\circ d^+_i(x)=xu_{i}-u_{i+1}x\]
\end{lemma}
\begin{proof}
It is straightforward. 
\end{proof}

We now have the necessary tools to define the total complex $C_{1 \pm 1}(D)$. Let $\mathfrak{C}=\{c_{1},...,c_{N} \}$ denote the crossings in $D$, and for each $v \in \{0,1\}^{N}$, let $D_{v}$ denote the corresponding complete singular resolution.

The complex $(C_{1 \pm 1}(D), d_{0})$ is given by 
\[ \bigoplus_{v \in \{0,1\}^{N}} C_{1 \pm 1}(D_{v}) \]

\noindent
For each $u \lessdot v$, the edge map $d_{u,v}$ from $C_{1 \pm 1}(D_{u})$ to $C_{1 \pm 1}(D_{v})$ for a positive and negative crossing, respectively, is given by the natural extensions of $d^{+}$ and $d^{-}$, respectively. The total edge map is given by 
\[ d_{1} = \sum_{u \lessdot v} (-1)^{\epsilon_{u,v}} d_{u,v} \]

The fact that $(d_{0}+d_{1})^{2}=0$ follows from the local definition of $d_{1}$. In particular, if $D = D_{1} \circ ... \circ D_{j}$ where each $D_{i}$ has a single crossing, we could rewrite the complex $\mathscr{M}(D)$ as a tensor product 
\[ \mathscr{M}(D) = \mathsf{M}(S_{cup}) \otimes_{A_n} \mathsf{M}(D_{1}) \otimes_{A_n} ... \otimes_{A_n} \mathsf{M}(D_{j}) \otimes_{A_n} \mathsf{M}(S_{cap})   \]

\noindent
where each $\mathsf{M}(D_{i})$ is the mapping cone on $d^{+}$ or $d^{-}$, depending on the sign of the crossing. The total complex is 
\[  C_{1 \pm 1}(D) = \mathscr{M}(D) \otimes \mathsf{K}(D) \]

\noindent
Since each tensorand satisfies $d^{2}=0$, the total complex does as well.

\subsection{A Bigrading for Complete Resolutions}\label{gradingsect1}

In this section we will define a bigrading on $C_{1 \pm 1}(S)$ for a complete resolution $S$. It is denoted ($\mathfrak{gr}_{q}, \mathfrak{gr}_{a})$. We will refer to $\mathfrak{gr}_{q}$ as the quantum grading and $\mathfrak{gr}_{a}$ as the horizontal grading in analogy with the corresponding gradings on Khovanov-Rozansky $\mathfrak{sl}_{2}$ homology. 

Let $R\{i,j\}$ denote a copy of the ground ring with $1 \in R$ in $\delta$-grading $i$ and horizontal grading $j$. The variables $U_{i}$ in $R$ have bigrading $(-2, 0)$.

We will start with the quantum grading. Since $R$ has an internal grading, it suffices to define the quantum grading on each generator $x_{Z}$. Given a cycle $Z$, let $T_{1}(Z)$ denote the number of 4-valent vertices at which $Z$ has edges $e_{1}$ and $e_{3}$, let $T_{2}(Z)$ be the number of 4-valent vertices at which $Z$ has edges $e_{2}$ and $e_{4}$, and $E(Z)$ the number of 4-valent vertices at which $Z$ contains none of the four edges (see Figure \ref{labeled4v} for the edge labelings).

\begin{figure}[ht]
\centering
\def\svgwidth{3cm} 
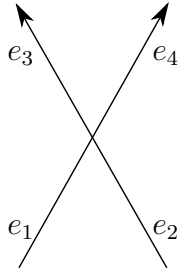
\caption{A 4-valent vertex with labeled edges}\label{labeled4v}
\end{figure}

At each $w_{i}^{+}$ and $w_{i}^{-}$, $Z$ contains either the left or the right edge. Define $w_{i}(Z)$ by 

\[
w_{i}(Z)=\begin{cases}
\begin{array}{ll}
1 &\text{ if } Z \text{ contains the left edge at both } w_{i}^{+} \text{ and } w_{i}^{-} \\
-1 &\text{ if } Z \text{ contains the right edge at both } w_{i}^{+} \text{ and } w_{i}^{-} \\
0&\text{otherwise}
\end{array}
\end{cases}
\]

\noindent
Let $w(Z) = \sum_{i}w_{i}(Z)$.

We define $x_{Z}$ to lie in quantum grading $T_{1}(Z)-T_{2}(Z)+E(Z)+w(Z)$. For $e_{i}$ not in $Z$, multiplication by $U_{i}$ decreases $\mathfrak{gr}_{q}$ by 2 by construction. Assume $e_{i}$ is in $Z$. The vertices $v_{t}D(Z, e_{i})$ and  $v_{b}D(Z, e_{i})$ contribute 2 more to $T_{1}(Z)-T_{2}(Z)+E(Z)+w(Z)$ than to $T_{1}(U_{i}(Z))-T_{2}(U_{i}(Z))+E(U_{i}(Z))+w(U_{i}(Z))$. Further, the total contributions of vertices on $\partial(D(Z, e_{i}))\setminus\{ v_{t}D(Z, e_{i}), v_{b}D(Z, e_{i})  \}$ to $T_{1}(U_{i}(Z))-T_{2}(U_{i}(Z))+E(U_{i}(Z))+w(U_{i}(Z))$ and $T_{1}(Z)-T_{2}(Z)+E(Z)+w(Z)$ differ by $\mathfrak{gr}_{q}(U(D,e_i))$.  Thus, the quantum grading is well-defined on $\mathscr{M}(S)$.

%
%
%
%

For the horizontal grading, we place the whole complex $\mathscr{M}(S)$ in grading $0$. The bigrading extends to $C_{1 \pm 1}(S)$ by placing gradings on the matrix factorizations in Section \ref{defcomplex} as follows: 

\[\xymatrix{R\{0,-2\}\ar@<1ex>[r]^{ L_{w_i^\pm}}&R\{0,0\}\ar@<1ex>[l]^{L'_{w_i^{\pm}}}}\]

Let $d_{0+}$ denote the component of the differential $d_0$ which is homogeneous of degree $2$ with respect to the horizontal grading and $d_{0-}$ the component which is homogeneous of degree $-2$. Both $d_{0+}$ and $d_{0-}$ are homogeneous of degree $-2$ with respect to the quantum grading.

\section{Computing the homology for complete resolutions}

In this section we will compute the homology of $C_{1\pm 1}(S)$ for any singular braid $S$. The homology will first be computed as a bigraded vector space over $\mathbb{Q}$. This requires two main tools: a filtration induced by basepoints in the diagram, and a formula called the composition product. After computing the bigraded vector space, we will describe it as an $R$-module by showing that $H_{1+1}(S)$ satisfies the MOY relations of Murakami, Ohtsuki, and Yamada \cite{murakami1998homfly}.

\subsection{The Basepoint Filtration}\label{bpf}

Let $A_{1},...,A_{r}$ denote the set of bounded regions in $\RR^{2}$ in the complement of $S$, and for each $A_{i}$, let $p_{i}$ denote a basepoint in $A_{i}$. It turns out that each $p_{i}$ gives a filtration on $M(S)$, which we will define in terms of a grading $\gr_{i}$. 

Let $Z$ be a cycle in $\pc(S)$. Instead of viewing $Z$ as strands in $\pc(S)$, we can move to the singular closure viewpoint of Section \ref{closuresection}, so that the cycle consists of $n$ disjoint circles in $\x(S)$. If $Z$ is a cycle in $\pc(S)$, let $\overline{Z}$ denote the corresponding cycle in $\x(S)$.

The grading $\gr_{i}$ is defined as follows. Let $\mathcal{C}_{Z}$ denote the unique 2-chain in $\RR^{2}$ with $\partial \mathcal{C}_{Z} = \overline{Z}$. Recall that as a $\QQ$-module, 

\[M(S)=\bigoplus _{Z\in c(S)}R_{Z}\]

\noindent
On each summand $R_{Z}$, define $\gr_i$ to be the coefficient of $\mathcal{C}_{Z}$ at $p_{i}$. Since the $R$-action only moves cycles to the right, it is non-increasing with respect to each $\gr_{i}$. Thus, each $p_{i}$ induces an ascending filtration on $M(S)$
\[  ... \subseteq \mathcal{F}_{k-1}(M(S)) \subseteq \mathcal{F}_{k}(M(S)) \subseteq \mathcal{F}_{k+1}(M(S)) \subseteq ...  \]

\noindent
where $\mathcal{F}_{k}(M(S))$ consists of elements of $M(S)$ with $\gr_{i} \le k$. 

The module $\mathscr{M}(S) = M(S) \otimes R/L$ inherits the quotient filtration from $M(S)$. We make $C_{1 \pm 1}(S)=\mathscr{M}(S) \otimes \mathsf{K}(S)$ a filtered complex by putting the generators of $\mathsf{K}(S)$ in the same filtration level. Thus, the complex $C_{1 \pm 1}(S)$ comes equipped with a $\ZZ^{r}$ filtration by the basepoints $p_{1},...,p_{r}$, which we will call the basepoint filtration.



Note that when the filtration is bounded, this means that if $f$ induces an isomorphism on the $E_{k}$ pages for some $k$, then it induces an isomorphism on the $E_{\infty}$ pages, i.e. the total homology.

\subsection{The $\mathfrak{sl}_{1}$ Homology of Singular Braids} 

Let $S$ be a singular braid, and let $\mathsf{b}(S)$ denote the standard (clockwise oriented) braid closure of $S$. Each vertex in $\mathsf{b}(S)$ is either 2-valent or 4-valent and has the same number of incoming and outgoing edges. We take the convention that there is a bivalent vertex on each arc corresponding to the braid closure, so that if the diagram were cut at these vertices, the resulting diagram would be $S$. These vertices will be called \emph{closing vertices}.

In \cite{KhovanovRozansky08:MatrixFactorizations}, Khovanov and Rozansky defined a family of knot invariants called $\mathfrak{sl}_{n}$ homology. Each $\mathfrak{sl}_{n}$ homology categorifies the $\mathfrak{sl}_{n}$ polynomial, and the $\mathfrak{sl}_{2}$ homology agrees with Khovanov homology over $\QQ$. In this section, we will describe the $\mathfrak{sl}_{1}$ homology of a singular braid.

The $\mathfrak{sl}_{1}$ complex $C_{1}(S)$ is defined over the same ground ring $R$. The complex comes with two gradings, the quantum grading $gr_{q}$ and the horizontal grading $gr_{h}$. As with the $C_{1 \pm 1}(S)$ complex, each variable $U_{i}$ has quantum grading $-2$ and horizontal grading $0$, and we write $R\{i,j\}$ for a copy of the ground ring with $1 \in R$ in bigrading $(i,j)$.

To each 2-valent vertex $v$, let $e_{i}$ be the incoming edge at $v$ and $e_{j}$ the outgoing edge. We define a matrix factorization corresponding to $v$ by

\vspace{-2mm}
\begin{figure}[!h]
\centering
\begin{tikzpicture}
  \node at (-5,0) {\large$C_{1}(v)=$};
  \matrix (m) [matrix of math nodes,row sep=5em,column sep=8em,minimum width=2em] {
     R\{0, -2\} & R\{0, 0\} \\};
  \path[-stealth]
    (m-1-1) edge [bend left=15] node [above] {$U_{i(v)}-U_{j(v)}$} (m-1-2)
    (m-1-2) edge [bend left=15] node [below] {$U_{i(v)}+U_{j(v)}$} (m-1-1);
\end{tikzpicture}
\end{figure}

\noindent
Now suppose $v$ is a 4-valent vertex. In sections \ref{sub3.2} and \ref{defcomplex}, we define linear polynomials $L_{v}$, $L'_{v}$, and a quadratic polynomial $Q_{v}$. For 4-valent $v$, we define the matrix factorization $C_{1}(v)$ by the Koszul complex in Figure \ref{sl14}. The total $\mathfrak{sl}_{1}$ complex $C_{1}(S)$ is given by 

\[ C_{1}(S) = \bigotimes_{v \in S} C_{1}(v)    \]

One can see by inspection that at each vertex (both 2-valent and 4-valent), $d^{2}= \sum_{e_{i} \in In(v)} U_{i}^{2} - \sum_{e_{j} \in Out(v)} U_{j}^{2}$, so each $C_{1}(v)$ is a matrix factorization. However, the sum of these potentials over all vertices in $S$ is zero, so $d^{2}=0$ on $C_{1}(S)$ making it a true chain complex.

\begin{figure}[!h]
\centering
\begin{tikzpicture}
  \node at (-5.5,0) {\large$C_{1}(v)=$};
  \matrix (m) [matrix of math nodes,row sep=7em,column sep=8em,minimum width=2em] {
     R\{-1,-4\} & R\{-1,-2\} \\
     R\{1,-2\} & R\{1,0\} \\};
  \path[-stealth]
    (m-1-1) edge [bend right = 15] node [left] {$Q_{v}$} (m-2-1)
            edge [bend left = 15] node [above] {$L_{v}$} (m-1-2)
    (m-2-1) edge [bend left = 15] node [above] {$-L_{v}$} (m-2-2)
            edge [bend right = 15] node [right] {$2$} (m-1-1)    
    (m-1-2) edge [bend right=15] node [left] {$Q_{v}$} (m-2-2)
            edge [bend left=15] node [below] {$L'_{v}$} (m-1-1)
    (m-2-2) edge [bend left=15] node [below]  {$-L'_{v}$} (m-2-1)
            edge [bend right=15] node [right] {$2$} (m-1-2);        
\end{tikzpicture}
\caption{The $\mathfrak{sl}_{1}$ complex at a 4-valent vertex} \label{sl14}
\end{figure}

Let $d_{0+}$ denote the component of the differential which has horizontal grading $2$ and $d_{0-}$ the component which has horizontal grading $-2$. Both $d_{0+}$ and $d_{0-}$ are homogeneous of degree $-2$ with respect to the quantum grading.

The $\mathfrak{sl}_{1}$ homology of $S$ is defined to be
\[ H_{1}(S) = H_{*}(C_{1}(S), d_{0+}+d_{0-}) \]

\noindent
The homology only has one grading - the quantum grading - as $d_{0+}+d_{0-}$ is not homogeneous with respect to the horizontal grading. However, if we take homology first with respect to $d_{0+}$, then with respect to $d_{0-}^{*}$, this $E_{2}$ page turns out to be isomorphic to the total homology, but now the differentials are homogeneous so the homology admits a horizontal grading. We will denote this homology by $H_{1}^{\pm}(S)$:
\[ H_{1}^{\pm}(S) = H_{*}(H_{*}(C_{1}(S), d_{0+}), d^{*}_{0-}) \]

\begin{example}

Let $\mathcal{U}$ be the 1-strand diagram for the unknot with a single bivalent vertex, and a single edge $e_{1}$. Then $C_{1}(\mathcal{U})$ is given by 

\begin{figure}[!h]
\centering
\begin{tikzpicture}
  \node at (-6,0) {\large$C_{1}(S)=$};
  \matrix (m) [matrix of math nodes,row sep=5em,column sep=8em,minimum width=2em] {
     \QQ[U_{1}]\{0, -2\} & \QQ[U_{1}]\{0, 0\} \\};
  \path[-stealth]
    (m-1-2) edge node [below] {$2U_{1}$} (m-1-1);
\end{tikzpicture}
\end{figure}

\noindent
Thus, $H^{\pm}_{1}(\mathcal{U}) \cong \QQ\{0,-2\}$.

\end{example}
\begin{remark}
The $\mathfrak{sl}_{1}$ homology of a singular diagram is invariant under addition or removal of bivalent vertices, provided each component of $S$ has at least one vertex.
\end{remark}

The $\mathfrak{sl}_{1}$ homology of singular diagrams in general is described in the following lemma.

\begin{lemma}
Let $S$ be a $k$-strand singular braid. Then 

\[
H^{\pm}_{1}(S)=\begin{cases}
\begin{array}{ll}
\QQ\{0,-2k \} &\text{ if } \mathsf{b}(S) \text{ is the }k \text{ component unlink} \\
0&\text{otherwise}
\end{array}
\end{cases}
\]

\end{lemma}

\begin{proof}

The computation for the $k$ component unlink follows from the computation for the unknot, together with the fact that disjoint union corresponds to tensor product. In other words, if $S=S_{1} \bigsqcup S_{2}$, then $C_{1}(S)=C_{1}(S_{1}) \otimes C_{1}(S_{2})$.

If $\mathsf{b}(S)$ is not the $k$ component unlink, then it has at least one 4-valent vertex. In Figure \ref{sl14}, the arrows with coefficient $2$ are isomorphisms because we are working over $\QQ$. Thus, $C_{1}(S)$ is acyclic.

\end{proof}

\subsection{The Filtered Homology for Singular Braids}\label{bpsect}

In this section we will compute the homology of the associated graded object for $C_{1 \pm 1}(S)$ with respect to the basepoint filtration. The complex splits over cycles $Z$, with the corresponding complex given by 

\[
R_{Z}\otimes\big{(}
\xymatrix{R\ar@<1ex>[r]^{L_{w_1^\pm}}&R\ar@<1ex>[l]^{L'_{w_1^{\pm}}}}\big{)}\otimes\big{(}
\xymatrix{R\ar@<1ex>[r]^{L_{w_2^\pm}}&R\ar@<1ex>[l]^{L'_{w_2^{\pm}}}}\big{)}\otimes...\otimes\big{(}
\xymatrix{R\ar@<1ex>[r]^{L_{w_n^\pm}}&R\ar@<1ex>[l]^{L'_{w_n^{\pm}}}}\big{)}.
\]

Since $Z$ contains one of the edges at each $w_{k}^{+}$ and $w_{k}^{-}$, we can rewrite this as

\begin{equation}\label{sl1cycle}
R_{Z}\otimes\big{(}
\xymatrix{R\ar@<1ex>[r]^{U_{i_{1}}-U_{j_{1}}}&R\ar@<1ex>[l]^{U_{i_{1}}+U_{j_{1}}}}\big{)}\otimes\big{(}
\xymatrix{R\ar@<1ex>[r]^{U_{i_{2}}-U_{j_{2}}}&R\ar@<1ex>[l]^{U_{i_{2}}+U_{j_{2}}}}\big{)}\otimes...\otimes\big{(}
\xymatrix{R\ar@<1ex>[r]^{U_{i_{n}}-U_{j_{n}}}&R\ar@<1ex>[l]^{U_{i_{n}}+U_{j_{n}}}}\big{)}
\end{equation}

\noindent
where $e_{i_{k}}$ is the edge adjacent to $w_{k}^{+}$ which is not in $Z$, and $e_{j_{k}}$ is the edge adjacent to $w_{k}^{-}$ which is not in $Z$.

We can view $S-Z$ as an open singular braid on $n$ strands. Let $\mathsf{b}(S-Z)$ denote the braid closure of $S-Z$ so that each $w_{i}^{+}=w_{i}^{-}$ is now a bivalent vertex, which we will denote $w_{i}$. 

\begin{lemma}

Up to an overall grading shift, the complex in Equation \ref{sl1cycle} is quasi-isomorphic to the $\mathfrak{sl}_{1}$ complex of $\mathsf{b}(S-Z)$.

\end{lemma}

\begin{proof}

This is equivalent to showing that the set of relations $\{L_{v},Q_{v}:v \ne w_{i}\}$ form a regular sequence in $R$. To see that they do in fact form a regular sequence, we can start at the bottom of the braid, and use these relations to perform variable exclusion in the polynomial ring on the incoming edges at each vertex. 

At each bivalent vertex, we get the complex in Figure \ref{bivcomplex}, where $e_{i}$ is the incoming edge and $e_{j}$ is the outgoing edge. We will use the relation $U_{i}-U_{j}$ to substitute for $U_{i}$.

\begin{figure}[!h]
\centering
\begin{tikzpicture}
  \matrix (m) [matrix of math nodes,row sep=5em,column sep=8em,minimum width=2em] {
     R\{0, -2\} & R\{0, 0\} \\};
  \path[-stealth]
    (m-1-1) edge [bend left=15] node [above] {$U_{i(v)}-U_{j(v)}$} (m-1-2)
    (m-1-2) edge [bend left=15] node [below] {$U_{i(v)}+U_{j(v)}$} (m-1-1);
\end{tikzpicture}
\caption{}\label{bivcomplex}
\end{figure}

At each 4-valent vertex, we get the complex in Figure \ref{4comp}. We can use the relation $L_{v}=U_{i(v)}+U_{j(v)}-U_{k(v)}-U_{l(v)}$ to substitute for $U_{i(v)}$. After making this substitution, $Q_{v}$ becomes equal to $-U^{j(v)}_{2}+U_{j(v)}(U_{k(v)}+U_{l(v)})-U_{k(v)}U_{l(v)}$, so we can use this relation to exclude $U_{j}$. 

\begin{figure}[!h]
\centering
\begin{tikzpicture}
  \matrix (m) [matrix of math nodes,row sep=7em,column sep=8em,minimum width=2em] {
     R\{-1,-4\} & R\{-1,-2\} \\
     R\{1,-2\} & R\{1,0\} \\};
  \path[-stealth]
    (m-1-1) edge [bend right = 15] node [left] {$Q_{v}$} (m-2-1)
            edge [bend left = 15] node [above] {$L_{v}$} (m-1-2)
    (m-2-1) edge [bend left = 15] node [above] {$-L_{v}$} (m-2-2)
            edge [bend right = 15] node [right] {$2$} (m-1-1)    
    (m-1-2) edge [bend right=15] node [left] {$Q_{v}$} (m-2-2)
            edge [bend left=15] node [below] {$L'_{v}$} (m-1-1)
    (m-2-2) edge [bend left=15] node [below]  {$-L'_{v}$} (m-2-1)
            edge [bend right=15] node [right] {$2$} (m-1-2);        
\end{tikzpicture}
\caption{}\label{4comp}
\end{figure}

Since at each vertex the relations can be written in terms of exclusions on the incoming edges, the relations form a regular sequence in $R$. After canceling the differentials corresponding to these relations in $C_{1}(\mathsf{b}(S-Z))$, we get exactly the complex in \ref{sl1cycle}, so the two complexes are quasi-isomorphic. To see that the bigradings match up (up to an overall grading shift), note that the gradings assigned to the Koszul complex on a bivalent vertex in $C_{1}(\mathsf{b}(S-Z))$ are the same as the gradings assigned to the closing-off Koszul complex in $C_{1 \pm 1}(S)$.

\end{proof}

The element $x_{Z}$ in $R_{Z}$ has bigrading $(T_{1}(Z)-T_{2}(Z)+E(Z)+w(Z),0)$, while the element at the bottom of the Koszul complex on $L_{v}$ and $Q_{v}$ has bigrading $(E(Z), 0)$. Thus, we can write the basepoint-filtered homology of $C_{1 \pm 1}(S)$ as follows:

\begin{lemma} \label{compproduct}

Let $d_{0}^{f}$ denote the component of the differential on $C_{1 \pm 1}(S)$ which preserves the basepoint filtration. Then 

\[  H_{*}((C_{1 \pm 1}(S), d_{0+}^{f}),d_{0-}^{f}) \cong \bigoplus_{Z}H^{\pm}_{1}(S-Z)\{T_{1}(Z)-T_{2}(Z)+w(Z), 0 \}   \]

\end{lemma}

There is a spectral sequence from $H_{*}((C_{1 \pm 1}(S), d_{0+}^{f}),d_{0-}^{f})$ to the total homology $H_{1+1}(S)$. However, $H^{\pm}_{1}(S-Z)$ always lies in a single horizontal grading, so $H_{*}((C_{1 \pm 1}(S), d_{0+}^{f}),d_{0-}^{f})$ does as well. Thus, there are no higher differentials, and with respect to the quantum grading, we have the following isomorphism:

\[ H_{1+1}(S) \cong  \bigoplus_{Z}H^{\pm}_{1}(S-Z)\{T_{1}(Z)-T_{2}(Z)+w(Z)\}  \]

\begin{corollary} \label{KhovCor}

With respect to the quantum grading, $H_{1+1}(S) \cong Kh(\mathsf{sm}(\pc(S)))$ as graded $\QQ$-vector spaces.

\end{corollary}

\begin{proof}

As before, if $Z$ is a cycle in $\pc(S)$, let $\overline{Z}$ be the corresponding cycle in $\mathsf{x}(S)$. Then $T_{1}(Z)-T_{2}(Z)+w(Z) = T_{1}(\overline{Z}) - T_{2}(\overline{Z})$, and we get

\[ H_{1+1}(S) \cong  \bigoplus_{Z}H^{\pm}_{1}(\mathsf{x}(S)-\overline{Z})\{T_{1}(\overline{Z}) - T_{2}(\overline{Z})\}  \]

\noindent
This is precisely the composition product formula for $m=n=1$, which proves that this sum is isomorphic to 
$H_{2}(\mathsf{x}(S))$. (See \cite{dowlin2015knot}, equation (5).) 

The $\mathfrak{sl}_{2}$ homology of a singular diagram is known to be isomorphic to the Khovanov homology of the smoothing of that diagram \cite{hughes2014note}:
\[ H_{2}(\mathsf{x}(S)) \cong Kh(\sm(\x(S)) = Kh(\sm(\pc(S)) \]

\noindent
The second equality is due to the fact that $\sm(\x(S))$ and $\sm(\pc(S))$ are the same diagram.

\end{proof}

\subsection{MOY Relations} Khovanov-Rozansky homology satisfies a set of graph relations known as the MOY relations. In this section, we will show that $H_{1+1}(S)$ satisfies the analogous relations to $\mathfrak{sl}_{2}$ homology.

First, we recall these relations. Let $\mathcal{A} = \mathbb{Q}[U]/(U^2)$, the Khovanov homology of the unknot, and $S$ and $S'$ be singular braids with even strands. 
\begin{itemize}
\item (MOY 0) If $\pc(S')$ is the disjoint union of $\pc(S)$ with the plat closure of the trivial braid with two strands (i.e. unknot) then
\[H_{1+1}(S')\cong H_{1+1}(S)\otimes \mathcal{A}\{1\}.\]

\item (MOY \rom{1}) If $\pc(S')$ is obtained is obtained from $\pc(S)$ by one of the local moves as in Figure \ref{fig:MI} then
\[H_{1+1}(S')\cong H_{1+1}(S)\]

\begin{figure}[ht]
\centering
\def\svgwidth{8cm}
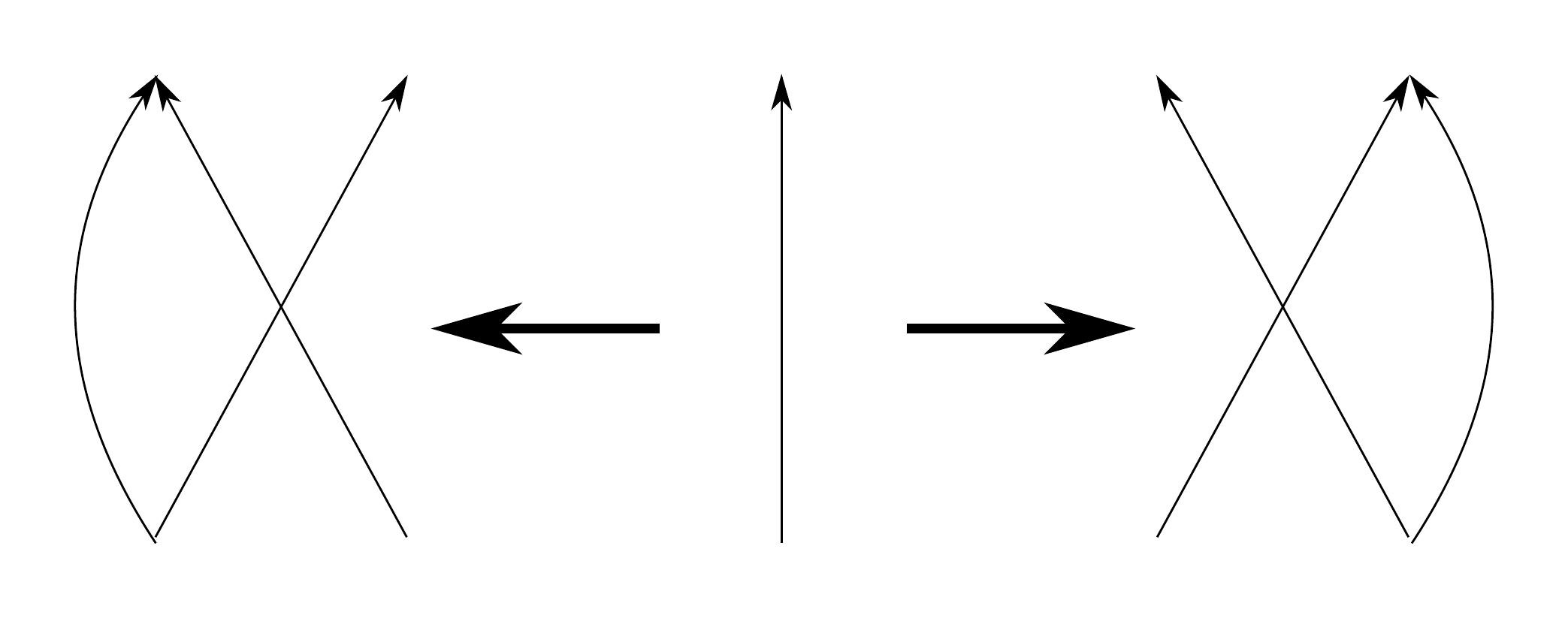
\caption{MOY \rom{1} moves.}\label{fig:MI}
\end{figure}

\item (MOY \rom{2}) If $\pc(S')$ is obtained from $\pc(S)$ by one of the local moves as in Figure \ref{fig:MII} then 
\[H_{1+1}(S')\cong H_{1+1}(S)\otimes \mathcal{A}.\]

\begin{figure}[ht]
\centering
\def\svgwidth{8.5cm}
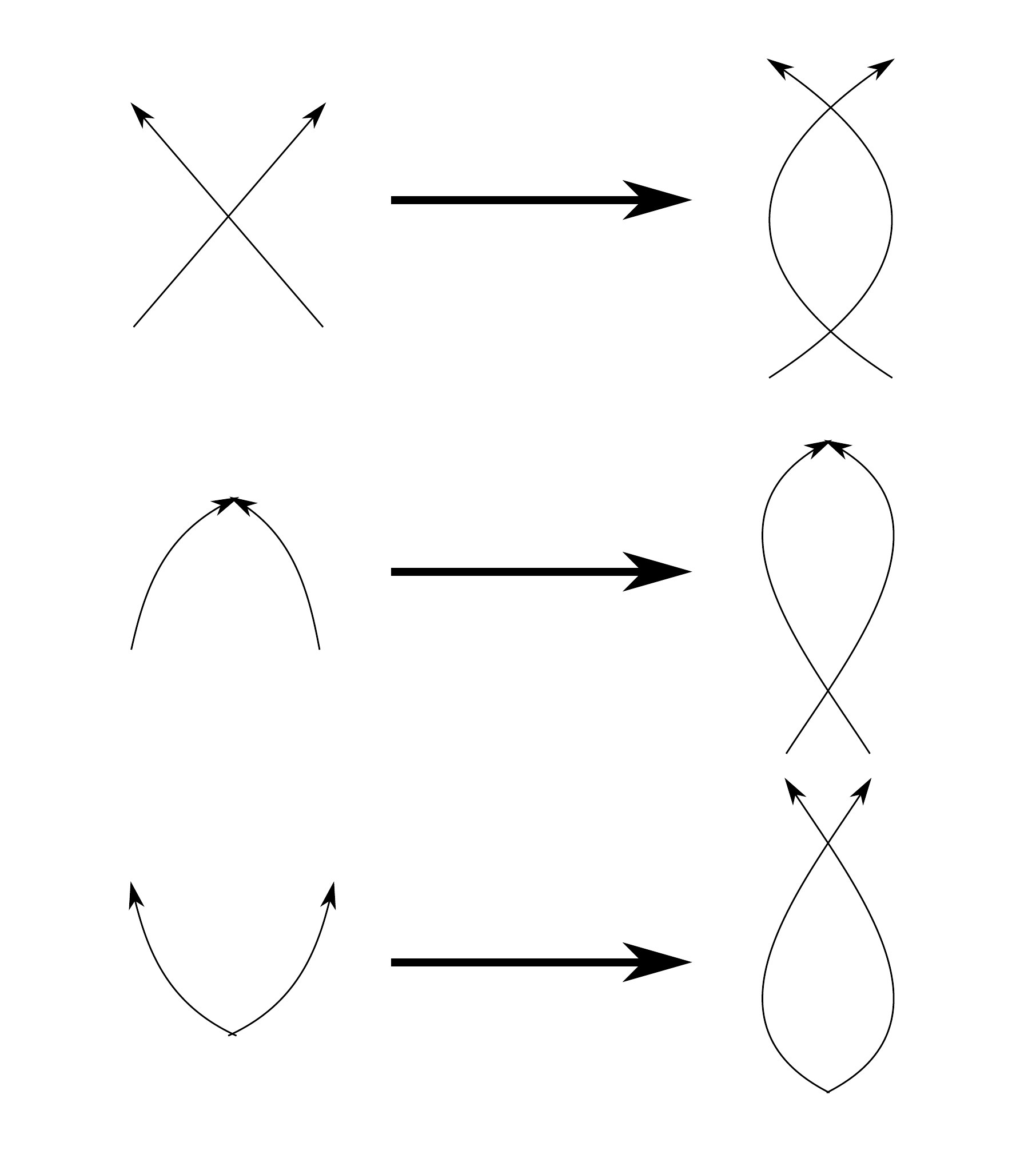
\caption{MOY \rom{2} moves.}\label{fig:MII}
\end{figure}

\item (MOY \rom{3}) If $\pc(S')$ is obtained from $\pc(S)$ by one  of the local move as in Figure \ref{fig:MIII}, then 
\[H_{1+1}(S')\cong H_{1+1}(S)\]

\end{itemize}

\begin{figure}[ht]
\centering
\def\svgwidth{15.5cm}
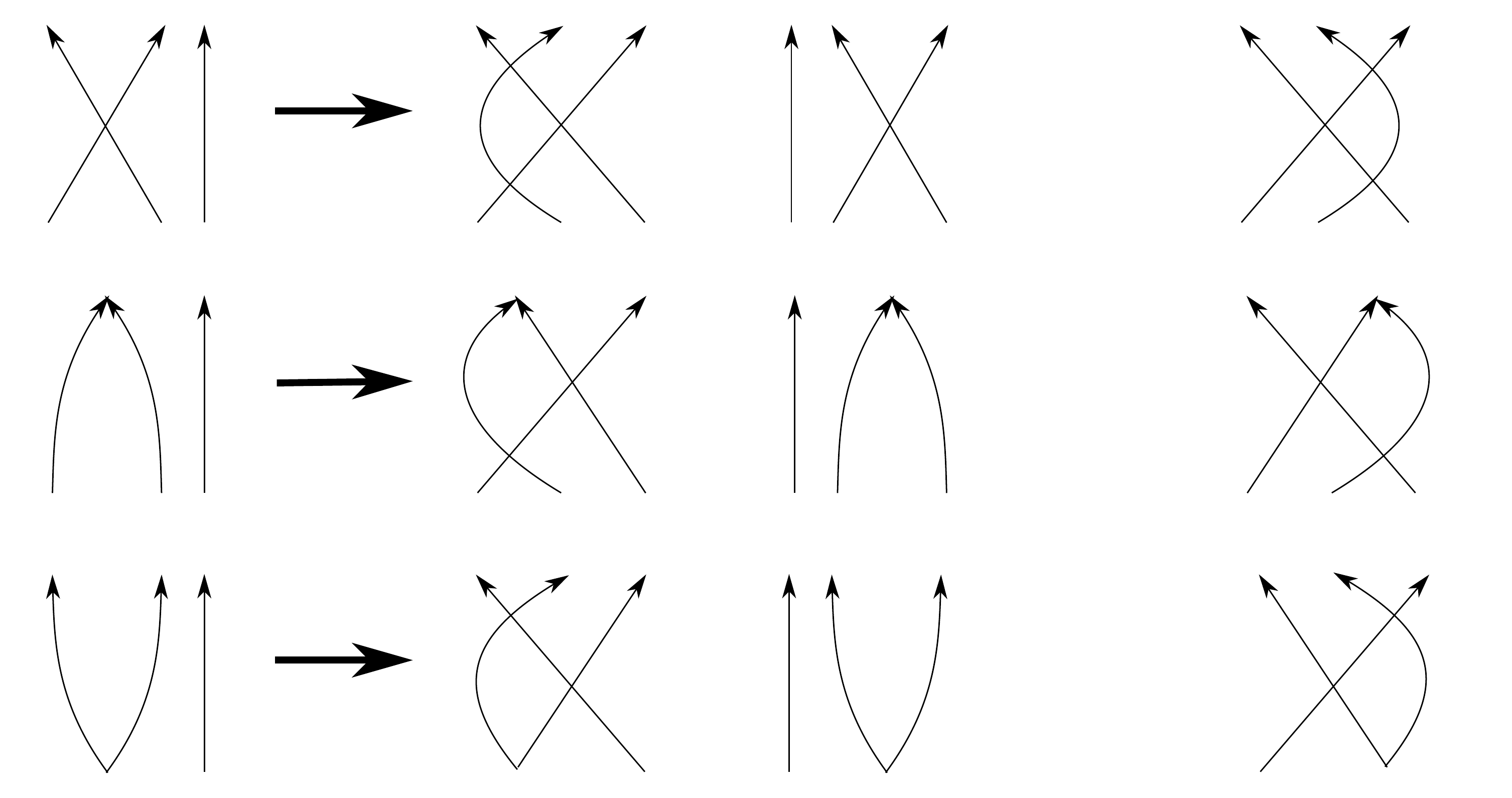
\caption{MOY \rom{3} moves.}\label{fig:MIII}
\end{figure}



\noindent
 Note that these isomorphisms follow from Corollary \ref{KhovCor} as graded vector spaces. However, in this section we will make the isomorphisms explicit on the chain level and show that the maps commute with the $U_{i}$-actions.


\subsubsection{MOY 0} Suppose $S$ and $S'$ are singular braids with even number of strands. Then, it is clear from the definition that 
\[C_{1\pm1}(S|S')\simeq C_{1\pm1}(S)\otimes C_{1\pm1}(S'),\]
where $S|S'$ is the disjoint union of $S$ and $S'$. In particular, $C_{1+1}(S|\mathcal{U}_2)\simeq C_{1+1}(S)\otimes C_{1+1}(\mathcal{U}_2)$ for any braid $S$ with an even number of strands. Let $e_1$ and $e_2$ denote the left and right strands of $\mathcal{U}_2$, respectively.  The diagram $\mathcal{U}_2$ has two cycles:  $Z_{1}=\{e_{1}\}$ and $Z_{2}=\{e_{2}\}$. Let $x_{1}$ and $x_{2}$ be the corresponding generators in $\mathscr{M}(\mathcal{U}_2)$. The relations are given by $U_{1}x_{1}=x_{2}$, $U_{2}x_{2}=0$. Thus, $\mathscr{M}(\mathcal{U}_2) \cong \QQ[U_{1}, U_{2}]/(U_{1}U_{2})$. So, the complex $C_{1\pm1}(\mathcal{U}_2)$ becomes
\[  
\xymatrix@C+2pc{\QQ[U_{1},U_{2}]/(U_{1}U_{2})&\QQ[U_{1},U_{2}]/(U_{1}U_{2})\ar[l]^{2U_{1}+2U_{2}}}. \]

Since $x_{Z_{1}}$ lies in quantum grading $1$, the generator of $\mathscr{M}(\mathcal{U}_2)$ also lies in quantum grading $1$ and the homology of this complex is isomorphic to $\mathcal{A}\{1\}$, with $U_{1}$ acting as $U$ and $U_{2}$ acting as $-U$. Thus, 
\[ H_{1+1}(S')=H_{1+1}(S|\mathcal{U}_2) \cong H_{1+1}(S) \otimes \mathcal{A}\{1\} \]

\subsubsection{MOY \rom{2}} Let $\mathsf{X}_i$ be the elementary singular braid where the singularization takes place between strands $i$ and $i+1$ as in Example \ref{ex:sing}. Suppose $S_{\RN{2}}^i=\mathsf{X}_i\circ \mathsf{X}_i$ i.e. $S^i_{\RN{2}}$ is obtained from $\mathsf{X}_i$ by a MOY \rom{2} move as in Figure \ref{fig:MII}.


For each cycle $Z$ in $\mathsf{X}_i$, there is a cycle $\imath(Z)$ in $S_{\RN{2}}^i$ such that $b(Z)=b(\imath(Z))$, $t(Z)=t(\imath(Z))$ and if $Z$ is not locally empty at the crossing then $e_5\subset \imath(Z)$.  This induces an $A_n$-bimodule homomorphism $\imath_{\RN{2}}:\mathsf{M}[\mathsf{X}_i]\to\mathsf{M}[S^i_{\RN{2}}]$ such that $\imath_{\RN{2}}(x_Z)=x_{\imath(Z)}$.

\begin{lemma}\label{lem:MII}
With the above notation, 
\[\mathsf{M}[S_{\RN{2}}^i]\cong\mathsf{M}[\mathsf{X}_i]\otimes_{A_n}\mathsf{M}[\mathsf{X}_i]\cong\mathsf{M}[\mathsf{X}_i]\{1\}\oplus\mathsf{M}[\mathsf{X}_i]\{-1\}\]
as $A_n$-bimodules. Moreover, $\QQ[U_5,U_6]$ acts on this module such that multiplication by $U_5$ is given by 
\[\begin{bmatrix}0&-u_iu_{i+1}\\1&u_i+u_{i+1}\end{bmatrix},\] while $U_6=u_i+u_{i+1}-U_5$.
\end{lemma}

\begin{proof}
It is easy to check that the $A_n$-subbimodule $\mathsf{M}_1=\imath_{\RN{2}}(\mathsf{M}[\mathsf{X}_i])$ is canonically isomorphic to $\mathsf{M}[\mathsf{X}_i]$. Furthermore, $U_5\mathsf{M}_1\cong \mathsf{M}_1\cong\mathsf{M}[\mathsf{X}_i]$ and $\mathsf{M}_1\oplus U_5\mathsf{M}_1\subset \mathsf{M}[S_{\RN{2}}^i]$.  

For each element $x\in\mathsf{M}[S_{\RN{2}}^i]$, it follows from the relations that $(u_i+u_{i+1})x=x(u_i+u_{i+1})$. Using the linear relation $U_{6} = u_{i} + u_{i+1} - U_{5}$, we can substitute for $U_{6}$ in the internal action on $\mathsf{M}[S_{\RN{2}}^i]$. Similarly, the quadratic relation $U_{5}^{2} = (u_{i}+u_{i+1})U_{5} -u_{i}u_{i+1}$ allows us to substitute for $U_{5}^{2}$. Therefore, 
\[\mathsf{M}[S_{\RN{2}}^i]\cong \mathsf{M}_1\oplus U_5\mathsf{M}_1\cong\mathsf{M}[\mathsf{X}_i]\oplus\mathsf{M}[\mathsf{X}_i]\]
as $A_n$-bimodules (up to quantum grading shifts), and $U_5$ and $U_6$ acts as described. 

For each locally empty cycle $Z$ in $\mathsf{X}_i$, the complement of $Z$ contains one less $4$-valent vertex comparing to the complement of $\imath(Z)$ in $S_{\RN{2}}^i$. Thus, $\mathfrak{gr_q}(x_{\imath(Z)})=\mathfrak{gr_q}(x_{Z})+1$ for each locally empty cycle $Z$. Similarly, for the other cycles in $\mathsf{X}_i$ one can check that $\mathfrak{gr_q}(x_{\imath(Z)})=\mathfrak{gr_q}(x_{Z})+1$. Therefore, $\mathsf{M}_1\cong \mathsf{M}[\mathsf{X}_i]\{1\}$ and so $\mathsf{M}_2=U_5\mathsf{M}_1\cong\mathsf{M}[\mathsf{X}_i]\{-1\}$.

%
%
%

\end{proof}

\begin{corollary}\label{cor:MII}
Suppose $S=S_1\circ \mathsf{X}_i\circ S_2$ is a singular braid with $2n$ strands. Let $S'=S_1\circ S_{\RN{2}}^i\circ S_2$. Then,
\[H_{1+1}(S')\cong H_{1+1}(S)\otimes_{\QQ}\mathcal{A}\{1\}\]
\end{corollary}

\begin{proof} From Lemmas \ref{prop1} and \ref{prop2}, we see that on $H_{1+1}(S)$, $U_{5}=-U_{6}$, and $U_{5}^{2}=0$. So \[H_{1+1}(S)\cong H_{1+1}(S')\otimes_{\QQ}\mathcal{A}\{1\}\]
with $U_{5}$ acting as $U$ and $U_{6}$ acting as $-U$.
\end{proof}

\begin{lemma}\label{lem:cap} For each odd integer $i$
\[\mathsf{M}[\mathsf{X}_i]\otimes_{A_n}\mathsf{M}[\mycap_i]\cong \mathsf{M}[\mycap_i]\{1\}\oplus\mathsf{M}[\mycap_i]\{-1\}\]
as left $A_n$-modules. Moreover, considering the labels in Figure \ref{fig:MII} $\mathsf{M}[\mathsf{X}_i]\otimes_{A_n}\mathsf{M}[\mycap_i]$ is a module over $\QQ[U_5,U_6]$ where $U_5$ acts as 
\[
\begin{bmatrix}
0&-u_iu_{i+1}\\
1&u_i+u_{i+1}
\end{bmatrix}
\]
and $U_6=u_i+u_{i+1}-U_5$.
\end{lemma}
\begin{proof}
 For each cycle $Z$ in $\mycap_i$, let $\imath(Z)$ be the cycle in $Z$ where $b(Z)=b(\imath(Z))$ and $e_{5}\subset \imath(Z)$. The set of all generators $x_{\imath(Z)}$ generates an $A_n$-subbimodule $\mathsf{M}_1$ in $\mathsf[\mathsf{X}_i\circ \mycap_i]$, which is isomorphic to $\mathsf{M}[\mycap_i]$. Furthermore, $U_5\mathsf{M}_1\cong \mathsf{M}_1$ and \[\mathsf{M}[\mathsf{X}_i]\otimes_{A_n}\mathsf{M}[\mycap_i]\cong\mathsf{M}_1\oplus U_5\mathsf{M}_1.\]
The rest is similar to the proof of Lemma \ref{lem:MII}.
\end{proof}

Similarly, we can prove that: 

\begin{lemma}\label{lem:cup}
As right $A_n$-modules
\[\mathsf{M}[\mycup_i]\otimes_{A_n}\mathsf{M}[\mathsf{X}_i]\cong \mathsf{M}[\mycup_i]\{1\}\oplus\mathsf{M}[\mycup_i]\{-1\}.\]
Moreover, considering the labels in Figure \ref{fig:MII} $\mathsf{M}[\mycup_i]\otimes_{A_n}\mathsf{M}[\mathsf{X}_i]$ is a module over $\QQ[U_5,U_6]$ where $U_5$ acts as 
\[
\begin{bmatrix}
0&-u_iu_{i+1}\\
1&u_i+u_{i+1}
\end{bmatrix}
\]
and $U_6=u_i+u_{i+1}-U_5$.
\end{lemma}

\begin{corollary}
Let $S$ be a singular braid with $2n$-strands. Then, for each odd integer $i$
\[H_{1+1}(S\circ\mathsf{X}_i)\cong H_{1+1}(\mathsf{X}_i\circ S)\cong H_{1+1}(S)\otimes_{\QQ}\mathcal{A}\{1\}.\]
\end{corollary}
\begin{proof}
It is a straightforward corollary of Lemmas \ref{lem:cap} and \ref{lem:cup}, with the same proof as Corollary \ref{cor:MII}.

\end{proof}

Finally, we will review some properties of the chain map $\imath_{\RN{2}}$ that will be used later to prove invariance. 


\begin{lemma}\label{lem:propII}
For all $\mathsf{X}_i$ the followings hold:
\begin{enumerate}
\item $(d^-\otimes \mathrm{id})\circ \imath_{\RN{2}}=(\mathrm{id}\otimes d^-)\circ\imath_{\RN{2}}=\mathrm{id}$.
\item $\pi_{\RN{2}}^2\circ (d^+\otimes\mathrm{id})=\pi_{\RN{2}}^2\circ (\mathrm{id}\otimes d^+)=\mathrm{id}$, while  $\pi_{\RN{2}}^1\circ (d^+\otimes\mathrm{id})=-U_2\mathrm{id}$ and $\pi_{\RN{2}}^1\circ (\mathrm{id}\otimes d^+)=-U_4\mathrm{id}$

\end{enumerate}
Here, $\pi_{\RN{2}}^i:\mathsf{M}[S_{\RN{2}}^i]\to\mathsf{M}[\mathsf{X}_i]$ is the projection onto the $i$-th summand for $i=1,2$.\end{lemma}

\begin{proof}
The $A_n$-bimodule $\mathsf{M}[\mathsf{X}_i]$ is generated by $x_Z$ for all cycles $Z$ in $\mathsf{X}_i$ such that $b(Z)=t(Z)$ and   $i+1\notin b(Z)$. It is then clear from the definition of $\imath_{\RN{2}}$, $\pi_{\RN{2}}$, $d^-$ and $d^+$ that for any such $Z$ the Equations (1) and (2) hold. 
\end{proof}

%
%


\subsubsection{MOY \rom{1} and MOY \rom{3}} Consider the singular braid $S_{\RN{3}a}^i=\mathsf{X}_{i}\circ\mathsf{X}_{i+1}\circ\mathsf{X}_{i}$ which is obtained from $\mathsf{X}_i$ by applying the MOY \rom{3} move as in Figure \ref{fig:MIII}.  We define the $A_n$-bimodule homomorphism \[\imath_{\RN{3}a}:\mathsf{M}[\mathsf{X}_i]\to\mathsf{M}[S_{\RN{3}a}^i]\cong \mathsf{M}[\mathsf{X}_{i}]\otimes_{A_n}\mathsf{M}[\mathsf{X}_{i+1}]\otimes_{A_n}\mathsf{M}[\mathsf{X}_{i}]\]
as the composition $\imath_{\RN{3}a}=(\mathrm{id}\otimes d^+\otimes \mathrm{id})\circ\imath_{\RN{2}}$ where 
\begin{multline*}
\mathrm{id}\otimes d^+\otimes \mathrm{id}: \mathsf{M}[S_{\RN{2}}^{i}]\cong\mathsf{M}[\mathsf{X}_{i}]\otimes_{A_n}\mathsf{M}[S_{id}]\otimes_{A_n}\mathsf{M}[\mathsf{X}_{i}]\\ \to \mathsf{M}[\mathsf{X}_{i}]\otimes_{A_n}\mathsf{M}[\mathsf{X}_{i+1}]\otimes_{A_n}\mathsf{M}[\mathsf{X}_{i}].
\end{multline*}
Let $\mathsf{M}_1^a=\imath_{\RN{3}a}(\mathsf{M}[\mathsf{X}_{i}])$.

\begin{lemma}\label{lem:submodMOY3}
The homomorphism $\imath_{\RN{3}a}$ is injective and so $\mathsf{M}_1^a\cong\mathsf{M}[\mathsf{X}_i]$. \end{lemma}

\begin{proof}
For each cycle $Z$ in $\mathsf{X}_i$ there is a unique cycle $\ol{Z}$ in $S_{\RN{3}a}^i$ such that $b(\ol{Z})=b(Z)$, $t(\ol{Z})=t(Z)$ and if $b(Z)\cap\{i,i+1\}\neq\emptyset$ then $e_9\subset \ol{Z}$. It follows from the definition of $d^+$ and $\imath_{\RN{2}}$ that 
\[\imath_{\RN{3}a}(x_Z)=\begin{cases}
\begin{array}{lll}
(U_8-U_3)x_{\ol{Z}}&&\text{if}\ i+2\notin b(Z)\\
x_{\ol{Z}}-\epsilon R_{i+2}x_{\ol{u_{i+2}(Z)}}L_{i+2}&&\text{otherwise}
\end{array}
\end{cases}
\]
where $\epsilon=0$ if $i+3\in b(Z)$, otherwise $\epsilon=1$. It is straightforward that $\imath_{\RN{3}a}$ is injective and thus 
$\mathsf{M}^a_1\cong\mathsf{M}[\mathsf{X}_i]$.
\end{proof}

As before, $\mathcal{Z}\{i_1,...,i_k\}$ denotes the set of cycles that locally contain the edges $e_{i_1},...,e_{i_k}$ for any subset $\{i_1,...,i_k\}\subset \{1,...,9\}$. Let $\mathsf{M}_2^a$ be the $A_n$-bimodule generated by $x_Z$ for every $Z\in\mathcal{Z}\emptyset\cup\mathcal{Z}\{1,9,4\}$.  In particular, $x_Z\in\mathsf{M}_2^a$ for all 

\begin{multline*}
Z\in\mathcal{Z}\emptyset\amalg\mathcal{Z}\{1,9,4\}\amalg\mathcal{Z}\{1,9,5\}\amalg\mathcal{Z}\{2,9,4\}\amalg\mathcal{Z} \{2,9,5\}\amalg\\
\mathcal{Z}\{1,7,6\}\amalg \mathcal{Z}\{2,7,6\}\amalg \mathcal{Z}\{3,8,4\}\amalg\mathcal{Z}\{3,8,5\}.
\end{multline*}
Moreover, if $Z\in\mathcal{Z}\{3,6\}$ then $U_9x_Z\in \mathsf{M}_2$.


\begin{lemma}\label{lem:MIIIsplit}
The $A_n$-bimodule $\mathsf{M}[S_{\RN{3}a}^i]$  splits as the direct sum $\mathsf{M}_1^a\oplus\mathsf{M}_2^a$. 
\end{lemma}

%
%
%
%
%
%
%
%

%
\begin{proof}

The fact that $\mathsf{M}_{1}^a$ and $\mathsf{M}_{2}^a$ are disjoint follows from checking the local relations, which we leave to the reader. To see that $\mathsf{M}_{1}^a \oplus \mathsf{M}_{2}^a$ spans $\mathsf{M}(S^i_{\RN{3}a})$, note that we can substitute for $U_{7}$, $U_{8}$, and $U_{9}^{2}$ by the relations

\[  
U_{7} = U_{1}+U_{2}-U_{9}, \hspace{5mm} U_{8} = U_{4}+U_{5} - U_{9}, \hspace{5mm} U_{9}^{2} = (U_{1}+U_{2})U_{9} - U_{1}U_{2}
\]

\noindent
The variables $U_{1},...,U_{6}$ are represented in the two algebra actions, so $\mathsf{M}(S^i_{\RN{3}a})$ is generated by $ x_{Z}, U_{9}x_{Z}$ over all cycles $Z$.

For any cycle $Z$ that $i+2\notin t(Z)\cup b(Z)$ i.e. $Z$ does not contain $e_3$ and $e_6$, $\mathsf{M}_2^a$ contains $x_Z$ and $\mathsf{M}_1^a$ contains $ (U_{3}-U_{8})x_{Z} = (U_{3}-U_{4}-U_{5}+U_{9})x_{Z}$. Since $U_{3}$, $U_{4}$, and $U_{5}$ are in the bimodule action, the span of these two elements is $\langle x_{Z}, U_{9}x_{Z} \rangle$. 


If $Z\in\mathcal{Z}\{3,6\}$, then $\mathsf{M}_1^a$ contains $x_Z-\epsilon R_{i+2}x_{u_{3}(Z)}R_{i+2}$ and $\mathsf{M}_2^a$ contains $x_{u_{3}(Z)}$, so $x_Z$ is in the direct sum. Also, $\mathsf{M}_2^a$ contains $U_9x_{Z}$.

If $Z\in\mathcal{Z}\{1,7,6\}$, then $x_Z\in\mathsf{M}_2^a$ and $U_9x_{Z}=(U_1+U_2-U_7)x_Z$. Thus, it suffices to show that $U_7x_Z\in\mathsf{M}_1^a\oplus\mathsf{M}_2^a$. Note that $U_{7}x_{Z}=R_{i}R_{i+1}x_{U_7(Z)}$ and $U_7(Z)\in\mathcal{Z}\{3,6\}$. So $x_{U_7(Z)}$ and consequently $R_{i}R_{i+1}x_{U_{7}(Z)}$ is in $\mathsf{M}_{1}^a \oplus \mathsf{M}_{2}^a$, as well. The proof for the cycles $Z\in\mathcal{Z}\{2,7,6\}\amalg\mathcal{Z}\{3,8,4\}\amalg\mathcal{Z}\{3,8,5\}$ are similar. 

Finally, for any cycle $Z$ in \[\mathcal{Z}\{1,9,4,3,6\}\amalg\mathcal{Z}\{1,9,5,3,6\}\amalg\mathcal{Z}\{2,9,4,3,6\}\amalg\mathcal{Z}\{2,9,5,3,6\},\] $\mathsf{M}_2^a$ contains $x_{u_3(Z)}$, and so $\mathsf{M}_1^a\oplus\mathsf{M}_2^a$ contains $x_{Z}$. Furthermore, $U_9x_Z=0$. Thus, we are done. 
\end{proof}

Similar statements hold for the singular braid $S_{\RN{3}b}^i=\mathsf{X}_{i}\circ\mathsf{X}_{i-1}\circ\mathsf{X}_{i}$ obtained from $\mathsf{X}_i$ by applying a type b MOY \rom{3}, as in Figure \ref{fig:MIII}. Specifically, the $A_n$-bimodule $\mathsf{M}[S_{\RN{3}b}^i]$ splits as the direct sum $\mathsf{M}^b_1\oplus\mathsf{M}_2^b$ where $\mathsf{M}_1^b=\imath_{\RN{3}b}(\mathsf{M}[\mathsf{X}_i])\cong\mathsf{M}[\mathsf{X}_i]$, and $\mathsf{M}_2^b$ gives a cyclic summand of the total complex. 

In addition, the summands $\mathsf{M}_2^a$ and $\mathsf{M}_2^b$ of $\mathsf{M}[S_{\RN{3}a}^i]$ and $\mathsf{M}[S_{\RN{3}b}^{i+1}]$ are isomorphic, with isomorphisms
\begin{equation}\label{eq:isomofcyclic}
j_{ab}:\mathsf{M}_2^a\to\mathsf{M}_2^b\ \ \ \text{and}\ \ \ j_{ba}:\mathsf{M}_2^b\to\mathsf{M}_2^a,
\end{equation}
defined as follows. For every cycle $Z\in\mathcal{Z}\emptyset\amalg\mathcal{Z}\{1,9,4\}$ of $S_{\RN{3}a}^i$ there is a unique cycle $Z'\in\mathcal{Z}\emptyset\amalg\mathcal{Z}\{1,4\}$ of $S_{\RN{3}b}^i$ such that $b(Z)=b(Z')=t(Z)=t(Z')$ and vice versa. We set $j_{ab}(x_Z)=x_{Z'}$ and $j_{ba}(x_{Z'})=x_Z$. By checking the local relations, it is straightforward that these relations induce well-defined isomorphisms as above.

\begin{corollary}\label{cor:IIIcap}
Let $S$ be the singular tangle obtained by applying a MOY \rom{3} move to the singular cap $\mycap_i$ (resp. cup $\mycup_i$), as in Figure \ref{fig:MIII}. Then, there is an injective homomorphism from $\mathsf{M}[\mycap_i]$ (resp. cup $\mathsf{M}[\mycup_i]$) to $\mathsf{M}[S]$ 
such that its corresponding short exact sequence splits. 
\end{corollary}

\begin{proof} Without loss of generality, we assume that $S=\mathsf{X}_i\circ\mathsf{X}_{i+1}\circ\mycap_i$ i.e. the singular tangle obtained from applying type a MOY \rom{3} move to $\mycap_i$. Let $S'=S_{\RN{3}a}^i\circ\mycap_i$, and consider the $A_n$-bimodule homomorphism 
\[\imath_{\RN{3}a}\otimes\mathrm{id}:\mathsf{M}[\mathsf{X}_i]\otimes_{A_n}\mathsf{M}[\mycap_i]\to \mathsf{M}[S']\cong\mathsf{M}[S_{\RN{3}a}^i]\otimes_{A_n}\mathsf{M}[\mycap_i].\]
By Lemma \ref{lem:cap} we have 
\[
\begin{split}
&\mathsf{M}[S']\cong \mathsf{M}[S]\{1\}\oplus\mathsf{M}[S]\{-1\}\\
&\mathsf{M}[\mathsf{X}_i\circ\mycap_i]\cong\mathsf{M}[\mycap_i]\{1\}\oplus\mathsf{M}[\mycap_i]\{-1\}
\end{split}
\]
It is clear from the definition that $\imath_{\RN{3}a}\otimes\mathrm{id}$ respects these splittings, and induces an $A_n$-bimodule homomorphism
\[\imath:\mathsf{M}[\mycap_i]\to\mathsf{M}[S].\]
Then, injectivity of $\imath_{\RN{3}a}$ implies injectivity of $\imath$ and since the short exact sequence for $\imath_{\RN{3}a}$ splits, the short exact sequence for $\imath$ also splits. 
\end{proof}


\begin{lemma}\label{lem:MIII}
Let $S$ and $S'$ be singular braids such that $\pc(S')$ is obtained from $\pc(S)$ by the MOY \rom{3} local move in Figure \ref{fig:MIII}. Then, the induced chain map from $C_{1\pm1}(S)$ to $C_{1\pm 1}(S')$ by the homomorphism $\imath$ from Lemma \ref{lem:submodMOY3} or Corollary \ref{cor:IIIcap} is a quasi-isomorphism. 
\end{lemma}

\begin{proof}
By Lemmas \ref{lem:submodMOY3} and \ref{lem:MIIIsplit}, and Corollary \ref{cor:IIIcap}, we get a splitting $C_{1\pm1}(S')=C_1\oplus C_2$ so that $H_\star(C_1)=H_{1+1}(S)$. By applying Corollary \ref{KhovCor}, $\dim H_{1+1}(S) = \dim H_{1+1}(S')$, so all of the homology must come from the complex $C_{1}$. Thus, $C_2$ is acyclic and the chain map induced by $\imath$ is a quasi-isomorphism. 
\end{proof}

\begin{corollary}(MOY \rom{1}) Assume $S$ and $S'$ be singular braids, such that $\pc(S')$ is obtained from $\pc(S)$ by one of the local moves in Figure \ref{fig:MI}. Then, $H_{1+1}(S)\cong H_{1\pm 1}(S')$.
\end{corollary}
\begin{proof}
Suppose $S'$ and $S$ differ by a type a MOY \rom{1} move, and the extra $4$-valent vertex of $S'$ is a crossing between second and third strands. Let $S''=S'\circ \mathsf{X}_1$. Then, $\pc(S'')$ is obtained from $\pc(S)\amalg \pc(\mathcal{U}_2)$ by a MOY \rom{3} move, so $\imath_{\star}$ induces an isomorphism $H_{1+1}(S'')\cong H_{1+1}(S)\otimes_{\QQ}\mathcal{A}\{1\}$. On the hand, $\pc(S'')$ differs from $\pc(S')$ by a MOY \rom{2} move. Thus, $H_{1+1}(S'')\cong H_{1+1}(S')\otimes_{\QQ}\mathcal{A}\{1\}$. It is straightforward, that $\imath_{\star}$ induces and isomorphism between $H_{1+1}(S')$ and $H_{1+1}(S)$.

\end{proof}


Finally, we will review some properties of this splitting that will be used in Section \ref{sec:invariance} for proving invariance.

For $\circ\in\{a,b\}$, let $\pi_{\RN{3}\circ}^1$ be the map from $\mathsf{M}[S_{\RN{3}\circ}^i]$ to $\mathsf{M}[\mathsf{X}_i]$ defined by composing the projection on $\mathsf{M}_1^\circ$ with $\imath_{\RN{3}\circ}^{-1}$, and $\pi_{\RN{3}\circ}^2$ be the projection on $\mathsf{M}_2^\circ$.

\begin{lemma}\label{lem:propIII-1}
For every $i$ and $\circ=a,b$ we have:
\[\pi_{\RN{3}\circ}^1=-\pi_{\RN{2}}^2\circ (\mathrm{id}\otimes d^-\otimes\mathrm{id})\]
\end{lemma}
\begin{proof}
By definition,
\[\begin{split}
\pi_{\RN{2}}^2\circ(\mathrm{id}\otimes d^-\otimes \mathrm{id})\circ \imath_{\RN{3}a}&=\pi_{\RN{2}}^2\circ(\mathrm{id}\otimes d^-\otimes \mathrm{id})\circ (\mathrm{id}\otimes d^+\otimes \mathrm{id})\circ\imath_{\RN{2}}\\
&=\pi_{\RN{2}}^2\circ (U_6 \imath_{\RN{2}})=  \pi_{\RN{2}}^2\circ ((U_3+U_4-U_5) \imath_{\RN{2}})=-\mathrm{id}.
\end{split}\]
Note that in the second line of the above equalities, we are considering $S_{\RN{2}}^i$ with the labeling as in Figure \ref{fig:MII}.
\end{proof}

Let $S=\mathsf{X}_i\circ \mathsf{X}_{i+1}$. Then, 
\[S_{\RN{3}a}^i=S\circ \mathsf{X}_i\quad \text{and}\quad S_{\RN{3}b}^{i+1}=\mathsf{X}_{i+1}\circ S.\]
The edge homomorphism $d^+$ induces homomorphisms $f_a=\mathrm{id}\otimes d^+$ and $f_b=d^+\otimes \mathrm{id}$ from $\mathsf{M}[S]$ to $\mathsf{M}[S_{\RN{3}a}^i]$ and $\mathsf{M}[S_{\RN{3}b}^{i+1}]$. 

Moreover, $\mathsf{X}_i$ and $\mathsf{X}_{i+1}$ are obtained from $S$ by smoothing the top and bottom singular points, respectively.  So, the edge map $d^-$ induces homomorphisms $g_a=\mathrm{id}\otimes d^-$ and $g_b=d^-\otimes \mathrm{id}$ from $\mathsf{M}[S]$ to $\mathsf{M}[\mathsf{X}_i]$ and $\mathsf{M}[\mathsf{X}_{i+1}]$.

\begin{lemma}\label{lem:propIII-2} With the above notation fixed,
\begin{enumerate}
\item For $\circ=a,b$ we have $g_\circ=-\pi_{\RN{3}\circ}^1\circ f_\circ$.
\item $j_{ab}\circ \pi_{\RN{3}a}^2\circ f_a=\pi_{\RN{3}b}^2\circ f_b$.
\end{enumerate}
\end{lemma}

\begin{proof} First, we prove part (1). For the singular braid $S_{\RN{2}}^i$, the edge maps induce homomorphisms $\mathrm{id}\otimes d^-\otimes\mathrm{id}$ and  $\mathrm{id}\otimes d^+$ from $\mathsf{M}[S_{\RN{3}a}^i]$ and $\mathsf{M}[\mathsf{X}_i]$ to $\mathsf{M}[S_{\RN{2}}^i]$, respectively. By Lemma \ref{lem:propIII-1} we have 
\[
-\pi_{\RN{3}a}^1\circ f_a=\pi_{\RN{2}}^2\circ (\mathrm{id}\otimes d^-\otimes \mathrm{id})\circ f_a=\pi_{\RN{2}}^2\circ(\mathrm{id}\otimes d^+)\circ g_a=g_a
\]
The last equality follows from part (2) of Lemma \ref{lem:propII}.  The proof for $\circ=b$ is similar. 

For part (2), it is enough to check it for the generators $x_Z$ where 
\[Z\in\mathcal{Z}\emptyset\amalg\mathcal{Z}\{1,4\}\amalg\mathcal{Z}\{3,5\}\amalg \mathcal{Z}\{1,3,4,5\}\]
Each $Z\in \mathcal{Z}\emptyset\amalg\mathcal{Z}\{1,4\}\amalg\mathcal{Z}\{3,5\}$ extends to unique cycles $Z_a$ and $Z_b$ in $S_{\RN{3}a}^i$ and $S_{\RN{3}b}^i$ such that $b(Z)=t(Z)=b(Z_a)=t(Z_a)=b(Z_b)=t(Z_b)$. Moreover, $j_{ab}(x_{Z_a})=x_{Z_b}$.  It follows from the definitions that for $Z\in \mathcal{Z}\emptyset\amalg\mathcal{Z}\{1,4\}$
\[j_{ab}\circ\pi_{\RN{3}a}^2\circ f_a(x_Z)=j_{ab}\left((U_4-U_3)x_{Z_a}\right)=\pi_{\RN{3}b}^2\circ f_b(x_Z),\]
and for $Z\in\mathcal{Z}\{3,5\}$
\[j_{ab}\circ\pi_{\RN{3}a}^2\circ f_a(x_Z)=j_{ab}\left(x_{Z_a}-R_{i+2}x_{u_3(Z)_a}L_{i+2}L_{i+1}\right)=\pi_{\RN{3}b}^2\circ f_b(x_Z).\]
Finally, for $Z\in \mathcal{Z}\{1,3,4,5\}$, 
\[j_{ab}\circ\pi_{\RN{3}a}^2\circ f_a(x_Z)=-j_{ab}\left(R_{i+2}x_{u_3(Z)_a}L_{i+2}L_{i+1}\right)=\pi_{\RN{3}b}^2\circ f_b(x_Z).\]
\end{proof}

Similarly, the edge homomorphism $d^-$ induces homomorphisms $f_a'=\mathrm{id}\otimes d^-$ and $f_b'=d^-\otimes \mathrm{id}$ from $\mathsf{M}[S^{i}_{\RN{3}a}]$ and $\mathsf{M}[S^{i+1}_{\RN{3}b}]$ to $\mathsf{M}[S]$. Further, $d^+$ induces homomorphisms $g_a'=\mathrm{id}\otimes d^+$ and $g_{b}'=d^+\otimes\mathrm{id}$ from $\mathsf{M}[\mathsf{X}_i]$ and $\mathsf{M}[\mathsf{X}_{i+1}]$ to $\mathsf{M}[S]$. 

\begin{lemma}\label{lem:propIII-3} For $\circ=a,b$ we have $g_{\circ}'=f_{\circ}\circ \imath_{\RN{3}\circ}$ and $f_a'|_{\mathsf{M}_2^a}=f_b'\circ j_{ab}$.
\end{lemma}
\begin{proof}
The proof is similar to the proof of Lemma \ref{lem:propIII-2}.
\end{proof}

Analogous statements hold for $S=\mathsf{X_{i+1}}\circ \mathsf{X}_i$.


\subsection{The Module Action for Singular Braids} \label{ModuleActionSection} Let $S$ be a singular braid such that $\sm(\pc(S))$ consists of $k$ circles. In Corollary \ref{modulelemma1}, we showed that $H_{1+1}(S)$ is a module over $\mathcal{A}^{\otimes k}$. Moreover, from Corollary \ref{KhovCor}, we know that the rank over $\mathbb{Q}$ is $2^k$. In this section, we will prove the following:

\begin{theorem} \label{SingularKhovanovTheorem}
The homology $H_{1+1}(S)$ is a free, rank one $\mathcal{A}^{\otimes k}$-module.
\end{theorem}

We will prove this by induction on the number of components in $\sm(\pc(S))$. Let $c$ be a component of $\sm(\pc(S))$ which does not contain any other components, and let $E(c)$ be the set of edges in $c$. In $S$, the edges $E(c)$ bound a set of regions $A_{i}$ which make up the interior of $c$ in $\sm(\pc(S))$. Some of these regions will meet at a 4-valent vertex as in Figure \ref{Tree1}. 

\begin{figure}[ht]
\centering
\def\svgwidth{5cm} 
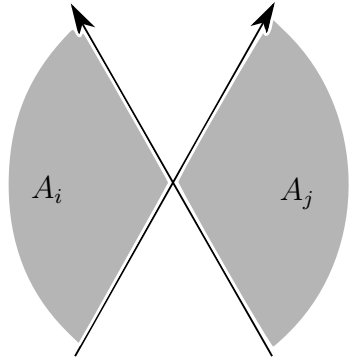
\caption{Two regions meeting at a 4-valent vertex}\label{Tree1}
\end{figure}

\begin{definition}

Define the graph $G(c)$ to have vertices $A_{i}$, and add in an edge between $A_{i}$ and $A_{j}$ for each 4-valent vertex as in Figure \ref{Tree1}.

\end{definition}

Since the edges correspond to a single circle when the singularizations are replaced with smoothings, this graph must be a tree.

\begin{definition}

We say that a region $A_{i}$ meets a vertex $v$ from the left, right, top, or bottom, depending on which of the 4 quadrants are occupied.

\end{definition}

\begin{lemma} \label{prunelemma}

Let $\mathcal{L}$ be a leaf on the tree $G(c)$. Then there is a MOY I or a MOY III move which trims this leaf.

\end{lemma}

\begin{proof}

Let $A_{i}$ be the vertex on $\mathcal{L}$, and $v$ the unique 4-valent vertex at which $A_{i}$ meets another region $A_{j}$. Without loss of generality, assume that $A_{i}$ meets $v$ from the left. Since the boundary of $A_{i}$ can not meet any other vertices from the left or right, the only option is for it to meet a single vertex from the bottom and a single vertex from the top (see Figure \ref{Tree2}). The resulting region corresponds to a MOY I or a MOY III move, and applying this relation trims the leaf. 

\end{proof}

\begin{figure}[ht]
\centering
\def\svgwidth{5cm} 
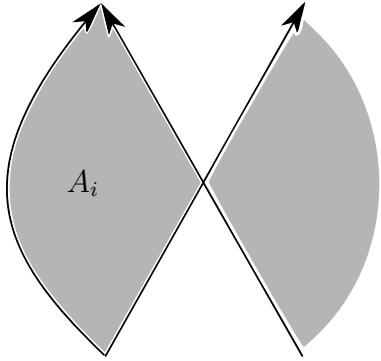
\caption{A leaf on the tree $G(c)$}\label{Tree2}
\end{figure}

\begin{proof}[Proof of Theorem \ref{SingularKhovanovTheorem}]

Given an innermost circle $c$ in $\sm(\pc(S))$, we can recursively apply MOY I and MOY III moves as in Lemma \ref{prunelemma} to prune $G(c)$ down to a single vertex. This gives a new singular diagram $S'$ with $H_{1+1}(S)=H_{1+1}(S')$. 

In $S'$, the circle $c$ has been replaced with a circle $c'$ with $G(c')$ equal to a single vertex. Let $A'_{i}$ denote the corresponding region in $S'$. Since $A'_{i}$ can't meet any regions from the left or right, it must meet a single vertex from the bottom and a single vertex from the top, i.e. it is a bigon. Thus, we can apply a MOY 0 or a MOY II move to remove the two edges bounding $A'_{i}$, obtaining a new diagram $S''$ satisfying 

\[ H_{1+1}(S) \cong H_{1+1}(S'') \otimes \mathcal{A} \]

For $e_{i}$ in $E(c)$, $U_{i}$ acts on $H_{1+1}(S'') \otimes \mathcal{A}$ by $\text{id} \otimes (-1)^{l}U$. The remaining edges in $S$ correspond to edges in $S''$. The smoothed diagram $\sm(\pc(S''))$ has $k-1$ components, so by the inductive hypothesis $H_{1+1}(S'') \cong \mathcal{A}^{\otimes(k-1)}$, proving the theorem.

\end{proof}

\subsection{The Quantum and $\delta$-gradings on the Total Complex} Now that we understand each vertex in the cube of resolutions, it is time to turn our attention to the total complex. Before identifying the edge maps with the Khovanov edge maps, we'll need the total complex $C_{1 \pm 1}(D)$ to be graded. 

We have a grading $\mathfrak{gr}_{q}$ on each vertex of the cube, with respect to which the variables $U_{i}$ have grading $-2$, and the differential $d_{0}$ is homogeneous of degree $-2$. We extend $\mathfrak{gr}_{q}$ to the total complex as follows: let $S = D_{v}$ be a complete resolution of $D$, and let $|v|$ be the height of $v$ in the cube of resolutions. Then we give $\mathfrak{gr}_{q}$ a shift of $|v|$ from the definition in Section \ref{gradingsect1}.

To be more precise, if $Z$ is a cycle in $S=D_{v}$, then $x_{Z}$ has quantum grading 
\[ \mathfrak{gr}_{q}(x_{Z}) = T_{1}(Z)-T_{2}(Z)+E(Z)+w(Z) +|v|+n_{+}-2n_{-} \]

With respect to $\mathfrak{gr}_{q}$, we can see by inspection that the edge maps $d_{1}$ are homogeneous of degree $0$. Unfortunately, since $d_{0}$ is homogeneous of degree $-2$, $\mathfrak{gr}_{q}$ does not give a well-defined grading on the total homology. 

\begin{definition}

The $\delta$-grading $\mathfrak{gr}_{\delta}$ is given by $\mathfrak{gr}_{q} - 2 |v| + 2 n_{-}$.

\end{definition}

With respect to the $\delta$-grading, $d_{0}+d_{1}$ is homogeneous of degree $-2$, and multiplication by $U_{i}$ is as well. 

\begin{lemma}

Under the module isomorphism
\[ H_{*}(C_{1 \pm 1}(D), d_{0}) \cong CKh(\pc(D)) \]

\noindent
the gradings $\mathfrak{gr}_{q}$ and $\mathfrak{gr}_{\delta}$ on $C_{1 \pm 1}(D)$ correspond to the gradings $\gr_{q}$ and $\gr_{\delta}$ on the Khovanov complex.

\end{lemma}

\begin{proof}

Applying the MOY relations to a complete resolution $D_{v}$, we see that the generator of the homology $H_{1+1}(D_{v})$ lies in quantum grading $k_{v}+|v|+n_{+}-2n_{-}$, just as in the Khovanov complex. The $\delta$-gradings being the same follows from the definition of the $\delta$-grading on Khovanov homology.

\end{proof}

\subsection{Isomorphism Between $H_{1+1}(L)$ and Khovanov Homology}

In this section, we will show that the $E_{2}$ page of the spectral sequence on $C_{1 \pm 1}(D)$ induced by the cube filtration is isomorphic to the Khovanov homology of $\pc(D)$.

\begin{theorem} \label{FullKhovIsoTheorem}

Let $D$ be an open braid whose plat closure $\pc(D)$ is a diagram for a link $L$. Then 

\[H_{1+1}(D)=H_{*}(H_{*}(C_{1\pm 1}(D), d_{0}), d^{*}_{1}) \cong Kh(L) \]

\end{theorem}

Let $S$ be a complete resolution of $D$. In Theorem \ref{SingularKhovanovTheorem}, we showed that $H_{1+1}(S)$ is isomorphic to $Kh(\sm(\pc(S)))$. Thus, $H_{*}(C_{1 \pm 1}(D), d_{0})$ is isomorphic to the Khovanov complex as a module, so Theorem \ref{FullKhovIsoTheorem} will follow from the induced edge maps being the Khovanov edge maps.

\begin{proof}
Let $S_{1}$ and $S_{2}$ be two complete resolutions which differ at a single crossing, with $S_{1}$ having the singularization and $S_{2}$ the smoothing, and suppose the edges are labeled $e_{1}, e_{2}, e_{3}, e_{4}$ as in Figure \ref{labeled4v}. There are two edge maps, corresponding to a positive and negative crossing, respectively:
\[ d_{1}^{-}: C_{1 \pm 1}(S_{1}) \to C_{1 \pm 1}(S_{2}) \]
\[ d_{1}^{+}: C_{1 \pm 1}(S_{2}) \to C_{1 \pm 1}(S_{1}) \]

Applying Lemma \ref{lem:comp}, we see that
\[ d_{1}^{-} \circ d_{1}^{+} = d_{1}^{+} \circ d_{1}^{-}= U_{1}-U_{4} \]

\noindent
Thus, after taking homology with respect to $d_{0}$, the induced edge maps satisfy the same property.

After identifying $\QQ[U_{1},...,U_{k}]$ with the Khovanov ground ring $\QQ[X_{1},...,X_{k}]$ via $X_{i}=(-1)^{b(i)}U_{i}$, where $b(i)$ is the the number of the strand on which $e_{i}$ lies, we get 
\[ (d_{1}^{-})^{*} \circ (d_{1}^{+})^{*} = (d_{1}^{+})^{*} \circ (d_{1}^{-})^{*}= \pm(X_{1}+X_{4}) \] 

In other words, up to sign, the composition of our two edge maps is precisely the composition of the two Khovanov edge maps. Since the edge maps are module homomorphisms, this composition uniquely determines the two edge maps up to scaling by some non-zero element of $\QQ$.

Consider first the homomorphism corresponding to a merge of two circles. Since the map preserves the quantum grading, the only option is 
\[ 1 \mapsto r \]

\noindent
for some $r \in \QQ$. Thus, the merge map is $r$ times the Khovanov merge map. The split map then sends $r$ to $\pm(X_{1}+X_{4})$, so it is $\pm 1/r$ times the Khovanov split map.

So we have shown that each edge map is the Khovanov edge map, up to scaling. It is not hard to see that any such complex is isomorphic to the Khovanov complex, where the isomorphism is obtained by scaling the complex at each vertex so that the edge maps become the Khovanov edge maps.

\end{proof}

\section{Invariance of the Total Homology}\label{sec:invariance} 

In this section we will prove our main theorem, that the total homology is a link invariant.

\begin{theorem}

The total homology \[H_{1-1}(D)=H_{*}(C_{1\pm1}(D), d_0 + d_1)\] is a graded link invariant.

\end{theorem}

To prove invariance, we need to show that the homology is invariant under five types of local moves: braid-like Reidemeister I, II, and III moves, twists at the top and bottom, and cap swaps/cup swaps. For Reidemeister I, twists, and cap swaps / cup swaps, we will use the following lemma from homological algebra:

\begin{lemma}[\cite{SSBook}]\label{SSLemma} Let $f:C_{1} \to C_{2}$ be a filtered chain map of filtered chain complexes. If $f$ induces an isomorphism on the $E_{k}$ pages of the associated spectral sequences, then it induces an isomorphism on the $E_{l}$ pages for $l \ge k$. 
\end{lemma}

Together with Theorem \ref{FullKhovIsoTheorem} and familiar results from Khovanov homology, these invariance proofs will be fairly easy. For the braid-like Reidemeister II and III moves, we will actually prove that the chain homotopy type of the bimodules is invariant, making the ($A_{n}, A_{n}$)-bimodule an invariant of open braids.

\subsection{Twists at Top and Bottom}

Let $D_{1}$ and $D_{2}$ be two diagrams which differ by a single positive twist near $w_{i}^{+}$ as in Figure \ref{fig:PTwist}.

Let $D_1^0$ and $D_1^1$ denote the $0$- and $1$-resolutions of $D_1$, respectively, at the crossing shown in Figure \ref{fig:PTwist}. Note that $D_1^0$ is isotopic to $D_2$, and $D_1^1$ differs from $D_2$ by a MOY \rom{2} move. We define 
\[f=(f_0,f_1):C_{1\pm 1}(D_2)\to C_{1\pm 1}(D_1)=C_{1\pm 1}(D_1^0)\oplus C_{1\pm 1}(D_1^1)\]
such that $f_0=0$ and $f_1=\imath_{\RN{2}}$.



\begin{figure}[ht]
\centering
\def\svgwidth{6.5cm}
\input{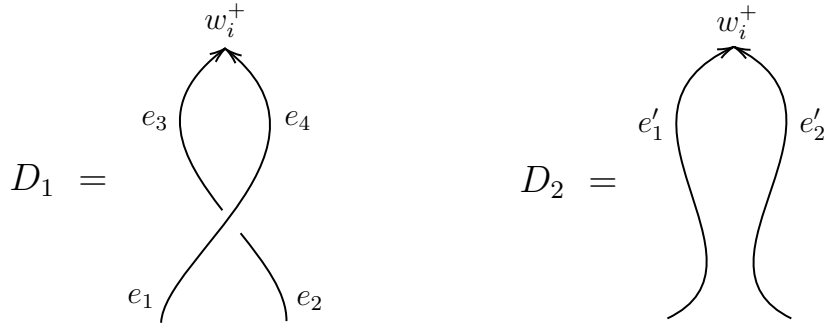}
\caption{Two diagrams differing by a positive twist at $w_{i}^{+}$}\label{fig:PTwist}
\end{figure}

The map $f$ is a filtered chain map with respect to the cube filtration. Let $E_{k}(C_{1 \pm 1}(D))$ denote the spectral sequence induced by the cube filtration. By the MOY II relation, $E_{1}(C_{1 \pm 1}(D_{1}^{1})) \cong E_{1}(C_{1 \pm 1}(D_{2})) \otimes \mathcal{A}$, and with respect to this isomorphism, $f$ maps $E_{1}(C_{1 \pm 1}(D_{2}))$ to $E_{1}(C_{1 \pm 1}(D_{2})) \otimes 1$. 

This is the standard quasi-isomorphism on Khovanov homology corresponding to a Reidemeister I move, so $f^{*}: E_{2}(D_{2}) \to E_{2}(D_{1})$ is an isomorphism. Applying Lemma \ref{SSLemma}, $f$ induces an isomorphism on the total homology. Moreover, since $f^{*}$ preserves the $\delta$-grading on Khovanov homology, it preserves $\mathfrak{gr}_{\delta}$.

The negative twist, as well as the two twists at the bottom, follow from similar arguments.

\subsection{Reidemeister I}

Let $D_{1}$ and $D_{2}$ be two diagrams which differ by a positive Reidemeister I move as in Figure \ref{NegativeR1}. 

\begin{figure}[ht]
\centering
\def\svgwidth{6.5cm}
\input{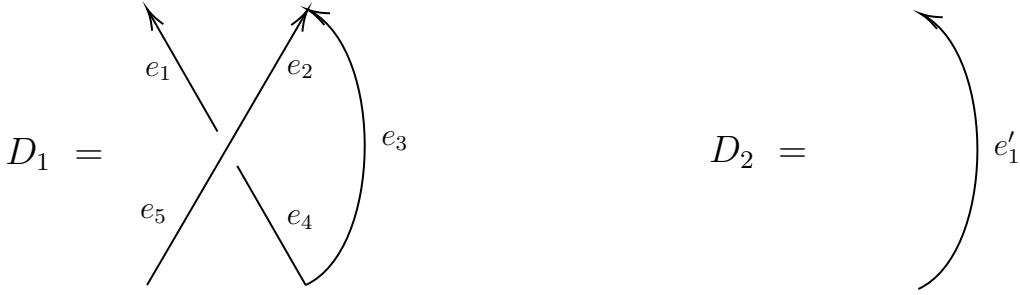}
\caption{Two diagrams differing by a positive Reidemeister I move}\label{NegativeR1}
\end{figure}

This argument is similar to the twists, as both correspond to a Reidemeister I move on Khovanov homology. Let $D_{1}^{i}$ denote the $i$-resolution of $D_{1}$ at the crossing in Figure \ref{NegativeR1}. By the MOY 0 relation, there is a splitting 
\[  C_{1 \pm 1}(D_{0}^{1}) \cong C_{1 \pm 1}(D_{2}) \otimes C_{1+1}(\mathcal{U}) \]

Since $U_{2}=U_{4}$ on $C_{1 \pm 1}(\mathcal{U})$, it can be written 
\[  C_{1 \pm 1}(\mathcal{U}) =
\xymatrix@C+2pc{\QQ[U_{2},U_{3}]/(U_{2}U_{3})&\QQ[U_{2},U_{3}]/(U_{2}U_{3})\ar[l]^{2U_{2}+2U_{3}}}. \]

Let $x_{1}$ and $x_{2}$ be the generators of the two copies of $\QQ[U_{2},U_{3}]/(U_{2}U_{3})$ so that we can write 
\[  C_{1 \pm 1}(\mathcal{U}) =
\xymatrix@C+2pc{x_1 & x_2 \ar[l]^{2U_{2}+2U_{3}}}. \]

We define 
\[f=(f_0,f_1): C_{1\pm 1}(D_1)=C_{1\pm 1}(D_1^0)\oplus C_{1\pm 1}(D_1^1) \to C_{1\pm 1}(D_2) \]
such that $f_0 (a \otimes U_{2}x_{1}) = a$, $f_0 (a \otimes x) =0 $ for $x \ne U_{2}x_{1}$, and $f_{1}=0$. 

The map $f$ is clearly a filtered chain map with respect to the cube filtration. Let $E_{k}(C_{1 \pm 1}(D))$ denote the spectral sequence induced by the cube filtration. By the MOY 0 relation, $E_{1}(C_{1 \pm 1}(D^{0}_{1}) \cong E_{1}(C_{1 \pm 1}(D_2) \otimes \mathcal{A}$, and $f^{*}$ is defined by 
\[ f^{*}(a \otimes 1)=0 \hspace{8mm} f^{*}(a \otimes U)= a \]

This is the standard quasi-isomorphism on Khovanov homology corresponding to a Reidemeister I move, so $f^{*}: E_{2}(D_{2}) \to E_{2}(D_{1})$ is an isomorphism. Applying Lemma \ref{SSLemma}, $f$ induces an isomorphism on the total homology. Moreover, since $f^{*}$ preserves the $\delta$-grading on Khovanov homology, it preserves $\mathfrak{gr}_{\delta}$.

\subsection{Reidemister II}

Let $D_{1}$ and $D_{2}$ be diagrams which differ by a Reidemeister II move as in Figure \ref{RII}. Let $D_{1}^{ij}$ denote the diagram $D_{1}$ where the top crossing has been resolved with the $i$-resolution and the bottom crossing has been resolved with the $j$-resolution.

\begin{figure}[ht]
\centering
\def\svgwidth{11cm}
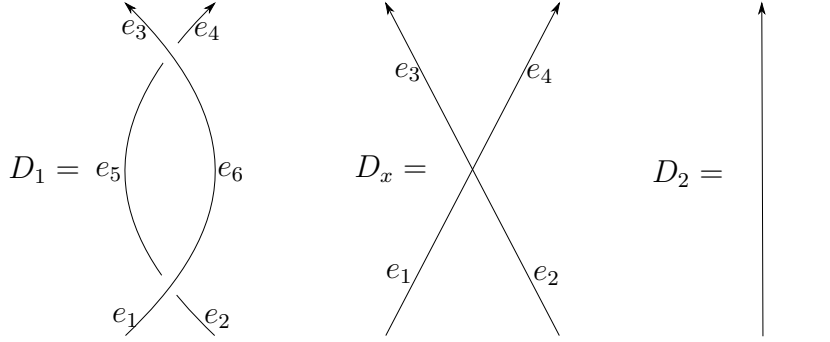
\caption{Two diagrams differing by a Reidemeister II move}\label{RII}
\end{figure}

Then, $D_{1}^{00}$ and $D_{1}^{11}$ are isotopic to the diagram $D_x$ obtained by locally replacing the two crossings with a singularization. Moreover, $D_1^{01}$ differs from $D_2$ by a MOY \rom{2} move, while $D_1^{10}$ is isotopic to $D_2$. So, $C_{1\pm 1}(D_1^{01})\cong C_{1\pm 1}(D_x)\{1\}\oplus C_{1\pm 1}(D_x)\{-1\}$, and with respect to this decomposition we can write $C_{1\pm 1}(D_1)$ as follows:

\[
\begin{diagram}
C_{1\pm 1}(D_x)&\rTo^{d_1^b}&C_{1\pm 1}(D_2)\\
\dTo_{d_1^a}&&\dTo_{d_1^d}\\
C_{1\pm 1}(D_x)\{1\}\oplus C_{1\pm 1}(D_x)\{-1\}&\rTo^{d_1^c}&C_{1\pm 1}(D_x)
\end{diagram}
\]
%
%


Then, part (2) of Lemma \ref{lem:propII} implies that $d_a^1=(-U_2\mathrm{id},\mathrm{id})$ under the above identification. Similarly, part (1) of Lemma \ref{lem:propII} implies that under the above identification the restriction of  $d_1^c$ to the first and the second summands is equal to $\mathrm{id}$ and $U_3\mathrm{id}$, respectively.

We want to define a chain map $f: C_{1\pm 1}(D_{2}) \to C_{1\pm 1}(D_{1})$ which is a homotopy equivalence. Let $f^{ij}$ denote the component of $f$ mapping to $C_{1\pm 1}(D^{ij}_{1})$. We define $f^{00}=f^{11}=0$, and under the above identification $f^{10}=\mathrm{id}$ and $f^{01}=(-d_1^d,0)$.

It's not hard to see that $f$ is a chain map. 

\begin{lemma}

The chain map $f$ is a homotopy equivalence. 

\end{lemma}

\begin{proof}
We define a chain map $g:C_{1\pm 1}(D_1)\to C_{1\pm 1}(D_2)$ such that $g\circ f=\mathrm{id}$ and $f\circ g\simeq \mathrm{id}$. Let $g^{ij}$ denote the restriction of $g$ to $C_{1\pm 1}(D_1^{ij})$. Then, we define $g^{00}=g^{11}=0$ and under the above identification $g^{10}=\mathrm{id}$ and $g^{01}=d_1^b\circ \pi_2$, where $\pi_2$ denotes the projection on the second summand. It is clear that $g\circ f=\mathrm{id}$.

Furthermore, \[\mathrm{id}-f^{01}\circ g^{01}=\begin{bmatrix}
1&U_3-U_2\\
0&1
\end{bmatrix}=d_1^a\circ \pi_2+i_1\circ d_1^c.\]

Therefore, $f\circ g\simeq \mathrm{id}$.

%
%

\end{proof}

This proves invariance of $C_{1\pm 1}$ under the Reidemeister II move shown in Figure \ref{RII}. The other Reidemeister II move is similar.

\subsection{Reidemister \rom{3}}

Suppose $D$ and $D'$ are braid diagrams, such that $D'$ is obtained from $D$ by a Reidemeister III move, and $D''$ is the braid diagram obtained from them by the move indicated in Figure \ref{fig:RIII}. For $i=0,1$, $D_i$ and $D_i'$ denote the $i$-resolutions of the middle crossing in $D$ and $D'$, respectively.

\begin{figure}[!h]
\centering
\tikzset{every picture/.style={line width=0.75pt}} 

\begin{tikzpicture}[x=0.6pt,y=0.6pt,yscale=-1,xscale=1]

\draw    (94.93,261) -- (219.72,100.58) ;
\draw [shift={(220.95,99)}, rotate = 487.88] [color={rgb, 255:red, 0; green, 0; blue, 0 }  ][line width=0.75]    (10.93,-3.29) .. controls (6.95,-1.4) and (3.31,-0.3) .. (0,0) .. controls (3.31,0.3) and (6.95,1.4) .. (10.93,3.29)   ;

\draw    (153.74,176.14) -- (96.14,100.59) ;
\draw [shift={(94.93,99)}, rotate = 412.68] [color={rgb, 255:red, 0; green, 0; blue, 0 }  ][line width=0.75]    (10.93,-3.29) .. controls (6.95,-1.4) and (3.31,-0.3) .. (0,0) .. controls (3.31,0.3) and (6.95,1.4) .. (10.93,3.29)   ;

\draw    (163.26,185.14) -- (220.95,261) ;

\draw  [color={rgb, 255:red, 255; green, 255; blue, 255 }  ,draw opacity=1 ][fill={rgb, 255:red, 255; green, 255; blue, 255 }  ,fill opacity=1 ]  (119.15, 231.49) circle [x radius= 8.05, y radius= 8.05]  ;
\draw  [color={rgb, 255:red, 255; green, 255; blue, 255 }  ,draw opacity=1 ][fill={rgb, 255:red, 255; green, 255; blue, 255 }  ,fill opacity=1 ]  (118.73, 131.98) circle [x radius= 8.05, y radius= 8.05]  ;
\draw    (136.93,261) .. controls (104.22,231.47) and (103.89,130.87) .. (136.49,100.91) ;
\draw [shift={(137.49,100.03)}, rotate = 499.4] [color={rgb, 255:red, 0; green, 0; blue, 0 }  ][line width=0.75]    (10.93,-3.29) .. controls (6.95,-1.4) and (3.31,-0.3) .. (0,0) .. controls (3.31,0.3) and (6.95,1.4) .. (10.93,3.29)   ;

\draw [color={rgb, 255:red, 0; green, 0; blue, 0 }  ,draw opacity=1 ][line width=2.25]    (259,180) -- (355,180) ;
\draw [shift={(359,180)}, rotate = 180] [color={rgb, 255:red, 0; green, 0; blue, 0 }  ,draw opacity=1 ][line width=2.25]    (17.49,-5.26) .. controls (11.12,-2.23) and (5.29,-0.48) .. (0,0) .. controls (5.29,0.48) and (11.12,2.23) .. (17.49,5.26)   ;

\draw    (310,371) -- (310,283) ;
\draw [shift={(310,281)}, rotate = 450] [color={rgb, 255:red, 0; green, 0; blue, 0 }  ][line width=0.75]    (10.93,-3.29) .. controls (6.95,-1.4) and (3.31,-0.3) .. (0,0) .. controls (3.31,0.3) and (6.95,1.4) .. (10.93,3.29)   ;

\draw    (336,372) -- (336,284) ;
\draw [shift={(336,282)}, rotate = 450] [color={rgb, 255:red, 0; green, 0; blue, 0 }  ][line width=0.75]    (10.93,-3.29) .. controls (6.95,-1.4) and (3.31,-0.3) .. (0,0) .. controls (3.31,0.3) and (6.95,1.4) .. (10.93,3.29)   ;

\draw    (421.83,262) -- (550.69,103.32) ;
\draw [shift={(551.95,101.76)}, rotate = 489.08] [color={rgb, 255:red, 0; green, 0; blue, 0 }  ][line width=0.75]    (10.93,-3.29) .. controls (6.95,-1.4) and (3.31,-0.3) .. (0,0) .. controls (3.31,0.3) and (6.95,1.4) .. (10.93,3.29)   ;

\draw    (482.55,178.07) -- (423.07,103.33) ;
\draw [shift={(421.83,101.76)}, rotate = 411.49] [color={rgb, 255:red, 0; green, 0; blue, 0 }  ][line width=0.75]    (10.93,-3.29) .. controls (6.95,-1.4) and (3.31,-0.3) .. (0,0) .. controls (3.31,0.3) and (6.95,1.4) .. (10.93,3.29)   ;

\draw    (492.38,186.97) -- (551.95,262) ;

\draw  [color={rgb, 255:red, 255; green, 255; blue, 255 }  ,draw opacity=1 ][fill={rgb, 255:red, 255; green, 255; blue, 255 }  ,fill opacity=1 ]  (527.66, 131.52) circle [x radius= 6.94, y radius= 6.94]  ;
\draw  [color={rgb, 255:red, 255; green, 255; blue, 255 }  ,draw opacity=1 ][fill={rgb, 255:red, 255; green, 255; blue, 255 }  ,fill opacity=1 ]  (528.81, 232.75) circle [x radius= 6.94, y radius= 6.94]  ;
\draw    (510.31,262) .. controls (545.52,231.53) and (544.74,133.51) .. (510.49,101.94) ;
\draw [shift={(509.44,101)}, rotate = 400.87] [color={rgb, 255:red, 0; green, 0; blue, 0 }  ][line width=0.75]    (10.93,-3.29) .. controls (6.95,-1.4) and (3.31,-0.3) .. (0,0) .. controls (3.31,0.3) and (6.95,1.4) .. (10.93,3.29)   ;

\draw    (361,372) -- (361,284) ;
\draw [shift={(361,282)}, rotate = 450] [color={rgb, 255:red, 0; green, 0; blue, 0 }  ][line width=0.75]    (10.93,-3.29) .. controls (6.95,-1.4) and (3.31,-0.3) .. (0,0) .. controls (3.31,0.3) and (6.95,1.4) .. (10.93,3.29)   ;

\draw [line width=1.5]    (211,280) -- (249.13,327.66) ;
\draw [shift={(251,330)}, rotate = 231.34] [color={rgb, 255:red, 0; green, 0; blue, 0 }  ][line width=1.5]    (14.21,-4.28) .. controls (9.04,-1.82) and (4.3,-0.39) .. (0,0) .. controls (4.3,0.39) and (9.04,1.82) .. (14.21,4.28)   ;

\draw [line width=1.5]    (420,280) -- (381.87,327.66) ;
\draw [shift={(380,330)}, rotate = 308.65999999999997] [color={rgb, 255:red, 0; green, 0; blue, 0 }  ][line width=1.5]    (14.21,-4.28) .. controls (9.04,-1.82) and (4.3,-0.39) .. (0,0) .. controls (4.3,0.39) and (9.04,1.82) .. (14.21,4.28)   ;

\draw (91.93,245.57) node   {$e_{1}$};
\draw (139.93,245.57) node   {$e_{2}$};
\draw (225.15,245.57) node   {$e_{3}$};
\draw (92.93,114.43) node   {$e_{4}$};
\draw (140.93,115.43) node   {$e_{5}$};
\draw (229.35,114.43) node   {$e_{6}$};
\draw (416.96,245.98) node   {$e_{1} '$};
\draw (505.31,246.74) node   {$e_{2} '$};
\draw (556.08,245.98) node   {$e_{3} '$};
\draw (415.96,116.26) node   {$e_{4} '$};
\draw (505.7,116.26) node   {$e_{5} '$};
\draw (559.75,116.26) node   {$e_{6} '$};
\draw (151.33,207) node   {$e_{7}$};
\draw (147.33,148) node   {$e_{8}$};
\draw (100.33,183.86) node   {$e_{9}$};
\draw (551.95,184.93) node   {$e_{7} '$};
\draw (497.03,146.78) node   {$e_{8} '$};
\draw (498.03,215.45) node   {$e_{9} '$};
\draw (309,160) node  [align=left] {RIII};
\draw (61.72,178.83) node   {$D=$};
\draw (400,180) node   {$D'=$};
\draw (283,330) node   {$D''=$};

\end{tikzpicture}
\caption{}\label{fig:RIII}
\end{figure}
The diagrams $D_0$ and $D_0'$ differ from $D''$ by a Reidemeister II move, so both $C_{1\pm1}(D_0)$ and $C_{1\pm1}(D_0')$ are homotopy equivalent to $C_{1\pm1}(D'')$. Specifically, there exists chain maps
\[\imath: C_{1\pm1}(D'')\to C_{1\pm1}(D_0)\ \ \ \text{and}\ \ \ p:C_{1\pm1}(D_0)\to C_{1\pm1}(D''), \]
such that $p\circ \imath=\mathrm{id}$ and $\imath\circ p=dh+hd$ for some $h:C_{1\pm1}(D_0)\to C_{1\pm1}(D_0)$ such that $h\imath=0$. In other word, $p$ is a strong deformation retraction. Similarly, there exists a strong deformation retraction $p':C_{1\pm1}(D_0')\to C_{1\pm1}(D'')$. Denote the corresponding inclusion by $\imath'$. So \cite[Lemma 4.5]{BarNatan05:Kh-tangle-cob} implies that the chain complex  $C_{1\pm1}(D)$, given by the mapping cone
\[C_{1\pm1}(D_0)\xrightarrow{d^+} C_{1\pm1}(D_1)\]
 is chain homotopy equivalent to the mapping cone
\[C_{1\pm1}(D'')\xrightarrow{d^+\circ\imath} C_{1\pm1}(D_1).\]
Similarly, $C_{1\pm1}(D')$ is homotopy equivalent to the mapping cone
\[C_{1\pm1}(D'')\xrightarrow{d^+\circ\imath'} C_{1\pm1}(D_1').\]

Thus, it is enough to show that the above cones are homotopy equivalent. 

For $\bullet=0,1$, let $D_{\bullet ij}$ (resp. $D_{\bullet ij}'$) denote the result of resolving the top and bottom crossings of $D_\bullet$ (resp. $D_\bullet'$)  with the $i$-resolution and the $j$-resolution.  For $ij=00$ and $11$, the diagrams $D_{1ij}$ and $D_{1ij}'$ are isotopic, while for $ij=01$ and $10$, they differ by a MOY III move. So, $C_{1\pm 1}(D_{110})\cong C_{1\pm 1}(D_{110}')\oplus C_2$, $C_{1\pm 1}(D'_{101})=C_{1\pm 1}(D_{101})\oplus C_2'$ and the isomorphisms $j_{ab}$ and $j_{ba}$ (see Equation \ref{eq:isomofcyclic}) induce isomorphisms 
\[j:C_2\to C_2'\quad\text{and}\quad j':C_2'\to C_2,\]
respectively. Abusing the notation, we will denote the map from $C_{1\pm1}(D_1)$ to $C_{1\pm1}(D_1')$ (resp. $C_{1\pm1}(D_1')$ to $C_{1\pm 1}(D_1)$) that is equal to $j$ on $C_2$ (resp. $j'$ on $C_2'$) and zero on the rest of the summands by $j$ (resp. $j'$).

Let $f^{ij}:C_{1\pm1}(D_{1ij})\to C_{1\pm1}(D'_{1ij})$ be the canonical isomorphism from $ij=00, 11$, and the MOY III maps $\imath_{\RN{3}b}$ and $\pi_{\RN{3}a}^1$ for $ij=01$ and $10$, respectively. Putting these chain maps together, we get a map $f:C_{1\pm 1}(D_1)\to C_{1\pm 1}(D_1')$. 



\begin{lemma}
$f+j$ is a chain homotopy equivalence from $C_{1\pm 1}(D_1)$ to $C_{1\pm 1}(D_1')$. Further,  the identity map on $C_{1\pm 1}(D'')$ along with $f+j$ is a chain homotopy equivalence between the mapping cones of $d^+\circ \imath$ and $d^+\circ \imath'$, and thus $C_{1\pm 1}(D)$ and $C_{1\pm 1}(D')$.

\end{lemma}
\begin{proof}
Lemma's \ref{lem:propIII-2} and \ref{lem:propIII-2} imply that $f+j$ is a chain map, and it is clear from the definition that $f+j$ is invertible.  Additionally, it follows from the definition of $\imath$ and $\imath'$ that $(\mathrm{id}, f+j)$ is a chain map, and thus a chain homotopy equivalence. 

\end{proof}

%

%
%


\subsection{Cup and Cap Swaps}

The final move that we have to prove invariance under is cup and cap swaps. The cap swap moves are depicted in Figure \ref{R4} - the cup swap moves are similar, but they occur at the bottom of the braid instead of the top.

\begin{figure}[ht]
\centering
\def\svgwidth{3.5cm}
\input{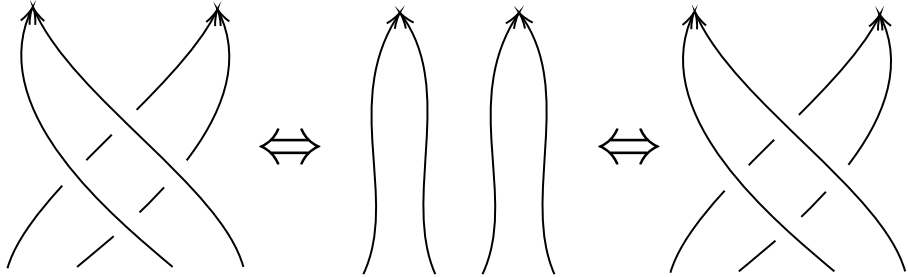}
\caption{The cap swap moves}\label{R4}
\end{figure}

Since the left and right diagrams have 4 crossings, they will have 16 vertices in the cube of resolutions, so proving invariance for these diagrams directly would be quite messy. Instead, since we already have Reidemeister II invariance, the moves pictured in Figure \ref{R4Simplified} imply invariance under the two cap swaps.

\begin{figure}[ht]
\centering
\def\svgwidth{3.5cm}
\input{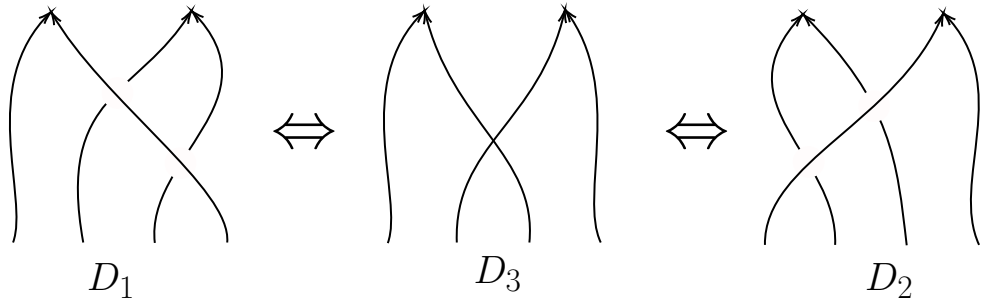}
\caption{An alternate version of the first cap swap}\label{R4Simplified}
\end{figure}

Let $D_{1}^{ij}$ denote the resolution of $D_{1}$ where the central crossing has resolution $i$ and the right-most crossing has resolution $j$. The four resolutions are shown in Figure \ref{R4Cube}. Note that $D_{1}^{10}$ is isotopic to $D_{3}$.

\begin{figure}[ht]
\centering
\def\svgwidth{3.5cm}
\input{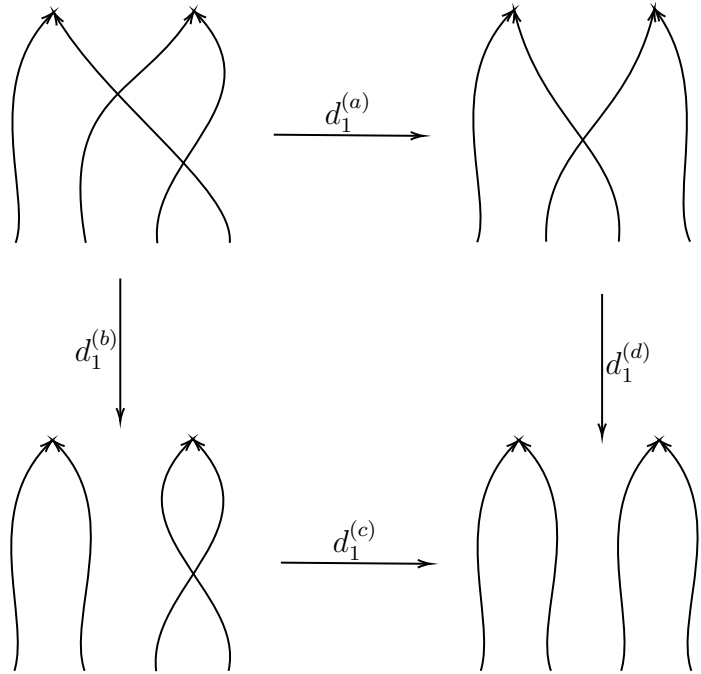}
\caption{The cube of resolutions for $D_{1}$}\label{R4Cube}
\end{figure}

We define $f$ by 
\[ f = (f_{00}, f_{01}, f_{10}, f_{11}): C_{ 1 \pm 1}(D_{3}) \to C_{1 \pm 1}(D_{1}) \]

\begin{equation*}
\begin{split}
f_{00} & = 0 \\
f_{10} & = \mathrm{id} \\
f_{01} & = -\imath_{\RN{2}} \circ d_{1}^{(d)} \circ f_{10} \\
f_{11} & = 0 \\
\end{split}
\end{equation*}

As with the twists and the Reidemeister I move, this map is a filtered chain map with respect to the cube filtrations on $C_{1 \pm 1}(D_{1})$ and $C_{1 \pm 1}(D_{3})$. To see that it is a chain map, note that $d_{1}^{(c)} \circ -\imath_{\RN{2}} = \mathrm{id}$.

Let $E_{k}(C_{1 \pm 1}(D))$ denote the spectral sequence induced by the cube filtration on $C_{1 \pm 1}(D)$. Then $f^{*}: E_{1}(D_{3}) \to E_{1}(D_{1})$ is the standard quasi-isomorphism on Khovanov homology corresponding to a Reidemeister II move. This can be seen by replacing each diagram $S$ in Figure \ref{R4Cube} with $\sm(S)$. Thus, $f^{*}$ induces an isomorphism on the $E_{2}$ pages, and therefore on the total homology as well.

The isomorphism $H_{1-1}(D_{2}) \cong H_{1-1}(D_{3})$ follows from a similar argument, as do the cup swaps.

\section{Relationship with Ozsv\'{a}th-Szab\'{o} Bordered Algebra}
Let $D_1$ and $D_2$ be braid diagrams with $n$-strands, and $D=D_1\circ D_2$. In this Section, we define a quotient $\mathcal{A}_n$ of  the algebra $A_n$ which is isomorphic to $\overline{\mathcal{B}}'(2n+1,n)$ and 
\[\mathscr{M}[D]\cong (\mathsf{M}[S_{cup}\circ D_1]\otimes\mathcal{A}_n)\otimes_{\mathcal{A}_n}(\mathsf{M}'[D_2\circ S_{cap}])\otimes\mathcal{A}_n.\]

First, we review the definition of the algebra $\mathcal{B}'(m,k)$. Given a set of $m$ points on a line identified with $\{1,2,...,m\}\subset\RR$, a \emph{local state} $\x$ is a choice of $k$ intervals $[i,i+1]$ where $1\le i\le m-1$. For each local state $\x$ there is an idempotent denoted by $I_{\x}$ in $\mathcal{B}'(m,k)$. 

We identify each interval $[i,i+1]$ by its midpoint $i+1/2$, and so each local state is a subset of $\{3/2,5/2,...,m-1/2\}$. If $\x\cap\{i-1/2,i+1/2\}=\{i-1/2\}$, then the local state $r_i(\x)=(\x\setminus\{i-1/2\})\cup\{i+1/2\}$. Similarly, if $\x\cap\{i-1/2,i+1/2\}=\{i+1/2\}$, then the local state $l_i(\x)=(\x\setminus\{i+1/2\})\cup\{i-1/2\}$. For each $2\le i\le m-1$, the algebra $\mathcal{B}'(m,k)$ contains elements $R'_i$ and $L'_i$ corresponding to shifting to the right and left, respectively. Specifically, $\mathcal{B}'(m,k)$ is the $\QQ[u_1,...,u_m]$-algebra generated by the idempotents $I_{\x}$ for all local states $\x$ and $R'_i$ and $L'_i$ for every $2\le i\le m-1$ modulo the following conditions:

\begin{itemize}
\item[B1)] For each local state $\x$, $I_{\x}R'_i=R'_iI_{r_i(\x)}=I_{\x}R'_iI_{r_i(\x)}$ and $I_{\x}L'_{i}=L'_iI_{l_i(\x)}=I_{\x}L'_iI_{r_i(\x)}$. Note that $I_{\x}R'_i$ and $I_{\x}L'_i$ vanish if $r_i(\x)$ and $l_i(\x)$ are not defined, respectively.  
\item[B2)] If $r_i(\x)$ is defined, then $I_{\x}R'_iL'_i=u_iI_{\x}$. Similarly, if $l_i(\x)$ is defined then $I_{\x}L'_iR'_i=u_iI_{\x}$.
\item[B3)] For every $i$, $R'_iR'_{i+1}=0$ and $L'_{i+1}L'_i=0$.
\item[B4)] If $\{i-1/2,i+1/2\}\cap\x=\emptyset$, then $u_iI_{\x}=0$.
\end{itemize}

Assume $m=2n+1$ and $k=n$. Each local state $\x$ determines a subset $S_{\x}\subset\{1,2,...,2n\}$ with $n$ elements defined as
\[S_{\x}=\{i\ |\ i+1/2\notin\x\}.\]
Therefore, each idempotent $I_{\x}\in\mathcal{B}'(2n+1,n)$ identifies an idempotent in $A_n$.

\begin{definition}
We define $\mathcal{A}_n$ to be the quotient of $A_n$ with the ideal generated by $R_{i}R_{i-1}$ and $L_{i-1}L_i$ for every $2\le i\le 2n$. 
\end{definition}

Let $\overline{\mathcal{B}}'(2n+1,n)$ be the quotient of $\mathcal{B}'(2n+1,n)$ with the ideal generated by $u_1$.
\begin{lemma}
There is an isomorphism 
\[h:\overline{\mathcal{B}}'(2n+1,n)\to \mathcal{A}_n\]
such that $h(I_{\x})=\iota_{S_{\x}}$ for every $\x$, and $h(R_{i}')=L_{i-1}$, $h(L_{i}')=R_{i-1}$ and $h(u_ia)=u_{i-1}a$ for every $i$. 
\end{lemma}

\begin{proof}
For every local state $\x$ if $r_i(\x)$ is defined, then $\{i-1,i\}\cap S_{\x}=\{i\}$ and 
\[h(I_{\x}R_i'L_i')=\iota_{S_{\x}}L_{i-1}R_{i-1}=u_{i-1}\iota_{S_{\x}}=h(u_iI_{\x}).\]
Similarly, if $l_i(\x)$ is defined, then $\{i-1,i\}\cap S_{\x}=\{i-1\}$ and 
\[h(I_{\x}L_i'R_i')=\iota_{S_{\x}}R_{i-1}L_{i-1}=u_{i-1}\iota_{S_{\x}}=h(u_iI_{\x}).\]
By the definition of $\mathcal{A}_n$ for every $i$, $h(R_i'R_{i+1'})=L_{i-1}L_{i}=0$ and $h(L_{i+1}'L_{i'})=R_{i}R_{i-1}=0$.

Finally, if $\{i-1/2,i+1/2\}\cap \x=\emptyset$, then $\{i-1,i\}\subset S_{\x}$ and \[h(u_iI_{\x})=u_{i-1}\iota_{S_{\x}}=\iota_{S_{\x}}R_{i-1}L_{i-1}=0\]
Thus, $h$ is an isomorphism of $\QQ[u_1,...,u_{2n}]$-modules.
\end{proof}

\begin{theorem}
For any braid $D=D_1\circ D_2$, \[\mathscr{M}[D]\cong (\mathsf{M}[S_{cup}\circ D_1]\otimes\mathcal{A}_n)\otimes_{\mathcal{A}_n}(\mathsf{M}'[D_2\circ S_{cap}])\otimes\mathcal{A}_n.\]
\end{theorem}

\noindent
This theorem allows us to work over the quotient algebra $\mathcal{A}_n$ instead of $A_{n}$ 

\begin{proof}
Consider a complete resolution $S$ of $D$ and assume $S=S_1\circ S_2$ where $S_i$ is a complete resolution of $D_i$. We show that for any cycle $Z_1$ of $S_{cup}\circ S_1$
\[x_{Z_1}R_{i}R_{i-1}=0\quad\text{for all}\ i\]
and for every cycle $Z_2$ of $S_2\circ S_{cap}$ 
\[L_{i-1}L_ix_{Z_2}=0\quad\text{for all}\ i\]
This essentially follows from the argument in part (b) of Section \ref{welldef}, but we will give the argument here as well. 

Let $Z_1$ be a cycle in $S_{cup}\circ S_1$, and suppose the $n$ outgoing edges of $S_{1}$ are labeled $e_{1},...e_{n}$. Then
\[ x_{Z}R_{i} =  x_{U_{i}(Z)} U(D(Z,e_{i}))\]

\noindent
The variables in the coefficient $U(D(Z,e_{i}))=U_{j_{1}}\cdot ... \cdot U_{j_{m}}$ can be ordered from the top to the bottom based on where they come in to $\partial_{L}D(Z,e_{i}))$. We want to show that $ x_{U_{i}(Z)} U(D(Z,e_{i})) R_{i-1}=0$.

If $U(D(Z,e_{i}))=1$ then we are done, as the bottom vertex of the disc $v_{b}(D(Z,e_{i}))$ lies in both strands of $U_{i-1}(U_{i}(Z))$. Otherwise, 
\[  x_{U_{i}(Z)} U(D(Z,e_{i})) R_{i-1} = x_{U_{i-1}(U_{i}(Z))}U(D(Z,e_{i})) U(D(U_{i}(Z),e_{i-1})) \]

\noindent
We claim that $x_{U_{i-1}(U_{i}(Z))}U(D(Z,e_{i}))=0$. To see this, note that $e_{j_{1}}$ is an edge in $U_{i-1}(U_{i}(Z))$. Thus, the $U_{j_{1}}$ in $U(D(Z,e_{i}))$ maps the cycle farther to the right. But we can do this recursively, as for each $k$, the edge $e_{j_{k}}$ lies in $U_{j_{k-1}}(...(U_{j_{1}}(U_{i-1}(U_{i}(Z))))...)$. In the end, $v_{b}(D(Z,e_{i}))$ is a vertex in both strands of the cycle $U_{j_{m}}(...(U_{j_{1}}(U_{i-1}(U_{i}(Z))))...)$, making the product zero.

The product \[L_{i-1}L_ix_{Z_2}=0\quad\text{for all}\ i\] is similar.

The reason that Ozsv\'{a}th and Szab\'{o} work with $\mathcal{B}'(2n,n)$ while we work with $\overline{\mathcal{B}}'(2n+1,n)$ is that our planar Heegaard diagram has an extra reduced unknotted component `at infinity.' This extra component component can be viewed as a single strand to the left of the diagram in the Ozsv\'{a}th-zab\'{o} picture, which we reduce by setting $u_1=0$.

\end{proof}




\bibliographystyle{hamsalpha}
\bibliography{heegaardfloer}
\end{document}